\documentclass[11pt]{article}
\usepackage{geometry}
\usepackage{paralist}
\usepackage{fullpage,marginnote,caption}
\usepackage{amsmath}
\usepackage[utf8]{inputenc} 
\usepackage[T1]{fontenc}    
\usepackage{hyperref}       
\usepackage{url}            
\usepackage{booktabs}       
\usepackage{amsfonts}       
\usepackage{nicefrac}       
\usepackage{microtype}      
\usepackage{enumitem}
\usepackage{amsfonts}
\usepackage{amssymb}
\usepackage{amsthm}
\usepackage{graphicx}
\usepackage{subcaption}
\graphicspath{{Images/}}
\usepackage{fancyhdr,color}
\usepackage{mathtools}
\usepackage{verbatim}
\usepackage{bbm}
\usepackage{lipsum}
\usepackage{multirow}
\usepackage{rotating}
\usepackage{pdflscape}
\usepackage{lscape}
\usepackage{afterpage}
\usepackage{longtable}
\usepackage[table,xcdraw]{xcolor}
\usepackage{textcomp}
\usepackage{stfloats}
\usepackage{floatrow}
\usepackage[tableposition=top]{caption}
\usepackage{float}
\floatstyle{plaintop}
\restylefloat{table}

\usepackage{arydshln}
\usepackage{mathtools}

\usepackage{fourier} 
\usepackage{array}
\usepackage{makecell}
\usepackage{tikz}
\usepackage{algorithmic}
\usepackage{algorithm}

\usepackage{placeins}

\usepackage{cite}
\usepackage{multicol}
\usepackage[normalem]{ulem} 
\usepackage{titlesec}

\usepackage[colorinlistoftodos,prependcaption,textsize=tiny]{todonotes}

\titlelabel{\thetitle. }

\newcommand\ignore[1]{}

\newcommand\bm[1]{\boldsymbol{#1}}
\newcommand\R{\mathbb{R}}
\newcommand\Z{\mathbb{Z}}
\newcommand{\sz}[1]{} 

\newcommand{\proj}{\text{proj}}

\DeclareRobustCommand
  \Compactcdots{\mathinner{\cdotp\mkern-2mu\cdotp\mkern-2mu\cdotp}}

\makeatletter
\newcommand{\pushright}[1]{\ifmeasuring@#1\else\omit\hfill$\displaystyle#1$\fi\ignorespaces}
\newcommand{\pushleft}[1]{\ifmeasuring@#1\else\omit$\displaystyle#1$\hfill\fi\ignorespaces}
\makeatother

\newcolumntype{C}[1]{>{\minwd c{#1}}c<{\endminwd}}

\newcolumntype{C}[1]{>{\centering\arraybackslash}m{#1}}

\def\fillandplacepagenumber{%
 \vbox to 0pt{\vss}\vfill
 \vbox to 0pt{\baselineskip0pt
   \hbox to\linewidth{\hss}%
   \baselineskip\footskip
   \hbox to\linewidth{%
     \hfil\thepage\hfil}\vss}}

\newcommand{\BA}{\begin{array}}
\newcommand{\EA}{\end{array}}

\newcommand{\cl}{\text{cl }}

\newcommand{\conv}{\text{conv}}

\newcommand{\mybar}[1]{\ifmmode\setbox0\hbox{$#1$}%
\else
\setbox0\hbox{#1}%
\fi
\makebox[\the\wd0][c]{%
\rule[0.42\ht0]{0.75\wd0}{0.7pt}}\hspace*{-\the\wd0}#1}

\newcommand{\bbR}{\mathbb{R}}

\makeatletter
\newcommand{\vast}{\bBigg@{4}}
\newcommand{\Vast}{\bBigg@{5}}
\makeatother

\newcommand{\ext}{\text{ext}}

\newtheorem{thm}{Theorem}

\newtheorem{cor}{Corollary}
\newtheorem{remark}{Remark}

\newtheorem{definition}{Definition}
\newtheorem{obs}{Observation}

\newtheorem{proposition}[thm]{Proposition}

\newtheorem{assumption}{Assumption}
\newtheorem{example}{Example}

\theoremstyle{remark}
\newtheorem{claim}{Claim}
\usepackage{etoolbox}
\AtEndEnvironment{proof}{\setcounter{claim}{0}}

\title{Convexification of multi-period quadratic programs with indicators}
\author{Jisun Lee, Andr\'{e}s G\'{o}mez, Alper Atamt\"{u}rk
\thanks{This research was supported, in part, by NSF AI Institute for Advances in Optimization Award 2112533, DOD ONR grant 12951270 and AFOSR grant FA9550-22-1-0369.}}

\begin{document}
\maketitle
\begin{abstract}
We study a multi-period convex quadratic optimization problem, where the state evolves dynamically as an affine function of the state, control, and indicator variables in each period. We begin by projecting out the state variables using linear dynamics, resulting in a mixed-integer quadratic optimization problem with a (block-) factorizable cost matrix. We discuss the properties of these matrices and derive a closed-form expression for their inverses. Employing this expression, we construct a closed convex hull representation of the epigraph of the quadratic cost over the feasible region in an extended space. Subsequently, we establish a tight second-order cone programming formulation with $\mathcal{O}(n^2)$ conic constraints. We further propose a polynomial-time algorithm based on a reformulation of the problem as a shortest path problem on a directed acyclic graph. To illustrate the applicability of our results across diverse domains, we present case studies in statistical learning and hybrid system control.
\end{abstract}

\section{Introduction} \label{subsec:Intro}

Given a number of periods $n\in \Z_+$, dimension $d\in \Z_+$, vector $\bm{c}\in \R^n$ and sequence of vectors $\{\bm{r}_i\}_{i=1}^{n+1}$, $\{\bm{f}_i\}_{i=1}^n$, $\{\bm{b}_i\}_{i=0}^n$ and matrices $\{\bm{P}_i\}_{i=1}^{n+1}$, $\{\bm{A}_i\}_{i=1}^n$ with $\bm{r}_i,\bm{f}_i, \bm{b}_i \in \R^d$, nonsingular $\bm{P}_i,\bm{A}_i\in \R^{d\times d}$ and $\bm{P}_i\succ \bm{0}$ , we consider a mixed-integer quadratic optimization problem (MIQP) of the form
\begin{subequations}\label{MIQP-ddim}
    \begin{align} \hspace{3cm}\min_{\bm{s},\bm{x},\bm{z}} \ & \sum_{i=1}^{n+1} \left(\boldsymbol{s}_{i} - \boldsymbol{r}_{i}\right)^\top \boldsymbol{P}_{i} \left(\boldsymbol{s}_{i} - \boldsymbol{r}_{i}\right)  + \sum_{i=1}^{n} \bm{f}_i^\top \bm{x}_{[i]} +\sum_{i=1}^{n} c_i z_i \hspace{-3cm} \label{MIQP-ddim_obj}\\ 
        \text{s.t. } \ & \bm{s}_1 = \bm{b}_0  \label{MIQP-ddim_init}\\
        & \boldsymbol{s}_{i+1} = \boldsymbol{A}_i \boldsymbol{s}_i + \boldsymbol{x}_{[i]} + \bm{b}_i, && i \in [n] \hspace{3cm} \label{MIQP-ddim_transition}\\ 
        & \boldsymbol{x}_{[i]} (1-z_i) = \boldsymbol{0},\; z_i\in \{0,1\}, && i \in [n] \label{MIQP-ddim_indicators}\\
        &\left(\{\bm{s}_i\}_{i=1}^{n+1},\bm{x},\bm{z}\right)\in C\subseteq \R^{(n+1)d}\times \R^{nd}\times \R^n, \hspace{-3cm}\label{MIQP-ddim_vars} 
    \end{align}
    \end{subequations}
where $\boldsymbol{0}$ is a $d$-dimensional vector of zeros.
Note that $\bm{x} \in \R^{dn}$ indicates a block vector consisting of $n$ of $d$-dimensional vectors, where the $i$-th block corresponds to $\bm{x}_{[i]}$ in period $i$.
Problem \eqref{MIQP-ddim} arises in hybrid control, statistical learning, and, more generally, sequential decision-making problems, where variables $\bm{s}_i$ represent the state of a system in period $i$. In each period, the system evolves according to the affine dynamics \eqref{MIQP-ddim_transition}: the state $\bm{s}_{i+1}$ depends on the state of the previous period $\bm{s}_i$ and on an input action $\bm{x}_{[i]}$ that can be undertaken by the decision-maker. If such action is done, then a variable cost $\bm{f}_i^\top \bm{x}_{[i]}$ is incurred as well as a fixed cost $c_i$ -- modeled through the use of indicator constraints \eqref{MIQP-ddim_indicators}, stating that $\bm{x}_{[i]} \neq \bm{0}\sz{d}\implies z_i=1$. The convex quadratic objective terms are used to model deviations between state $\bm{s}_i$ and a predetermined target $\bm{r}_i$, and constraints \eqref{MIQP-ddim_vars} encode additional constraints
of the problem.  

\paragraph{Contributions and outline} In this paper, we describe the convex hull of the mixed-integer epigraph of \eqref{MIQP-ddim_obj} over constraints \eqref{MIQP-ddim_init}-\eqref{MIQP-ddim_indicators}. The resulting formulation is SOCP-representable in an extended formulation with additional $\mathcal{O}(d n^2)$ variables and $\mathcal{O}(dn + n^2)$ constraints. Moreover, the convex representation naturally leads to an algorithm to solve \eqref{MIQP-ddim} --without side constraints \eqref{MIQP-ddim_vars}-- in $\mathcal{O}(n^2 \cdot \pi (d))$ time, where $\pi (d)$ denotes the computational complexity of inverting a $d \times d$ matrix and a multiplication of two $d \times d$ matrices. Our computational results show that when the side constraints encoded by $C$ are relatively simple, using the proposed formulations with off-the-shelf mixed-integer optimization solvers substantially improves the runtimes.
The proposed polynomial time algorithm can solve problems without side constraints  \eqref{MIQP-ddim_vars} in seconds that would otherwise take hours with branch-and-bound algorithms. 

The rest of this paper is organized as follows. In the remainder of this section, we discuss related literature and introduce the notations used throughout the paper. In \S\ref{sec:preliminary}, we review relevant concepts and the reformulation technique employed in the study. 
In \S\ref{sec:convexification}, we first reformulate \eqref{MIQP-ddim} into an MIQP with a special form of cost matrix $\bm{Q}$ by projecting out the state variables. We then present a convex hull representation of the epigraph set of the quadratic cost function by establishing a new expression for the inverses of (block-) factorizable matrices. 
In \S\ref{sec:algorithm}, two solution methods that effectively solve the MIQP are discussed.
Two case studies, detailed in \S\ref{sec:calcium} and \ref{sec:path-following}, demonstrate the efficient application of our results in the spike detection and path-following problems. 
In \S\ref{sec:conclusion}, we conclude the paper with final remarks.

\paragraph{Related work} 
Multi-period mixed-integer convex optimization problems with linear or convex quadratic cost functions have been widely studied in diverse domains. An example is the lot-sizing problem
\begin{equation*}
    \begin{aligned}
        \min_{\bm{s}, \bm{x}, \bm{z}} \ & \sum_{i=1}^{n+1} p_i s_i + \sum_{i=1}^{n} f_i x_i + \sum_{i=1}^n c_i z_i \\
        \text{s.t. }\ & s_1 = b_0 \\
        & s_{i+1} = s_i + x_i - b_i, &&  i \in [n] \\
        & x_i (1-z_i) = 0, \ z_i \in \{0,1\}, &&  i \in [n]\\
        & (\bm{s}, \bm{x}, \bm{z} ) \in C \subseteq \R^{n+1} \times \R^n \times \R^n,
    \end{aligned}
\end{equation*}
where $s_i$ is the stock at the end of period $i$, $x_i$ is the amount produced in period $i$, $z_i$ is a binary variable, indicating whether the item is produced in period $i$, and $b_i$ is the demand in period $i$. The cost comprises the unit holding cost $h_i$, the unit production cost $f_i$, and the setup cost $c_i$ in period $i$. For the uncapacitated version of the lot-sizing problem
can be solved in polynomial-time by dynamic programming (DP) \cite{wagner1958dynamic,wagelmans1992economic,federgruen1991simple,aggarwal1993improved}.
However, once the capacity constraints are included in $C$, the lot-sizing problem becomes NP-hard \cite{bitran1982computational}. 
Nevertheless, cutting plane approaches have been effective for solving the capacitated lot-sizing problem \cite{AM:lspoly,barany1984strong,leung1989facets,pochet1988valid}.

Other relevant problems with convex quadratic costs frequently arise in the context of statistical learning with time-series data. Notable applications include image segmentation \cite{hochbaum2001efficient}, change/signal detection \cite{atamturk2021sparse,bach2019submodular,jewell2018exact,yan2022real},
as well as the estimation of hidden states in Gaussian hidden Markov models \cite{gomez2021outlier,bhathena2024parametric}. These applications involve substructures that can be represented in the form of \eqref{MIQP-ddim}.
Similar problems are also actively studied in the hybrid system control domain. The discrete-time switched linear quadratic regulation (DSLQR) problem encompasses analogous subproblems \cite{borrelli2005dynamic,bemporad1999control,zhang2009study}. 
Specifically, the path-following problem is often formulated with a convex quadratic cost function and linear system dynamics. This problem is studied in diverse domains such as robotics \cite{mellinger2012mixed,kuindersma2016optimization}, autonomous vehicles \cite{qian2016optimal,quirynen2024real}, satellite control \cite{hu2019dynamic}, and hybrid electric vehicle control \cite{borhan2011mpc}.

Parametric programming has been widely utilized to tackle multi-period MIQP problems in hybrid control systems, leveraging the piecewise affinity of optimal control law \cite{bemporad2002explicit,dua2002multiparametric}, which also applies to general multi-period problems of the form \eqref{MIQP-ddim}. Parametric MIQPs are commonly addressed using DP algorithms \cite{borrelli2005dynamic,bhathena2024parametric} or branch-and-bound method \cite{axehill2010improved}. Due to the curse of dimensionality, numerous approximation methods have been developed \cite{summers2011multiresolution,axehill2014parametric}. 
Research on strong formulations of multi-period MIQPs includes comparisons of strengths of different formulations \cite{axehill2010convex, marcucci2019mixed} and development of extended formulations through disjunctive programming \cite{kurtz2021more, marcucci2024shortest}.
Although studies on strengthening the multi-period MIQP with dynamical system constraints are limited, strong formulations of relevant MIQPs with indicator variables have been actively studied \cite{agsuper,han2x2,wei2024convex,atamturk2018strong,liu2024polyhedral, frangioni2006perspective,akturk2009strong,gunluk2012perspective,bonami2015mathematical,anstreicher2021quadratic, wu2017quadratic, dong2013valid, gunluk2010perspective, bertsimas2023new, zheng2014improving}. These studies apply to general MIQPs and do not exploit the special structure imposed by
the linear dynamics constraint \eqref{MIQP-ddim_transition} that links adjacent state variables. In some cases, however, the linear transition can be embedded in special cost matrices. 
Related MIQP formulations have been studied, where the linear dependencies between periods are modeled as penalties in the objective function rather than as hard constraints.
Specifically, MIQPs with tridiagonal matrices\cite{liu2023graph}, banded matrices \cite{gomez2024real}, or adjacency matrices of tree graphs \cite{bhathena2024parametric} are particularly relevant to this study.

\paragraph{Notation} We denote scalars as plain lowercase letters $c$, vectors as boldface lowercase letters $\boldsymbol{v}$, and matrices as boldface uppercase letters $\boldsymbol{W}$. 
We use notations $[n] = \{1,\ldots, n\}$ and $[m,n] = \{m,  \ldots,n\}$ for $m,n \in \mathbb{Z}_+$ such that $m < n$. 
Additionally, we adopt the convention of $0 f(x/0) = \lim_{z \rightarrow 0+} z f(x/z)$ for any convex function $f: \bbR \rightarrow \bbR$, thus $z f(x/z)$ for $z \geq 0$ is the closure of the perspective function of $f$. The convex hull of a set $X$ is written as $\conv (X)$ and its closure is denoted by $\cl \conv (X)$.
We further specify some vector and matrix notations. Let $\boldsymbol{0}\sz{d} \in \bbR^d$ and $\boldsymbol{1}\sz{d} \in \bbR^d$ be vectors of zeros and ones of size $d$, respectively. Similarly, let $\boldsymbol{O}\sz{d} \in \bbR^{d \times d}$ be a square matrix of zeros and $\boldsymbol{I}\sz{d} \in \bbR^{d \times d}$ be an identity matrix.
Moreover, let $\boldsymbol{e}_i \in \bbR^n$ be an $n$-dimensional vector such that the $i$-th element is $1$, and other entries are $0$. Likewise, $\boldsymbol{e}_S \in \bbR^n$ for some $S \subseteq [n]$ indicates an $n$-dimensional vector of which the $i$-th element is $1$ if $i \in S$, and $0$ otherwise. In the same manner, let $\boldsymbol{E}_i \in \bbR^{dn \times d}$ denote a block matrix consisting of n blocks of $d \times d$ matrices, where the $i$-th block is $\boldsymbol{I}\sz{d}$ and other blocks are $\boldsymbol{O}\sz{d}$.
For an index set $S \subseteq [n]$ such that $|S| = k$, the vector $\boldsymbol{v}_S \in \bbR^k$ and $\boldsymbol{W}_S \in \bbR^{k\times k}$ indicate the subvector of a vector $\boldsymbol{v} \in \bbR^n$ and the principal submatrix of a matrix $\boldsymbol{W} \in \bbR^{n\times n}$ induced by $S$, respectively. Additionally, $\hat{\boldsymbol{v}}_S \in \bbR^n$ is the $n$-dimensional vector that $\hat{v}_i = v_i$ if $i \in S$, and $\hat{v}_i = 0$, otherwise. Similarly, $\hat{\boldsymbol{W}}_S \in \bbR^{n \times n}$ represents $n \times n$ matrix that fills zeros into $\boldsymbol{W}_S$ for $(i,j)$-th entries for $i \notin S$ or $j \notin S$. Moreover, the notation $\hat{\boldsymbol{W}}_{S}^{-1} \in \bbR^{n \times n}$ denotes the matrix obtained by first selecting the principal submatrix $\boldsymbol{W}_{S}$ from $\boldsymbol{W}$, computing its inverse, and then applying the hat operation.
For example, let $\bm{W} = \begin{bmatrix}
    5 & 4 & 2 \\ 4 & 8 & 4 \\ 2 & 4 & 8
\end{bmatrix}$. Then, for $S=\{1,3\}$, $\hat{\bm{W}}_{S}^{-1}$ is computed as follows: 
\begin{align*}
    \bm{W}_{S} = \begin{pmatrix}
        5 & 2 \\ 2 & 8
    \end{pmatrix} \ \ \ \longrightarrow \ \ \ \bm{W}_{S}^{-1} = \begin{pmatrix}
        2/9 & -1/18 \\ -1/18 & 5/36
    \end{pmatrix} \ \ \ \longrightarrow \ \ \ \hat{\bm{W}}_{S}^{-1} = \begin{pmatrix}
        2/9 & 0 & -1/18 \\ 0 &0 & 0 \\ -1/18 & 0 & 5/36
    \end{pmatrix}.
\end{align*}
For $S=\emptyset$, both $\hat{\boldsymbol{W}}_{\emptyset}$ and $\hat{\boldsymbol{W}}_{\emptyset}^{-1}$ correspond to the $n \times n$ matrix of zeros $\boldsymbol{O}$.
For a block vector $\boldsymbol{v}\sz{dn} \in \bbR^{dn}$, the block subvector of $\boldsymbol{v}\sz{dn}$ induced by a set $S \subseteq [n]$ is denoted as $\boldsymbol{v}\sz{dk}_{[S]} \in \bbR^{dk}$ for $k=|S|$, and $\hat{\boldsymbol{v}}\sz{dn}_{[S]} \in \bbR^{dn}$ represents the $dn$-dimensional vector that fills $\boldsymbol{0}\sz{d} \in \bbR^d$ into $\boldsymbol{v}\sz{dk}_{[S]}$ for $i \notin S$. For instance, let $\bm{v} = [\overbrace{\begin{array}{cc}
        1 & 2
    \end{array}}^{\bm{v}_{[1]}^\top} \ \ \overbrace{\begin{array}{cc}
        3 & 4
    \end{array}}^{\bm{v}_{[2]}^\top} \ \ \overbrace{\begin{array}{cc}
        5 & 6
    \end{array}}^{\bm{v}_{[3]}^\top}]^\top$.
Then, for $S = \{1,3\}$, $\hat{\bm{v}}_{[S]}$ is given by
\begin{align*}
    \bm{v}_{[S]} = \begin{bmatrix}
        1 & 2 & 5 & 6
    \end{bmatrix}^\top \ \ \ \longrightarrow \ \ \ \hat{\bm{v}}_{[S]} = \begin{bmatrix}
        1 & 2 & 0 & 0 & 5 & 6
    \end{bmatrix}^\top.
\end{align*}
Notations for a block matrix $\boldsymbol{W}\sz{dn} \in \bbR^{dn \times dn}$ composed of $n^2$ blocks of $d\times d$-dimensional matrices are defined in the same manner. We denote the $(i,j)$-th block of $\boldsymbol{W}\sz{dn}$ as $\boldsymbol{W}\sz{d}_{[ij]} \in \bbR^{d \times d}$ for $i,j \in [n]$, the block principal submatrix induced by a set $S \subseteq [n]$ with $k=|S|$ as $\boldsymbol{W}\sz{dk}_{[S]} \in \bbR^{dk\times dk}$, and the $dn \times dn$ matrix that pads $\boldsymbol{O}\sz{d}$ to $\boldsymbol{W}\sz{dk}_{[S]}$ for  the $(i,j)$-th block when either $i \notin S$ or $j \notin S$ as $\hat{\boldsymbol{W}}\sz{dn}_{[S]} \in \bbR^{dn \times dn}$. The notation $\hat{\boldsymbol{W}}_{[S]}^{-1} \in \bbR^{n \times n}$ is essentially the same as $\hat{\bm{W}}_{S}^{-1}$ of a factorizable matrix, where the block principal submatrix $\boldsymbol{W}_{[S]}$ is taken first, compute its inverse, and then apply the hat operation.
We use $\boldsymbol{U}\circ \boldsymbol{V}$ to denote the Hadamard product of the two matrices $\boldsymbol{U} \in \bbR^{n\times n}$ and $\boldsymbol{V}  \in \bbR^{n\times n}$, where the $(i,j)$-th entry is $\left(\boldsymbol{U}\circ \boldsymbol{V} \right)_{ij} = U_{ij} V_{ij}$ for $i,j \in [n]$. For two block symmetric matrices $\boldsymbol{U}\sz{dn} \in \bbR^{dn \times dn}$ and $\boldsymbol{V}\sz{dn} \in \bbR^{dn \times dn}$ constructed with $n^2$ blocks of the size $d \times d$, we use the notation $\boldsymbol{U}\sz{dn} \bullet \boldsymbol{V}\sz{dn}$ to indicate a $dn \times dn$-dimensional block symmetric matrix of which the $(i,j)$-th block is $\left(\boldsymbol{U}\sz{dn} \bullet \boldsymbol{V}\sz{dn} \right)_{[ij]} = \boldsymbol{U}\sz{d}_{[ij]} \boldsymbol{V}\sz{d}_{[ij]}^\top$ for $i \leq j$. 

\section{Preliminaries} \label{sec:preliminary}

In this section, we cover the preliminary results necessary for the development of the paper.
First, observe that the state variables $\bm{s}_i$, $i \in [n]$, in problem \eqref{MIQP-ddim},  can be projected out using the initial state condition \eqref{MIQP-ddim_init} and linear dynamics constraints \eqref{MIQP-ddim_transition}. The resulting optimization problem is of the form 
\begin{equation}
    \begin{aligned}
        \min \ & \boldsymbol{x}\sz{dn}^\top \boldsymbol{Q}\sz{dn} \boldsymbol{x}\sz{dn} + \boldsymbol{a}^\top \boldsymbol{x}\sz{dn} + \boldsymbol{c}^\top \boldsymbol{z} \\ 
        \text{s.t. } \ & \boldsymbol{x}_{[i]}^{} (1 - z_i) = \boldsymbol{0},\; z_i\in \{0,1\}, \qquad  i \in [n] \\ 
        & (\boldsymbol{x}\sz{dn},\bm{z})\in \tilde{C} \subseteq \mathbb{R}^{dn}\times\R^n, 
    \end{aligned} \label{MIQP-block}%
\end{equation}
where $\Tilde{C} = \proj_{(\bm{x}, \bm{z})} (C)$, 
$\bm{a} \in \R^{dn}$ and $\boldsymbol{Q}\sz{dn} \in \mathbb{R}^{dn \times dn}$ are appropriately defined. In Appendix \ref{Appendix:motivating_ex_singleton} ($d=1$) and Appendix \ref{Appendix:motivating_ex_block} ($d >1$), we derive the corresponding $\bm{a}$ and $\bm{Q}\sz{dn}$ and show that $\bm{Q}\sz{dn}$  is positive definite (\emph{block-}) \emph{factorizable matrix}, see Definition~\ref{def:blockSep}. 

\begin{definition}[Block-factorizable matrix]\label{def:blockSep}
A symmetric matrix $\bm{Q}\sz{dn}\in \R^{dn\times dn}$ is block-factorizable if there exists matrices $\left\{\boldsymbol{U}_i \in \mathbb{R}^{d\times d} \right\}_{i \in [n]}$ and $\left\{\boldsymbol{V}_i \in \mathbb{R}^{d\times d} \right\}_{i\in [n]}$ such that the $(i,j)$-th block of $\bm{Q}\sz{dn}$ is given by $\bm{Q_{[ij]}\sz{dn}}=\bm{U_iV_j^\top}$ for $i\leq j$, that is, 
\begin{align*}
\def\arraystretch{1.1}
    \boldsymbol{Q}\sz{dn} = \overbrace{\begin{bmatrix}
        \boldsymbol{U}_1 & \boldsymbol{U}_1 & \boldsymbol{U}_1 & \cdots & \boldsymbol{U}_1 \\
        \boldsymbol{U}_1 & \boldsymbol{U}_2 & \boldsymbol{U}_2 & \cdots & \boldsymbol{U}_2 \\
        \boldsymbol{U}_1 & \boldsymbol{U}_2 & \boldsymbol{U}_3 & \cdots & \boldsymbol{U}_3 \\
        \vdots & \vdots & \vdots & \ddots & \vdots \\
        \boldsymbol{U}_1 & \boldsymbol{U}_2 & \boldsymbol{U}_3 & \cdots & \boldsymbol{U}_n
    \end{bmatrix}}^{\boldsymbol{U}\sz{dn}} \bullet \overbrace{\begin{bmatrix}
        \boldsymbol{V}_1 & \boldsymbol{V}_2 & \boldsymbol{V}_3 & \cdots & \boldsymbol{V}_n \\
        \boldsymbol{V}_2 & \boldsymbol{V}_2 & \boldsymbol{V}_3 & \cdots & \boldsymbol{V}_n \\
        \boldsymbol{V}_3 & \boldsymbol{V}_3 & \boldsymbol{V}_3 & \cdots & \boldsymbol{V}_n \\
        \vdots & \vdots & \vdots & \ddots & \vdots \\
        \boldsymbol{V}_n & \boldsymbol{V}_n & \boldsymbol{V}_n & \cdots & \boldsymbol{V}_n
    \end{bmatrix}}^{\boldsymbol{V}\sz{dn}} := \begin{bmatrix}
        \boldsymbol{U}_1 \boldsymbol{V}_1^\top & \boldsymbol{U}_1 \boldsymbol{V}_2^\top & \boldsymbol{U}_1 \boldsymbol{V}_3^\top & \cdots & \boldsymbol{U}_1 \boldsymbol{V}_n^\top \\
        \boldsymbol{V}_2 \boldsymbol{U}_1^\top & \boldsymbol{U}_2 \boldsymbol{V}_2^\top & \boldsymbol{U}_2 \boldsymbol{V}_3^\top & \cdots & \boldsymbol{U}_2 \boldsymbol{V}_n^\top \\
        \boldsymbol{V}_3 \boldsymbol{U}_1^\top & \boldsymbol{V}_3 \boldsymbol{U}_2^\top & \boldsymbol{U}_3 \boldsymbol{V}_3^\top & \cdots & \boldsymbol{U}_3 \boldsymbol{V}_n^\top \\
        \vdots & \vdots & \vdots & \ddots & \vdots \\
        \boldsymbol{V}_n \boldsymbol{U}_1^\top & \boldsymbol{V}_n \boldsymbol{U}_2^\top & \boldsymbol{V}_n \boldsymbol{U}_3^\top & \cdots & \boldsymbol{U}_n \boldsymbol{V}_n^\top
    \end{bmatrix}. \hfill \blacksquare
\end{align*}
\end{definition}

A noteworthy special class of block-factorizable matrices arises when $d=1$. In that case, given $\{u_i\}_{i \in [n]}$ and $\{v_i\}_{i \in [n]}$, the $(i,j)$-th entry of the symmetric matrix $\bm{Q}$ is given as $Q_{ij}=u_iv_j$ for $i\leq j$, and $\bm{Q}$ can  be expressed as a Hadamard product of two matrices as
\begin{align}
    \boldsymbol{Q} = \overbrace{\begin{bmatrix}
        u_1 & u_1 & u_1 & \cdots & u_1 \\
        u_1 & u_2 & u_2 & \cdots & u_2 \\
        u_1 & u_2 & u_3 & \cdots & u_3 \\
        \vdots & \vdots & \vdots & \ddots & \vdots \\
        u_1 & u_2 & u_3 & \cdots & u_n
    \end{bmatrix}}^{\boldsymbol{U}} \circ \overbrace{\begin{bmatrix}
        v_1 & v_2 & v_3 & \cdots & v_n \\
        v_2 & v_2 & v_3 & \cdots & v_n \\
        v_3 & v_3 & v_3 & \cdots & v_n \\
        \vdots & \vdots & \vdots & \ddots & \vdots \\
        v_n & v_n & v_n & \cdots & v_n
    \end{bmatrix}}^{\boldsymbol{V}} = \begin{bmatrix}
        u_1 v_1 & u_1 v_2 & u_1 v_3 & \cdots & u_1 v_n \\
        u_1 v_2 & u_2 v_2 & u_2 v_3 & \cdots & u_2 v_n \\
        u_1 v_3 & u_2 v_3 & u_3 v_3 & \cdots & u_3 v_n \\
        \vdots & \vdots & \vdots & \ddots & \vdots \\
        u_1 v_n & u_2 v_n & u_3 v_n & \cdots & u_n v_n
    \end{bmatrix}. \label{two_seq_mat}
\end{align}
Matrices of the form \eqref{two_seq_mat} are referred to as \textit{matrices "factorisables"} (or \textit{factorizable matrices}) \cite{baranger1971breve} (which inspired the "block-factorizable" name we use for the multi-dimensional generalization), \textit{separable matrices} \cite{rozsa1987inverse}, or \textit{single-pair matrices} \cite{bossu2024tridiagonal}.
The distinguishing feature of factorizable matrices is that they correspond precisely to inverses of symmetric tridiagonal matrices\cite{baranger1971breve}. Conversely, nonsingular block-factorizable matrices have block-tridiagonal inverses \cite{meurant1992review}. 


\noindent To construct relaxations for \eqref{MIQP-ddim} and, equivalently, \eqref{MIQP-block}, we study the convex hull of sets
\begin{align*}
    X_{\boldsymbol{Q}} &:= \left\{ (\boldsymbol{x},\boldsymbol{z},\tau) \in \bbR^{n} \times \{0,1\}^n \times \mathbb{R}_+ : \ \tau \geq \boldsymbol{x}^\top \boldsymbol{Q} \boldsymbol{x}, \ x_i (1-z_i) = 0, \  i\in [n] \right\}\text{, and }\\
    X^B_{\boldsymbol{Q}} &:= \left\{ (\boldsymbol{x},\boldsymbol{z},\tau) \in \bbR^{dn} \times \{0,1\}^n \times \mathbb{R}_+ : \ \tau \geq \boldsymbol{x}\sz{dn}^\top \boldsymbol{Q} \boldsymbol{x}\sz{dn}, \ \boldsymbol{x}_{[i]}\sz{dn} (1- z_i) = \boldsymbol{0}\sz{d}, \  i \in [n] \right\}
\end{align*}
where the matrix $\bm{Q} \in \R^{n \times n}$ is factorizable in $X_{\bm{Q}}$, and $\bm{Q} \in \R^{dn \times dn}$ is block-factorizable in $X_{\bm{Q}}^B$.

Our approach to constructing the convex hulls relies on a recent result from \cite{wei2024convex}.
Proposition~\ref{prop:hull1d} states that the convex hull of $X_{\bm{Q}}$ can be expressed by a single positive definite constraint (explicitly given) and a polyhedral set associated with describing inverses of principal submatrices of matrix $\bm{Q}$.
\begin{proposition}[Wei et al \cite{wei2024convex}]\label{prop:hull1d} 
For a positive definite $\bm{T}$ and  $Z \subseteq \{0,1\}^n$, the closed convex hull of 
\begin{align*}
    \bar X_{\bm{T}} := \left\{ (\boldsymbol{x},\boldsymbol{z},\tau) \in \bbR^{n} \times Z \times \mathbb{R}_+ : \ \tau \geq \boldsymbol{x}^\top \boldsymbol{T} \boldsymbol{x}, \ x_i (1-z_i) = 0, \  i\in [n] \right\},
\end{align*}
is given by 
\begin{align*}
\cl \conv \left( \Bar{X}_{\bm{T}}\right)=\left\{(\bm{x},\bm{z},\tau)\in \R^{2n+1}:\exists \bm{W}\in \R^{n\times n}\text{ such that }\begin{bmatrix}
            \boldsymbol{W} & \boldsymbol{x}\\
            \boldsymbol{x}^\top & \tau
        \end{bmatrix} \succeq \bm{0},\; (\bm{z}, \bm{W})\in P_{\bm{T}}\right\},
\end{align*}
where $ P_{\boldsymbol{T}} := \conv \left( \left\{ \left(\boldsymbol{e}_{S}, \hat{\boldsymbol{T}}_{S}^{-1}\right) : \bm{e}_S \in Z \right\} \right)$.
\end{proposition}

\noindent For $X_{\bm{Q}}$, we have $Z = \{0,1\}^n$; therefore, $P_{\bm{Q}} = \conv \left( \left\{ \left(\boldsymbol{e}_{S}, \hat{\boldsymbol{Q}}_{S}^{-1}\right)\right\}_{S \subseteq [n]} \right)$.

\begin{example} \label{eg:1d-ext}  
For the matrix 
$\boldsymbol{Q} = \begin{bmatrix}
            5 & 4 & 2 \\
            4 & 8 & 4 \\
            2 & 4 & 8 
        \end{bmatrix}$, we have 
\begin{align*}
    \def\arraystretch{1.1}
    \bm{\hat Q_\emptyset^{-1}}= \left(\begin{array}{ccc}
         0&0&0\\
    0&0&0\\
    0&0&0
    \end{array} \right),\; \bm{\hat Q_{\{1\}}^{-1}}= \left(\begin{array}{ccc}
         \frac{1}{5}&0&0\\
    0&0&0\\
    0&0&0
    \end{array}\right),\;\bm{\hat Q_{\{1,3\}}^{-1}}= \left(\begin{array}{rrr}
        \frac{2}{9} &0&-\frac{1}{18}\\
        0&0&0\\
        -\frac{1}{18}&0&\frac{5}{36}
    \end{array} \right),\;\bm{\hat Q_{\{1,2,3\}}^{-1}}=\left(\begin{array}{rrr}
    \frac{1}{3} & -\frac{1}{6} & 0 \\
    -\frac{1}{6} & \frac{1}{4} & -\frac{1}{12}\\
    0 & -\frac{1}{12} & \frac{1}{6}
    \end{array} \right). 
\end{align*}
Moreover, the polyhedron $P_{\bm{Q}}$ associated with matrix $\boldsymbol{Q}$ is the convex hull of the following eight points $(\bm{z}, \bm{W})$; we omit lower triangular entries of $\bm{W}$ for brevity.
        \begin{center}
        \begin{tabular}{c c c c c c c c c}
        \hline       $z_1$&$z_2$&$z_3$&$W_{11}$&$W_{12}$&$W_{13}$&$W_{22}$&$W_{23}$&$W_{33}$\\
        \hline 
        0&0&0&0&0&0&0&0&0\\
        1&0&0&1/5&0&0&0&0&0\\
        0&1&0&0&0&0&1/8&0&0\\
        0&0&1&0&0&0&0&0&1/8\\
        1&1&0&1/3&-1/6&0&5/24&0&0\\
        1&0&1&2/9&0&-1/18&0&0&5/36\\
        0&1&1&0&0&0&1/6&-1/12&1/6\\
        1&1&1&1/3&-1/6&0&1/4&-1/12&1/6\\ \hline
        \end{tabular}
        \end{center}
Note that the matrix $\bm{Q}$ is a factorizable matrix defined by $\{u_i\}_{i \in [3]} = \{1,2,4\}$ and $\{v_i\}_{i\in [3]} = \{5,4,2\}$. 
\hfill $\blacksquare$
\end{example}

Proposition~\ref{prop:hulldd} extends Proposition~\ref{prop:hull1d} to the positive definite block matrix $\bm{T}$ case, allowing it to be applied to the set $X^B_{\bm{Q}\sz{dn}}$ for a block-factorizable positive definite matrix $\bm{Q}\sz{dn} \in \R^{dn \times dn}$.
\begin{proposition}\label{prop:hulldd}
For a positive definite block matrix $\bm{T}$, the closed convex hull of 
\begin{align*}
    \Bar{X}_{\bm{T}}^B := \left\{ (\boldsymbol{x},\boldsymbol{z},\tau) \in \bbR^{dn} \times \{0,1\}^n \times \mathbb{R}_+ : \ \tau \geq \boldsymbol{x}^\top \boldsymbol{T} \boldsymbol{x}, \ \bm{x}_{[i]} (1-z_i) = \bm{0}, \  i\in [n] \right\},
\end{align*}
is given by
\begin{align*} 
\cl \conv \left( \Bar{X}_{\bm{T}}^B\right)=\left\{(\bm{x},\bm{z},\tau)\in \R^{dn} \times \R^{n} \times \R :\exists \bm{W}\in \R^{dn\times dn}\text{ such that } \ \begin{bmatrix}
            \boldsymbol{W} & \boldsymbol{x}\\
            \boldsymbol{x}^\top & \tau
        \end{bmatrix} \succeq \bm{0},\; (\bm{z}, \bm{W})\in P_{\bm{T}}^B\right\},
\end{align*} 
for $ P^B_{\boldsymbol{T}\sz{dn}} := \conv \left( \left\{ \left(\boldsymbol{e}_{S}, \hat{\boldsymbol{T}}\sz{dn}_{[S]}^{-1}\right) \right\}_{S \subseteq [n]} \right).$
\end{proposition}
\begin{proof}
    For $i \in [n]$, make $d$ copies of $z_i$, $\Tilde{z}_{ij} = z_i$ for $j \in [d]$, where each $\Tilde{z}_{ij}$ performs as the indicator of the $j$-th entry of $\bm{x}_{[i]}$, which we denote $x_{ij}$. Then, for $i \in [n]$, it holds
    \begin{align*}
        \bm{x}_{[i]} (1- z_i) = \bm{0} \ \ \Longleftrightarrow \ \ x_{ij} \left(1-\Tilde{z}_{ij}\right) = 0, \
         \Tilde{z}_{ij} = z_i, \ \ \forall j \in [d].
    \end{align*}
    Therefore, we have an extended representation of $X^B$ given by
    \begin{align*}
        X^B = \left\{(\bm{x}, \Tilde{\bm{z}},\tau)\in \R^{dn} \times Z \times \R : \ \tau \geq \bm{x}^\top \bm{T} \bm{w}, \ x_{ij}\left( 1- \Tilde{z}_{ij} \right) = 0, \ i \in [n], \ j \in [d] \right\},
    \end{align*}
    where 
    \begin{align*}
        Z = \left\{ \Tilde{z} \in \{0,1\}^{dn} : \ \Tilde{z}_{ij} = \Tilde{z}_{i{j'}}, \ i \in [n], \ j, j' \in [d]\right\}.
    \end{align*}
    Applying Proposition~\ref{prop:hull1d}, $X^B_{\bm{T}\sz{dn}}$ becomes 
    \begin{align}
        \cl \conv \left( X^B\right)=\left\{(\bm{x}, \Tilde{\bm{z}},\tau)\in \R^{dn} \times \R^{dn} \times \R :\exists \bm{W}\in \R^{dn\times dn} \ \text{ such that } \ \begin{bmatrix}
                \boldsymbol{W} & \boldsymbol{x}\\
                \boldsymbol{x}^\top & \tau
            \end{bmatrix} \succeq \bm{0},\; (\Tilde{\bm{z}}, \bm{W})\in \Tilde{P}_{\bm{T}\sz{dn}}\right\}, \label{convX_block_ext}
    \end{align}
    for $\Tilde{P}_{\bm{T}\sz{dn}}:=  \conv \left( \left\{ \left(\boldsymbol{e}_{\Tilde{S}}, \hat{\boldsymbol{T}}_{\Tilde{S}}^{-1}\right) : \bm{e}_{\Tilde{S}} \in Z \right\} \right)$.
    Employing $z_i = \Tilde{z}_{ij}$, $\forall j \in [d]$, it holds
    \begin{align*}
        (\Tilde{\bm{z}}, \bm{W})\in \Tilde{P}_{\bm{T}} \ \ \Longleftrightarrow \ \ (\bm{z}, \bm{W})\in P^B_{\bm{T}}.
    \end{align*}
    As a result, the representation of $\cl \conv \left( X^B\right)$ in $\R^n \times \R^{dn \times dn}$ in the statement of the proposition is obtained.
\end{proof}


\section{Convexification} \label{sec:convexification}
The main goal of this section is to construct 
the closed convex hulls of $X_{\boldsymbol{Q}}$ and $X^B_{\boldsymbol{Q}}$. 
Based on Propositions~\ref{prop:hull1d} and \ref{prop:hulldd}, 
it suffices to derive a closed-form representation of polyhedra $P_{\boldsymbol{Q}}$ and $P_{\boldsymbol{Q}}^B$, describing the inverses of submatrices of factorizable and block-factorizable matrices, respectively.
Note that since $X_{\boldsymbol{Q}}$ is the special case of $X^B_{\boldsymbol{Q}}$ with $d=1$, in principle, it suffices to focus on the block-factorizable case. However, notations associated with the block-factorizable case are more complicated and less intuitive. Therefore, for clarity, we first present the derivation of the results for the factorable case  in \S\ref{subsec:Tridiag_Inv-1dim} and \S\ref{subsec:convexHull}. Then, in \S\ref{subsec:Tridiag_Inv-multidim}, we discuss the generalization of the results to the block-factorable case, where most of the results directly follow from the factorizable case.

\subsection{Inverse of factorizable matrix} \label{subsec:Tridiag_Inv-1dim}
Consider a factorizable matrix $\boldsymbol{Q}$ defined as a Hadamard product of two matrices $\boldsymbol{U}$ and $\boldsymbol{V}$
\begin{align}
    \boldsymbol{Q} = \overbrace{\begin{bmatrix}
        u_1 & u_1 & u_1 & \cdots & u_1 \\
        u_1 & u_2 & u_2 & \cdots & u_2 \\
        u_1 & u_2 & u_3 & \cdots & u_3 \\
        \vdots & \vdots & \vdots & \ddots & \vdots \\
        u_1 & u_2 & u_3 & \cdots & u_n
    \end{bmatrix}}^{\boldsymbol{U}} \circ \overbrace{\begin{bmatrix}
        v_1 & v_2 & v_3 & \cdots & v_n \\
        v_2 & v_2 & v_3 & \cdots & v_n \\
        v_3 & v_3 & v_3 & \cdots & v_n \\
        \vdots & \vdots & \vdots & \ddots & \vdots \\
        v_n & v_n & v_n & \cdots & v_n
    \end{bmatrix}}^{\boldsymbol{V}}. \label{two_seq_mat_repeat}
\end{align}


Note that in this paper, we are interested in positive definite matrices $\bm{Q}\succ \bm{0}$. Naturally, positive definiteness requires diagonal elements to be positive and the determinant of $2\times 2$ principal submatrices to be positive, thus we can make the following assumptions without loss of generality.
\begin{assumption}\label{assump:factorable}
It holds that $u_iv_i>0$ for all $i\in [n]$ and that $u_iv_j(u_jv_i-u_iv_j)>0$ for all $i<j$.
\end{assumption}
In fact, in Proposition~\ref{prop:PSD}, we show that $\bm{Q}\succ \bm{0}$ if and only if Assumption~\ref{assump:factorable} holds. Note that since $u_jv_j>0$, the second condition is equivalent to 
\begin{align}
    \frac{u_i}{u_j} \cdot (u_jv_i-u_iv_j) >0. \label{assump1_cond2_v2}
\end{align}

In Proposition~\ref{prop:rank-1}, we represent the inverse of a factorizable matrix $\boldsymbol{Q}$ as a weighted sum of $n$ symmetric $2\times 2$ rank-one matrices, which directly follows from the recursive expression of $\bm{Q}^{-1}$ presented in \cite{meurant2024direct}. This $2\times 2$ rank-one representation is key to deriving the polynomially-sized convex hull descriptions in~\S\ref{subsec:convexHull}.
\begin{proposition}\label{prop:rank-1}
The inverse of a nonsingular factorizable matrix $\bm{Q}= \bm{U} \circ \bm{V}$ defined as \eqref{two_seq_mat_repeat} can be expressed as the sum of $n$ rank-one matrices 
\begin{align}
    \boldsymbol{Q}^{-1}
    = \sum_{i=1}^n \delta_i \boldsymbol{\gamma}_i \boldsymbol{\gamma}_i^\top, \label{Q_inv_mat}
\end{align}
where 
\begin{gather}
    \boldsymbol{\gamma}_i = \boldsymbol{e}_i - \theta_{i+1} \boldsymbol{e}_{i+1}, \quad i \in [n-1], \quad \text{ and } \quad \boldsymbol{\gamma}_n = \boldsymbol{e}_n, \nonumber\\
    \delta_i = \frac{u_{i+1}}{u_i (u_{i+1} v_i - u_i v_{i+1} )}, \ \ \theta_{i+1} = \frac{u_{i}}{u_{i+1}}, \quad  i \in [n-1], \quad \text{ and } \quad \delta_n = \frac{1}{u_n v_n} \cdot \label{pq_closed_form}
\end{gather}
\end{proposition} 


\begin{proof}
A symmetric tridiagonal representation of $\boldsymbol{Q}^{-1}$ is given in Meurant \cite[Section~1.18]{meurant2024direct} as follows:
\begin{align*}
\boldsymbol{Q}^{-1} = \begin{bmatrix}
    \alpha_1 & -\beta_2 & & &\\
    -\beta_2 & \alpha_2 & -\beta_3 & & \\
    & \ddots & \ddots & \ddots & & \\
    & & -\beta_{n-1} & \alpha_{n-1} & -\beta_n \\
    & & & -\beta_n & \alpha_n
\end{bmatrix}
\end{align*}
where
\begin{equation}
    \begin{aligned}
        &\alpha_1 = \delta_1,\\
        &\beta_i=\frac{1}{u_iv_{i-1}-u_{i-1}v_i},\; \alpha_i = \delta_i + \beta_i \theta_i \qquad i \in [2,n].
    \end{aligned} \label{ab_def}
\end{equation}
Clearly, $\beta_i=\delta_{i-1} \theta_i$ for all $i\in [2,n]$, thus $\alpha_i=\delta_i+\delta_{i-1} \theta_i^2$. 
Thus, we find that $Q_{11}^{-1}=\delta_1$, $Q_{ii}^{-1}=\delta_{i-1} \theta_i^2+ \delta_i$ and $Q_{i,i+1}^{-1}=-\delta_{i} \theta_{i+1}$
which directly induces the representation \eqref{Q_inv_mat}.
\end{proof}


\begin{example} \label{eg:singleton_inv}
Consider a matrix $\boldsymbol{Q}$ given as a Hadamard product of two matrices $\boldsymbol{U}$ and $\boldsymbol{V}$:
\begin{align*}
    \boldsymbol{Q} = \begin{bmatrix}
            5 & 4 & 3 & 2 & 1 \\
            4 & 8 & 6 & 4 & 2 \\
            3 & 6 & 12 & 8 & 4 \\
            2 & 4 & 8 & 16 & 8 \\
            1 & 2 & 4 & 8 & 16 
        \end{bmatrix} = \overbrace{\begin{bmatrix}
            1 & 1 & 1 & 1 & 1 \\
            1 & 2 & 2 & 2 & 2 \\
            1 & 2 & 4 & 4 & 4 \\
            1 & 2 & 4 & 8 & 8 \\
            1 & 2 & 4 & 8 & 16 
        \end{bmatrix}}^{\boldsymbol{U}} \circ \overbrace{\begin{bmatrix}
            5 & 4 & 3 & 2 & 1 \\
            4 & 4 & 3 & 2 & 1 \\
            3 & 3 & 3 & 2 & 1 \\
            2 & 2 & 2 & 2 & 1 \\
            1 & 1 & 1 & 1 & 1 
        \end{bmatrix}}^{\boldsymbol{V}}.
\end{align*}
As matrices $\boldsymbol{U}$ and $\boldsymbol{V}$ are of the form in \eqref{two_seq_mat} with $u_i =2^{i-1}$ and $v_i = 6 - i$, $i \in [5]$, $\boldsymbol{Q}$ is a factorizable matrix. Therefore, the inverse $\boldsymbol{Q}^{-1}$ can be computed as follows:
\begin{align*}
    & \delta_i = \frac{u_{i+1}}{u_i (u_{i+1} v_i - u_i v_{i+1})} = \frac{1}{2^{i-2} (7-i)}, \ \ \theta_{i+1} = \frac{u_i}{u_{i+1}} , \ \  i \in [4], \quad \text{and} \quad \delta_5 = \frac{1}{u_5 v_5} \\
    \Rightarrow \ & \delta_1 = \frac{1}{3}, \ \delta_2 = \frac{1}{5}, \ \delta_3 = \frac{1}{8}, \ \delta_4 = \frac{1}{12}, \ \delta_5 = \frac{1}{16}, \ \ \theta_{i+1} = \frac{1}{2}, \ \forall i \in [4]
\end{align*}
and thus,
\begin{align}
    \bm{Q}^{-1} = \begin{pmatrix}
        \bm{\frac{1}{3}}&\bm{-\frac{1}{3}\cdot \frac{1}{2}}&0&0&0\\
    \bm{-\frac{1}{3}\cdot\frac{1}{2}}&\bm{\frac{1}{3}\cdot \left(\frac{1}{2}\right)^2}+\frac{1}{5}&-\frac{1}{5}\cdot\frac{1}{2}&0&0\\
    0&-\frac{1}{5}\cdot \frac{1}{2}&\frac{1}{5}\cdot\left(\frac{1}{2}\right)^2+\bm{\frac{1}{8}}&\bm{-\frac{1}{8}\cdot \frac{1}{2}}&0\\
    0&0&\bm{-\frac{1}{8}\cdot\frac{1}{2}}&\bm{\frac{1}{8}\cdot\left(\frac{1}{2}\right)^2}+\frac{1}{12}&-\frac{1}{12}\cdot\frac{1}{2}\\
    0&0&0&-\frac{1}{12}\cdot\frac{1}{2}&\frac{1}{12}\cdot\left(\frac{1}{2}\right)^2+\bm{\frac{1}{16}}\end{pmatrix}.\label{ex:rank1_Q}
\end{align}
In \eqref{ex:rank1_Q}, we alternate between bold and non-bold numbers to highlight the $2\times 2$ rank-one matrices adding up to $\bm{Q^{-1}}$. 
\end{example}

\noindent Note that using Proposition~\ref{prop:hull1d} requires describing not just $\boldsymbol{Q}^{-1}$, but also $\hat{\boldsymbol{Q}}_{S}^{-1}$ for all $S \subseteq [n]$, the inverses of all principal submatrices of $\bm{Q}$.  As we point out in Remark~\ref{rem:submatrix1d}, principal submatrices of factorizable matrices are also factorizable. Therefore, we can use Proposition~\ref{prop:rank-1} to describe their inverses. The closed form expressions for matrices $\hat{\boldsymbol{Q}}_{S}^{-1}$ are given in Corollary~\ref{cor:submat_pd}.

\begin{remark}\label{rem:submatrix1d}
Given an index set $S \subseteq [n]$, the principal submatrix of a factorizable matrix $\bm{Q} = \bm{U} \circ \bm{V}$ induced by $S$ is a factorizable matrix that can be represented as $\bm{Q}_S = \bm{U}_S \circ \bm{V}_S$, where $\boldsymbol{U}_{S}$ and $\boldsymbol{V}_{S}$ are the principal submatrices of $\ \boldsymbol{U}$ and $\boldsymbol{V}$ induced by $S$, respectively.
\end{remark}

\begin{cor} \label{cor:submat_pd}
    Given an index set $S = \left\{t_1, \ldots, t_k\right\} \subseteq [n]$ with $1 \le t_1 <\cdots<t_k \le n $, define 
    \begin{align}
        \delta^{S}_{\ell} = \frac{u_{t_{\ell+1}}}{u_{t_{\ell}} \left(u_{t_{\ell+1}} v_{t_{\ell}} - u_{t_{\ell}} v_{t_{\ell+1}} \right)}, \ \ \theta_{\ell+1}^{S} &= \frac{u_{t_{\ell}}}{u_{t_{\ell+1}}}, \quad  \ell \in [k-1], \quad \text{ and } \quad \delta^{S}_k = \frac{1}{u_{t_k} v_{t_k}} \cdot \label{sub_pq_closed_form}
    \end{align}
    Then, letting $\ \boldsymbol{\gamma}^{S}_{\ell} = \boldsymbol{e}_{t_{\ell}} - \theta_{\ell+1}^{S} \boldsymbol{e}_{t_{\ell+1}}$ for $\ell \in [k-1]$  and $\ \boldsymbol{\gamma}_k^S = \boldsymbol{e}_{t_k}$, matrix $\hat{\boldsymbol{Q}}_{S}^{-1}$ can be expressed as 
    \begin{align}
        \hat{\boldsymbol{Q}}_{S}^{-1} = \sum_{\ell=1}^k \delta^{S}_{\ell} \boldsymbol{\gamma}^{S}_{\ell} \left(\boldsymbol{\gamma}^{S}_{\ell}\right)^\top.
        \label{inv_singleton}
    \end{align}

\end{cor}
\noindent Example~\ref{eg:singleton_inv_2} below illustrates a key property arising from Corollary~\ref{cor:submat_pd}: several of the rank-one matrices appearing in the representations of $\bm{Q^{-1}}$ and $\hat{\bm{Q}}_S^{-1}$ are the same.

\setcounter{example}{1}
\begin{example}[Continued] \label{eg:singleton_inv_2}
 Suppose $S = \{1,2,4,5\}$. Then, $\bm{Q_S}$
 is a factorizable matrix constructed by $\{u_i\}_{i \in S} = \{1,2,8,16\}$ and $\{v_i\}_{i \in S} = \{5,4,2,1\}$
\begin{align*}
    &\boldsymbol{Q}_{S} = \begin{bmatrix}
            5 & 4 & 2 & 1 \\
            4 & 8 & 4 & 2 \\
            2 & 4 & 16 & 8 \\
            1 & 2 & 8 & 16
        \end{bmatrix} = \overbrace{\begin{bmatrix}
            1 & 1 & 1 & 1 \\
            1 & 2 & 2 & 2 \\
            1 & 2 & 8 & 8 \\
            1 & 2 & 8 & 16
        \end{bmatrix}}^{\boldsymbol{U}_{S}} \circ \overbrace{\begin{bmatrix}
            5 & 4 & 2 & 1\\
            4 & 4 & 2 & 1\\
            2 & 2 & 2 & 1 \\
            1 & 1 & 1 & 1
        \end{bmatrix}}^{\boldsymbol{V}_{S}} 
\end{align*}
and the corresponding $\{\delta_i\}_{i \in [4]}$ and $\{\theta_i\}_{i \in [2,4]}$ are computed as
\begin{align*}
    & \delta_1^S = \frac{1}{3}, \ \ \delta_2^S = \frac{1}{7}, \ \ \delta_3^S = \frac{1}{12}, \ \ \delta_4^S = \frac{1}{16}, \ \ \theta_2^S = \frac{1}{2}, \ \ \theta_3^S = \frac{1}{4}, \ \ \theta_4^S = \frac{1}{2}.
\end{align*}
Therefore, we find that 
\begin{align}
    \hat{\bm{ Q}}_S^{-1} = \begin{pmatrix}
        \bm{\frac{1}{3}}&\bm{-\frac{1}{3}\cdot \frac{1}{2}}&0&0&0\\
    \bm{-\frac{1}{3}\cdot\frac{1}{2}}&\bm{\frac{1}{3}\cdot \left(\frac{1}{2}\right)^2}+\frac{1}{7}&0&-\frac{1}{7}\cdot \frac{1}{4}&0\\
    0&0&0&0&0\\
    0&-\frac{1}{7}\cdot \frac{1}{4}&0&\frac{1}{7}\cdot\left(\frac{1}{4}\right)^2+\bm{\frac{1}{12}}&\bm{-\frac{1}{12}\cdot\frac{1}{2}}\\
    0&0&0&\bm{-\frac{1}{12}\cdot\frac{1}{2}}&\bm{\frac{1}{12}\cdot\left(\frac{1}{2}\right)^2}+\bm{\frac{1}{16}}\end{pmatrix}.\label{ex:rank1_Q_S}
\end{align}

Observe that every bold element in \eqref{ex:rank1_Q_S} also appears in the description of $\bm{Q^{-1}}$ in \eqref{ex:rank1_Q}. In fact, the only new rank-one matrix in the description of $\hat{\boldsymbol{Q}}_{S}^{-1}$ is of the form $\frac{1}{7}\begin{pmatrix}1&-1/4\\-1/4&1/16\end{pmatrix}$ corresponding to indexes $\{2,4\}$, that is, the indexes adjacent to the missing index $3\not\in S$.
\hfill $\blacksquare$
\end{example}

We now formalize the property revealed by Example~\ref{eg:singleton_inv_2}. Observe from \eqref{sub_pq_closed_form} that computing $\delta_\ell^S$ and $\theta_{\ell+1}^S$ requires only two indexes $t_\ell$ and $t_{\ell+1}$, and does not depend on the remaining elements of $S$. 

\begin{definition} \label{def:consecutive1d}
Given a factorizable matrix $\bm{Q} = \bm{U} \circ \bm{V}$ given by \eqref{two_seq_mat_repeat}, define rank-one matrices $\bm{\Lambda}_{i \to j }$ as
\begin{align*}
    \bm{\Lambda}_{i\to j}:=\frac{u_j}{u_i\left(u_jv_i-u_iv_j\right)}\left(\bm{e}_i-\frac{u_i}{u_j}\bm{e}_j\right)\left(\bm{e}_i-\frac{u_i}{u_j}\bm{e}_j\right)^\top, \   1 \leq i < j \leq n, \quad \text{and} \quad \bm{\Lambda}_{i\to n+1}:=\frac{1}{u_iv_i} \bm{e}_i \bm{e}_i^\top, \ \  1 \leq i \leq n.
\end{align*}
\end{definition}

\begin{obs} \label{obs:consecutive1d}
Given a factorizable matrix $\bm{Q}$ and an index set $S\subseteq [n]$, if $i<j$ are consecutive elements of $\ S$ (i.e., $i,j\in S$ and $\ell\not\in S$ for $i<\ell<j$), then the rank-one matrix
$\bm{\Lambda}_{i\to j}$ appears in the representation \eqref{inv_singleton} of $\ \hat{\bm{Q}}_S^{-1}$. Similarly, if $i$ is the last element of $S$, then the rank-one matrix $\bm{\Lambda}_{i\to n+1}:=\frac{1}{u_iv_i}\bm{e}_i\bm{e}_i^\top$ appears in the representation \eqref{inv_singleton} of $\ \hat{\bm{Q}}_S^{-1}$. Moreover, only $\ {n+1 \choose 2}$ rank-one matrices $\bm{\Lambda}_{i \to j}$ are needed to represent $\hat{\bm{Q}}_S^{-1}$ for all 
 $S \subseteq [n]$.
\end{obs}


\subsection{Convexification of \texorpdfstring{$X_{\boldsymbol{Q}}$}{P\_{\textbf{Q}\^B}}}\label{subsec:convexHull}

We present a compact representation of $P_{\boldsymbol{Q}\sz{dn}}$ defined in Proposition~\ref{prop:hull1d}, which we repeat for convenience:
\begin{align*}
    P_{\boldsymbol{Q}\sz{dn}} = \conv \left( \left\{ \left(\boldsymbol{e}_S, \hat{\boldsymbol{Q}}_{S}\sz{dn}^{-1} \right)\right\}_{S \subseteq [n]}\right).
\end{align*}
We start by giving a mixed-integer representation of the set of extreme points $\ext(P_{\bm{Q}})=\left\{ \left(\boldsymbol{e}_S, \hat{\boldsymbol{Q}}_{S}\sz{dn}^{-1} \right)\right\}_{S \subseteq [n]}$ in an extended formulation. Motivated by Observation~\ref{obs:consecutive1d}, we introduce for all $1\leq i<j\leq n$ variables $w_{ij}$ which equal to one if and only if $i$ and $j$ are consecutive elements of $S$, and zero otherwise. Moreover, we also introduce variables $w_{0i}=1$ if and only if $i$ is the first element of $S$, and $w_{i,n+1}=1$ if and only $i$ is the last element of $S$. Consider the linear constraints 
\begin{subequations}\label{eq:path}
\begin{align}
    &\sum_{i=0}^{\ell-1} w_{i \ell} - \sum_{j=\ell+1}^{n+1} w_{\ell j} = \begin{cases}
            -1, & \text{ if } \ \ell = 0\\
            1, & \text{ if } \ \ell = n+1\\
            0, & \text{ otherwise}
        \end{cases} \qquad \ell\in\{0,\dots,n+1\}\label{spp_flow_1d}\\
      &z_\ell  = \sum_{i=0}^{\ell-1} w_{i\ell},  \qquad  \ell \in [n] \label{spp_z_1d}\\ 
    &\boldsymbol{W} = \sum_{i=1}^{n} \sum_{j=i+1}^{n+1}  \boldsymbol{\Lambda}_{i\to j} w_{ij}\label{inv_decomp_1d}\\
    &\bm{w}\geq \bm{0}\label{eq:path_nonneg}
\end{align}
\end{subequations}
where matrices $\bm{\Lambda}_{i\to j}$ are given in Definition~\ref{def:consecutive1d}. Intuitively, due to constraints \eqref{spp_flow_1d}, $\bm{w}$ represents the arcs of a $0-(n+1)$ path in a directed acyclic graph $\mathcal{G}=(V,A)$ where $V=\{0,\dots, n+1\}$ and $A=\{(i,j): 0\leq i<j\leq n+1$\}, illustrated in Figure~\ref{fig:DAG}. 
Moreover, constraints \eqref{spp_z_1d} indicate that the nodes visited by a chosen path are precisely those corresponding to indexes $\ell\in [n]$ such that $z_\ell=1$. Proposition~\ref{prop:valid_1d} states that constraints \eqref{eq:path} are valid for $\ext(P_{\bm{Q}})$, and Proposition~\ref{prop:integral} states that the polytope \eqref{eq:path} is integral in $\bm{z}$ and $\bm{w}$.

\begin{figure}[htp]
    \centering
    \includegraphics[width=0.55\linewidth]{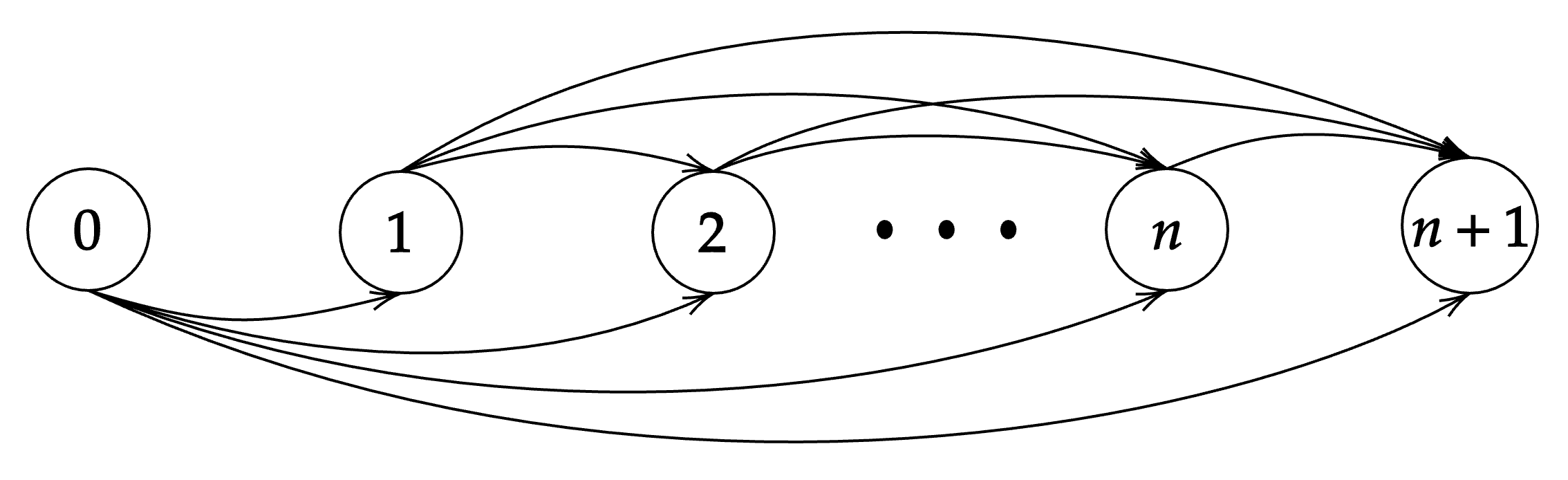}
    \caption{Directed acyclic graph $\mathcal{G}$: arc $(i,j)$ represents rank-one matrix  $\boldsymbol{\Lambda}_{i\to j}$.}
    \label{fig:DAG}
\end{figure}

\begin{proposition}\label{prop:valid_1d}
Given any $\bm{z}\in \{0,1\}^n$, the unique solution $(\bm{w},\bm{W})$ to \eqref{eq:path} satisfies $\bm{W}=\hat{\bm{Q}}_S^{-1}$, where $S=\left\{\ell\in [n]: z_\ell=1\right\}$.
\end{proposition}
\begin{proof}
The path constraints ensure that $w_{ij}=1$ if and only $i$ and $j$ are consecutive non-zero indexes of $\bm{z}$, and $w_{i,n+1}=1$ if and only if $i$ is the last index satisfying $z_i=1$. From constraint \eqref{inv_decomp_1d} and Observation~\ref{obs:consecutive1d}, we directly conclude that $\bm{W}=\hat{\bm{Q}}_S^{-1}$. \end{proof}

\begin{proposition}\label{prop:integral}
All extreme points of the polytope defined by constraints \eqref{eq:path} are integral in $\bm{w}$ and $\bm{z}$. 
\end{proposition}
\begin{proof}
Since extreme points of a polytope correspond to unique solutions to linear optimization problems, consider the optimization problem
$$\min_{\bm{z},\bm{w},\bm{W}}\bm{c^\top z}+\bm{b^\top w}+\sum_{i=1}^n\sum_{j=1}^n r_{ij}W_{ij} \text{  s.t.  }\eqref{eq:path}.$$
In this optimization problem, variables $\bm{z}$ and $\bm{W}$ can be easily projected out using equalities \eqref{spp_z_1d} and \eqref{inv_decomp_1d}, respectively, resulting in an optimization problem with just variables $\bm{w}$ and the totally unimodular constraints \eqref{spp_flow_1d} and \eqref{eq:path_nonneg}. Thus, there exists an optimal solution integral in $\bm{w}$, which in turn implies integrality in $\bm{z}$. 
\end{proof}

From Propositions~\ref{prop:valid_1d} and \ref{prop:integral}, we conclude that $$P_{\bm{Q}}=\left\{(\bm{z},\bm{W})\in \R^{n+n^2}: \exists \bm{w}\in \R^{(n+1)(n+2)/2}\text{ such that \eqref{eq:path} holds} \right\}.$$
Using Proposition~\ref{prop:hull1d}, we obtain a description for the convex hull of $X_{\bm{Q}}$.
\begin{thm} \label{cor:conv_singleton}
Given a factorizable matrix $\bm{Q}=\bm{U}\circ \bm{V}$, 
\begin{equation}
    \begin{aligned}
        \cl \conv \left( X_{\boldsymbol{Q}}\right)=\Bigg\{(\bm{x},\bm{z},\tau) \in \R^{2n+1}: \ & \exists \bm{W}\in \R^{n\times n}, \  \bm{w} \in \mathbb{R}_+^{(n+1)(n+2)/2} \ \text{ such that}
        &\begin{bmatrix}
            \boldsymbol{W} & \boldsymbol{x}\\
            \boldsymbol{x}^\top & \tau
        \end{bmatrix} \succeq \bm{0} \text{ and }  \eqref{eq:path}\Bigg\}.
    \end{aligned} \label{convX_singleton_sdp}
\end{equation}
\end{thm}

\subsection{Generalization to the block-factorable case} \label{subsec:Tridiag_Inv-multidim}

We extend the results in \S\ref{subsec:Tridiag_Inv-1dim} and \S\ref{subsec:convexHull} to derive a closed convex hull representation of $X_{\bm{Q}}^B$ for a block-factorizable positive definite matrix $\bm{Q}$. The structure, intuition, and results are essentially the same as in the factorizable case.

Consider a block-factorizable matrix $\boldsymbol{Q}\sz{dn} \in \mathbb{R}^{dn \times dn}$ defined by matrices $\boldsymbol{U}\sz{dn}$ and $\boldsymbol{V}\sz{dn}$
\begin{align}
    \overbrace{\begin{bmatrix}
        \boldsymbol{U}_1 \boldsymbol{V}_1^\top & \boldsymbol{U}_1 \boldsymbol{V}_2^\top & \boldsymbol{U}_1 \boldsymbol{V}_3^\top & \cdots & \boldsymbol{U}_1 \boldsymbol{V}_n^\top \\
        \boldsymbol{V}_2 \boldsymbol{U}_1^\top & \boldsymbol{U}_2 \boldsymbol{V}_2^\top & \boldsymbol{U}_2 \boldsymbol{V}_3^\top & \cdots & \boldsymbol{U}_2 \boldsymbol{V}_n^\top \\
        \boldsymbol{V}_3 \boldsymbol{U}_1^\top & \boldsymbol{V}_3 \boldsymbol{U}_2^\top & \boldsymbol{U}_3 \boldsymbol{V}_3^\top & \cdots & \boldsymbol{U}_3 \boldsymbol{V}_n^\top \\
        \vdots & \vdots & \vdots & \ddots & \vdots \\
        \boldsymbol{V}_n \boldsymbol{U}_1^\top & \boldsymbol{V}_n \boldsymbol{U}_2^\top & \boldsymbol{V}_n \boldsymbol{U}_3^\top & \cdots & \boldsymbol{U}_n \boldsymbol{V}_n^\top
    \end{bmatrix}}^{\boldsymbol{Q}\sz{dn}} := \overbrace{\begin{bmatrix}
        \boldsymbol{U}_1 & \boldsymbol{U}_1 & \boldsymbol{U}_1 & \cdots & \boldsymbol{U}_1 \\
        \boldsymbol{U}_1 & \boldsymbol{U}_2 & \boldsymbol{U}_2 & \cdots & \boldsymbol{U}_2 \\
        \boldsymbol{U}_1 & \boldsymbol{U}_2 & \boldsymbol{U}_3 & \cdots & \boldsymbol{U}_3 \\
        \vdots & \vdots & \vdots & \ddots & \vdots \\
        \boldsymbol{U}_1 & \boldsymbol{U}_2 & \boldsymbol{U}_3 & \cdots & \boldsymbol{U}_n
    \end{bmatrix}}^{\boldsymbol{U}\sz{dn}} \bullet \overbrace{\begin{bmatrix}
        \boldsymbol{V}_1 & \boldsymbol{V}_2 & \boldsymbol{V}_3 & \cdots & \boldsymbol{V}_n \\
        \boldsymbol{V}_2 & \boldsymbol{V}_2 & \boldsymbol{V}_3 & \cdots & \boldsymbol{V}_n \\
        \boldsymbol{V}_3 & \boldsymbol{V}_3 & \boldsymbol{V}_3 & \cdots & \boldsymbol{V}_n \\
        \vdots & \vdots & \vdots & \ddots & \vdots \\
        \boldsymbol{V}_n & \boldsymbol{V}_n & \boldsymbol{V}_n & \cdots & \boldsymbol{V}_n
    \end{bmatrix}}^{\boldsymbol{V}\sz{dn}}. \label{two_seq_block_mat}
\end{align}

The conditions guaranteeing the positive definiteness of a block matrix $\bm{Q}$ naturally extend from those for a positive definite matrix with $d=1$.
Specifically, every diagonal block, as well as the Schur complement of each $2 \times 2$ block principal submatrix, need to be positive definite.
Hence, we adopt the following assumption.

\begin{assumption} \label{assum:block-factorizable}
The block-factorizable matrix $\bm{Q}$ given by \eqref{two_seq_block_mat} satisfies $\bm{U}_i \bm{V}_i^\top \succ \bm{0}$ for all $i \in [n]$ and $\bm{U}_i \bm{V}_i^\top- \bm{U}_i \bm{U}_j^{-1} \bm{V}_j \bm{U}_i^\top \succ \bm{0}$, or equivalently $\bm{U}_i \bm{U}_j^{-1} \left( \bm{U}_j \bm{V}_i^\top- \bm{V}_j \bm{U}_i^\top\right) \succ \bm{0}$, for all $1 \leq i < j \leq n$.
\end{assumption}
\noindent Note that $\bm{U}_i \bm{V}_i^\top \succ \bm{0}$ guarantees the nonsingularity of $\bm{U}_i$ and $\bm{V}_i$ as $\det(\bm{U}_i \bm{V}_i^\top) = \det(\bm{U}_i ) \cdot \det(\bm{V}_i) > 0$. Additionally, if the matrices are taken to be singletons, the conditions in Assumption~\ref{assum:block-factorizable} reduce to those in Assumption~\ref{assump:factorable}.

In Proposition~\ref{prop:PD_block}, we demonstrate that Assumption~\ref{assum:block-factorizable} is a necessary and sufficient condition for $\bm{Q} \succ \bm{0}$. Although it has been shown that the inverse of a block tridiagonal matrix is a block-factorizable matrix in \cite{meurant1992review}, a closed-form representation of the inverse of a block-factorizable matrix has not been presented, to the best of our knowledge. Thus, in Proposition \ref{prop:block_factorizable_inverse}, we derive a closed-form expression for the inverse $\bm{Q}^{-1}$, and we express it as a sum of $n$ symmetric rank-$d$ matrices. This representation allows us to extend the results from the factorizable case to the block-factorizable case. We provide the proof of Proposition~\ref{prop:block_factorizable_inverse} in Appendix \ref{Appendix:proof_prop_block_inv} due to its length. 

\begin{proposition} \label{prop:block_factorizable_inverse}
For a nonsingular block-factorizable matrix $\boldsymbol{Q}\sz{dn}$ defined as \eqref{two_seq_block_mat}, its inverse is a block tridiagonal matrix that can be decomposed into $n$ of rank-$d$ matrices 
\begin{align*}
    \boldsymbol{Q}\sz{dn}^{-1} 
    & = \sum_{i=1}^{n} \boldsymbol{\Gamma}_i \boldsymbol{\Delta}_i \boldsymbol{\Gamma}_{i}^\top ,
\end{align*}
where
\begin{gather*}
    \boldsymbol{\Gamma}_i = \boldsymbol{E}_i - \boldsymbol{E}_{i+1} \boldsymbol{\Theta}_{i+1}^\top, \quad  i \in [n-1], \quad \text{ and } \quad \boldsymbol{\Gamma}_n = \boldsymbol{E}_{n}, \\[0.5em]
    \boldsymbol{\Delta}_i = \left( \boldsymbol{U}_i \boldsymbol{V}_i^\top - \boldsymbol{U}_i \boldsymbol{U}_{i+1}^{-1} \boldsymbol{V}_{i+1} \boldsymbol{U}_i^\top \right)^{-1}, \ \boldsymbol{\Theta}_{i+1} = \boldsymbol{U}_{i} \boldsymbol{U}_{i+1}^{-1}, \quad  i \in [n-1], \quad \text{ and } \quad
    \boldsymbol{\Delta}_n = \left( \boldsymbol{U}_n \boldsymbol{V}_n^\top \right)^{-1}.
\end{gather*}
\end{proposition}

As we discuss in Remark~\ref{remark:block_factorizable}, any block principal submatrix of a block-factorizable matrix is also block-factorizable, and it has the same properties. Therefore, the rank-$d$ decomposition in Proposition~\ref{prop:block_factorizable_inverse} can be extended to inverses of principal submatrices, $\bm{Q}_S^{-1}$, as stated in Corollary~\ref{cor:block_submat_pd}.

\begin{remark} \label{remark:block_factorizable}
Given a set $S \subseteq [n]$ and a block-factorizable matrix $\boldsymbol{Q}\sz{dn} = \bm{U} \bullet \bm{V}$, the block principal submatrix of $\bm{Q}$ defined by the set $S$ is also a block-factorizable matrix of the form $\bm{Q}_{[S]} = \bm{U}_{[S]} \bullet \bm{V}_{[S]}$, where $\bm{U}_{[S]}$ and $\bm{V}_{[S]}$ are block principal submatrices of $\ \bm{U}$ and $\bm{V}$ induced by $S$, respectively.
\end{remark}

\begin{cor} \label{cor:block_submat_pd}
Given an index set $S = \left\{t_1, \ldots, t_k\right\} \subseteq [n]$ with $1 \leq t_1 < \cdots < t_k \leq n$, define 
\begin{align*}
    \boldsymbol{\Delta}^{S}_{\ell} = \left( \boldsymbol{U}_{t_{\ell}} \boldsymbol{V}_{t_{\ell}}^\top - \boldsymbol{U}_{t_{\ell}} \boldsymbol{U}_{t_{\ell+1}}^{-1} \boldsymbol{V}_{t_{\ell+1}} \boldsymbol{U}_{t_{\ell}}^\top \right)^{-1}, \ \boldsymbol{\Theta}_{\ell+1}^{S} = \boldsymbol{U}_{t_{\ell}} \boldsymbol{U}_{t_{\ell+1}}^{-1}, \quad  \ell \in [k-1], \quad \text{ and } \quad \boldsymbol{\Delta}^{S}_k  = \left( \boldsymbol{U}_{t_k} \boldsymbol{V}_{t_k}^\top \right)^{-1}.
\end{align*}
Then, letting $\ \boldsymbol{\Gamma}_{\ell}^{S} = \boldsymbol{E}_{t_{\ell}} - \boldsymbol{E}_{t_{\ell+1}} \left(\boldsymbol{\Theta}_{\ell+1}^{S}\right)^\top$ for $\ell \in [k-1]$ and $\ \boldsymbol{\Gamma}_k^{S} = \boldsymbol{E}_{t_k}$,
the matrix $\hat{\bm{Q}}_{[S]}^{-1}$ is given by
\begin{align}
    \hat{\boldsymbol{Q}}_{[S]}^{-1} = \sum_{\ell=1}^{k} \boldsymbol{\Gamma}_{\ell}^{S} \boldsymbol{\Delta}_{\ell}^{S} \left( \boldsymbol{\Gamma}_{\ell}^{S} \right)^\top. \label{Qd_inv_rankd}
\end{align}
\end{cor}

As in the factorizable case, we observe that $\boldsymbol{\Gamma}_{\ell}^{S}$ and $\boldsymbol{\Delta}_{\ell}^{S}$ are solely determined by the two consecutive indexes $t_{\ell}$ and $t_{\ell+1}$ in $S$. 

\begin{definition} \label{def:consecutive_dd}
Given a block-factorizable matrix $\bm{Q} = \bm{U} \bullet \bm{V} \in \R^{dn \times dn}$ defined by $\{\bm{U}_i\}_{i\in [n]}$ and $\{\bm{V}_i\}_{i\in [n]}$, define rank-$d$ matrices $\bm{\Lambda}_{[i \to j] }$ as
\begin{align*}
    \bm{\Lambda}_{[i\to j]}:=\left( \bm{E}_i - \bm{E}_j \bm{U}_j^{-\top}\bm{U}_i^\top \right) \left( \boldsymbol{U}_i \boldsymbol{V}_i^\top - \boldsymbol{U}_i \boldsymbol{U}_{j}^{-1} \boldsymbol{V}_{j} \boldsymbol{U}_i^\top \right)^{-1} \left( \bm{E}_i - \bm{E}_j \bm{U}_j^{-\top}\bm{U}_i^\top \right)^\top, \qquad 1 \leq i < j \leq n
\end{align*}
and
\begin{align*}
    \bm{\Lambda}_{[i\to n+1]}:=\bm{E}_i \left(\bm{U}_i \bm{V}_i^\top \right)^{-1} \bm{E}_i^\top, \qquad 1 \leq i \leq n.
\end{align*}
\end{definition}

\begin{obs} \label{obs:consecutive_dd}
Given a block-factorizable matrix $\bm{Q}$ and an index set $S\subseteq [n]$, if $\ i<j$ are consecutive elements of $S$, then the rank-$d$ matrix
$\bm{\Lambda}_{[i\to j]}$ appears in the representation \eqref{Qd_inv_rankd} of $\ \hat{\bm{Q}}_{[S]}^{-1}$. Likewise, if $\ i$ is the last element of $S$, then the rank-$d$ matrix $\bm{\Lambda}_{[i\to n+1]}$ appears in the representation \eqref{Qd_inv_rankd} of $\ \hat{\bm{Q}}_{[S]}^{-1}$. Furthermore, computation of $\ {n+1 \choose 2}$ rank-$d$ matrices $\bm{\Lambda}_{[i \to j]}$ are sufficient to represent $\hat{\bm{Q}}_{[S]}^{-1}$ for every $S \subseteq [n]$.
\end{obs}


\begin{example} \label{eg:block_inv}
Consider a block-factorizable matrix $\boldsymbol{Q} \in \R^{(2\cdot 4) \times (2 \cdot 4)}$ defined by two block matrices $\bm{U}, \bm{V} \in \R^{(2\cdot 4) \times (2 \cdot 4)}$ given as 
\begin{align*}
    \overbrace{\left[\begin{array}{cc|cc|cc|cc}
        5 & 6 & 4 & 5 & 3 & 4 & 2 & 3 \\
        6 & 11 & 5 & 9 & 4 & 7 & 3 & 5 \\ \hline
        4 & 5 & 8 & 10 & 6 & 8 & 4 & 6 \\
        5 & 9 & 10 & 18 & 8 & 14 & 6 & 10 \\ \hline
        3 & 4 & 6 & 8 & 12 & 16 & 8 & 12 \\
        4 & 7 & 8 & 14 & 16 & 28 & 12 & 20 \\ \hline
        2 & 3 & 4 & 6 & 8 & 12 & 16 & 24 \\
        3 & 5 & 6 & 10 & 12 & 20 & 24 & 40 
    \end{array}\right]}^{\boldsymbol{Q}} = \overbrace{\left[\begin{array}{cc|cc|cc|cc}
         1 & 1 & 1 & 1 & 1 & 1 & 1 & 1 \\
         1 & 2 & 1 & 2 & 1 & 2 & 1 & 2 \\ \hline
         1 & 1 & 2 & 2 & 2 & 2 & 2 & 2 \\
         1 & 2 & 2 & 4 & 2 & 4 & 2 & 4 \\ \hline
         1 & 1 & 2 & 2 & 4 & 4 & 4 & 4 \\
         1 & 2 & 2 & 4 & 4 & 8 & 4 & 8 \\ \hline
         1 & 1 & 2 & 2 & 4 & 4 & 8 & 8 \\
         1 & 2 & 2 & 4 & 4 & 8 & 8 & 16
    \end{array}\right]}^{\boldsymbol{U}} \bullet \overbrace{\left[\begin{array}{cc|cc|cc|cc}
         4 & 1 & 3 & 1 & 2 & 1 & 1 & 1 \\
         1 & 5 & 1 & 4 & 1 & 3 & 1 & 2 \\ \hline
         3 & 1 & 3 & 1 & 2 & 1 & 1 & 1 \\
         1 & 4 & 1 & 4 & 1 & 3 & 1 & 2 \\ \hline
         2 & 1 & 2 & 1 & 2 & 1 & 1 & 1 \\
         1 & 3 & 1 & 3 & 1 & 3 & 1 & 2 \\ \hline
         1 & 1 & 1 & 1 & 1 & 1 & 1 & 1 \\
         1 & 2 & 1 & 2 & 1 & 2 & 1 & 2 
    \end{array}\right]}^{\boldsymbol{V}}.
\end{align*}
The rank-$d$ matrix $\bm{\Lambda}_{[i \to j]}$ for $1 \leq i < j \leq 3$ can be computed as follows:
\begin{align*}
    \bm{\Lambda}_{[1 \to 2]} &= \frac{2}{29} \left(\bm{E}_1 - \frac{1}{2}\bm{E}_2 \right) \begin{pmatrix}
         13 & -7 \\ -7 & 6
    \end{pmatrix} \left(\bm{E}_1 - \frac{1}{2}\bm{E}_2 \right)^\top\\[2pt]
    \bm{\Lambda}_{[2 \to 3]} &= \frac{1}{19} \left(\bm{E}_2 - \frac{1}{2}\bm{E}_3 \right) \begin{pmatrix}
         11 & -6 \\ -6 & 5
    \end{pmatrix} \left(\bm{E}_2 - \frac{1}{2}\bm{E}_3 \right)^\top\\[2pt]
    \bm{\Lambda}_{[1 \to 3]} &= \frac{4}{229} \left(\bm{E}_1 - \frac{1}{4}\bm{E}_3 \right) \begin{pmatrix}
         37 & -20 \\ -20 & 17
    \end{pmatrix} \left(\bm{E}_1 - \frac{1}{4}\bm{E}_3 \right)^\top.
\end{align*}
\end{example}

We now present a compact representation of the closed convex hull of $X_{\bm{Q}}^B$ for a block-factorizable matrix $\bm{Q}$. Note that $P_{\boldsymbol{Q}}$ for a factorizable matrix is a special case of $P_{\boldsymbol{Q}\sz{dn}}^B$ with $d=1$. Therefore, the construction of $P_{\boldsymbol{Q}\sz{dn}}^B$ is essentially the same as that of $P_{\boldsymbol{Q}}$, which can be obtained only by replacing \eqref{inv_decomp_1d} with 
\begin{align}
    \boldsymbol{W} = \sum_{i=1}^{n} \sum_{j=i+1}^{n+1}  \boldsymbol{\Lambda}_{[i\to j]} w_{ij}.\label{inv_decomp_dd}
\end{align}
Following the same reasoning, a closed convex hull representation of $X_{\bm{Q}}^B$ for a block-factorizable matrix $\bm{Q}$ is obtained as in Theorem~\ref{cor:conv_block}.

\begin{thm} \label{cor:conv_block}
Given a block-factorizable matrix $\bm{Q}=\bm{U}\bullet \bm{V}$, 
\begin{equation}
    \begin{aligned}
        \cl \conv \left( X_{\boldsymbol{Q}}^B\right)=\Bigg\{&(\bm{x},\bm{z},\tau) \in \R^{(d+1)n+1}: \\
        & \quad \exists \bm{W}\in \R^{dn\times dn}, \  \bm{w} \in \mathbb{R}_+^{(n+1)(n+2)/2} \text{ such that } \begin{bmatrix}
            \boldsymbol{W} & \boldsymbol{x}\\
            \boldsymbol{x}^\top & \tau
        \end{bmatrix} \succeq \bm{0},\; \eqref{spp_flow_1d}, \eqref{spp_z_1d}, \eqref{eq:path_nonneg}, \eqref{inv_decomp_dd}\Bigg\}.
    \end{aligned} \label{convX_block_sdp}
\end{equation}
\end{thm}

\section{Algorithms}\label{sec:algorithm}

In \S\ref{sec:convexification} we showed how mixed-integer optimization problems with (block-) factorable matrices can be reformulated as convex optimization problems using a polynomially-sized extended formulation. In particular, problems~\eqref{MIQP-ddim} and \eqref{MIQP-block} without the side constraints can be directly reformulated as a convex optimization problem. While such convex optimization problems theoretically can be solved efficiently (in polynomial time), practical challenges arise when handling positive semi-definite constraints involving matrices of size $dn\times dn$, especially when $n$ is moderately large. In this section, we discuss how to utilize the proposed relaxations more effectively. First, in \S\ref{sec:dynamicProgramming}, 
we show that for a problem without side constraints, the convex relaxation leads to a natural combinatorial algorithm by reducing the problem to a shortest path problem on a directed acyclic graph. Then, in \S\ref{sec:socp}, we address the side constraints and show that the proposed relaxations are SOCP-representable. This allows the use of off-the-shelf mixed-integer conic quadratic solvers, which, in practice, achieve provably optimal solutions by requiring significantly fewer branch-and-bound nodes compared to standard big-M relaxations.

\subsection{Shortest path algorithm}\label{sec:dynamicProgramming}
In this section, we present a combinatorial approach to solve the SDP formulation of the MIQP \eqref{MIQP-block} with $\Tilde{C} = \R^{dn} \times \R^n$:
\begin{equation}
    \begin{aligned}
        \min \ \ & \tau + \bm{a}^\top \bm{x} + \bm{c}^\top \bm{z} \\
        \text{s.t. } \ & (\bm{x}, \bm{z}, \tau) \in \cl \conv \left(X_{\bm{Q}}^B\right),
    \end{aligned} \label{SDP}
\end{equation}
where $\cl \conv \left(X_{\bm{Q}}^B \right)$ is defined as \eqref{convX_block_sdp}.
In this section we consider the general block-factorable case, and results for the factorable case can be recovered by setting $d=1.$ 

Note that since relaxation \eqref{SDP} describes the convex hull of the mixed-integer problem \eqref{MIQP-block}, we find that $\bm{z^*}\in \{0,1\}^n$ in extreme optimal solutions. 
Due to Proposition~\ref{prop:valid_1d}, which naturally extends to the block-factorizable case by replacing \eqref{inv_decomp_1d} with \eqref{inv_decomp_dd}, 
we find that $\bm{W^*}=\hat{\bm{Q}}_{[S]}^{-1}$, where $S$ is the support of $\bm{z^*}$. The positive definite constraint implies $\bm{x}_{[i]}=0$ for $i\not \in S$ and reduces to $\begin{pmatrix}\bm{Q}_{[S]}^{-1}&\bm{x}_{[S]}\\[2pt] \bm{x}_{[S]}^\top & \tau
\end{pmatrix}\succeq \bm{0}\Leftrightarrow \tau\geq \bm{x}_{[S]}^\top\bm{Q}_{[S]}\bm{x}_{[S]}$. 
Finally, projecting out $\tau$, we deduce from the first-order optimality conditions that $\bm{x}_{[S]}=-\frac{1}{2}\bm{Q}_{[S]}^{-1}\bm{a}_{[S]}$. Then, projecting out variable $\bm{x}=-\frac{1}{2}\bm{W}\bm{a}$ (capturing both zero and non-zero values of $\bm{x}$), we find that \eqref{SDP} reduces to

\begin{subequations}
    \begin{align}
        \min \ & -\frac{1}{4} \bm{a}^\top \bm{W} \bm{a} + \bm{c}^\top \bm{z} \label{linear_obj}\\
        \text{ s.t. } & (\bm{z}, \bm{W}) \in P_{\bm{Q}}^B. \label{linear_cvx_constr}
    \end{align}\label{linear_form}
\end{subequations}

We now show that an optimal solution of \eqref{linear_form} can be obtained by solving a shortest path problem on a directed acyclic graph.

\begin{proposition} \label{prop:DP}
Problem \eqref{linear_form} can be solved in $\mathcal{O}(n^2\cdot \pi(d))$ as a shortest path problem on a directed acyclic graph $\mathcal{G} = (V,A)$, where $V = \{0,\ldots,n+1\}$ and $A=\{(i,j): 0 \leq i < j \leq n+1\}$, with arc costs $\ell_{ij}$ given as 
\begin{align}
\ell_{ij} = 
    \begin{cases}
         0, &  i = 0, \ 1 \leq j \leq n+1\\
        c_i - \frac{1}{4} \boldsymbol{a}^\top \bm{\Lambda}_{[i \to j]}\boldsymbol{a},  &  1\leq i < j \leq n+1,
    \end{cases} \label{SPP_arc}
\end{align}
where $\pi(d)$ indicates the computational complexity of  multiplication of two $d \times d$ matrices and inversion of a $d \times d$ matrix.
\end{proposition}

\noindent The optimal solution $(\bm{x}^*, \bm{z}^*, \tau^*)$ of \eqref{SDP} can be recovered from the optimal path. Letting $\mathcal{P}$ be the set of arcs $(i,j) \in A$ on the path, $z_i^* = 1$ if and only if the vertex $i$ is visited by the path, and $\bm{x}^* = - \frac{1}{2} \left(\sum_{(i,j) \in \mathcal{P}} \bm{\Lambda}_{[i \to j]}\right) \bm{a}$.
We refer to this solution method as the \textit{shortest path problem (SPP) algorithm}. 

\begin{proof}[Proof of Proposition~\ref{prop:DP}]
The objective value \eqref{linear_obj} at $(\boldsymbol{z}, \boldsymbol{W})$ can be computed as follows: 
\begin{subequations}
    \begin{align}
    -\frac{1}{4} \boldsymbol{a}^\top \bm{W}^{-1} \boldsymbol{a} + \boldsymbol{c}^\top \boldsymbol{z} & = -\frac{1}{4} \boldsymbol{a}^\top \left( \sum_{i=1}^{n} \sum_{j=i+1}^{n+1} \bm{\Lambda}_{[i \to j]} w_{ij}\right) \boldsymbol{a} + \sum_{i=1}^n c_i z_i \label{spp_eq1} \\
    & = \sum_{i=1}^{n} \sum_{j=i+1}^{n+1} \left( c_i -\frac{1}{4} \boldsymbol{a}^\top \bm{\Lambda}_{[i \to j]} \boldsymbol{a} \right)  w_{ij} \label{spp_eq2} \\
    & = \sum_{(i,j) \in A} \ell_{ij} w_{ij} \label{spp_eq3}
\end{align} \label{spp_proof}%
\end{subequations}
where \eqref{spp_eq1} holds due to \eqref{inv_decomp_dd}, and \eqref{spp_eq2} holds as $z_i = \sum_{j=0}^{i-1} w_{ji} = \sum_{j=i+1}^{n+1} w_{ij}$, $ i \in [n]$. The last equation \eqref{spp_eq3} follows from the construction of $\ell_{ij}$ specified in \eqref{SPP_arc}. Therefore, an optimal solution of \eqref{linear_form} can be obtained by solving a shortest path problem on $G$ with costs given in \eqref{SPP_arc}.
The computation of the costs involves $\mathcal{O}(n^2)$ $d \times d$ matrix multiplications and inversions.
Once the costs are computed, the shortest path problem on the directed acyclic graph $G$ can be solved in $\mathcal{O}(|V|+|A|) = \mathcal{O}(n^2)$ utilizing the topological order \cite{cormen2001single}.
\end{proof}

\begin{cor} \label{cor:DP-singleton}
The problem \eqref{MIQP-block} with a factorizable matrix $\boldsymbol{Q} \in \R^n$ can be solved in $\mathcal{O}(n^2)$ as a shortest path problem on $G$ with the arc cost $\ell_{ij}$, $ (i,j) \in A$, assigned as
\begin{align}
\ell_{ij} = 
    \begin{cases}
        0, &  i = 0, \ 1 \leq j \leq n+1\\
         c_i - \frac{u_j \left(a_i - \frac{u_i}{u_j}a_j\right)^2}{4 u_i (u_j v_i - u_i v_j )} ,  &  1\leq i < j \leq n\\[3pt]
         c_i - \frac{a_i^2}{4 u_i v_i} ,  &  1\leq i \leq n, \ j = n+1.
    \end{cases} \label{SPP_arc_singleton}
\end{align}
\end{cor}

\subsection{SOCP formulation} \label{sec:socp}
The shortest path algorithm discussed in \S\ref{sec:dynamicProgramming} cannot be extended naturally to problems with additional constraints. To handle such problems, branch-and-bound algorithms based on the relaxations derived in \S\ref{sec:convexification} may be utilized along with the additional constraints. While relaxation \eqref{convX_block_sdp} may be perceived as practically challenging due to the positive semi-definite constraint, we point out that matrix $\bm{W}$ in \eqref{inv_decomp_dd} is constrained to be a non-negative sum of given positive semidefinite matrices. As such, a result from Nesterov and Nemirovskii \cite[pp.~227--229]{nesterov1994interior} allows us to reformulate the problem as a conic quadratic optimization problem involving matrices $\bm{\Phi}_{[i\to j]}$ such that $\bm{\Lambda}_{[i\to j]}=\bm{\Phi}_{[i\to j]}\bm{\Phi}_{[i\to j]}^\top$. 
From the construction of matrices $\bm{\Lambda}_{[i \to j]}$ in Definition~\ref{def:consecutive_dd}, it holds 
$$\bm{\Phi}_{[i \to j]} = 
\begin{cases}
\left(\bm{E}_i - \bm{E}_j \bm{U}_j^{-\top} \bm{U}_i^\top \right) \left( \bm{U}_i \bm{V}_i^\top - \bm{U}_i \bm{U}_j^{-1} \bm{V}_j \bm{U}_i^\top \right)^{-1/2} & \text{for } 1 \leq i < j \leq n, \\ 
 \bm{E}_i \left(\bm{U}_i \bm{V}_i^\top \right)^{-1/2} & \text{for } 1 \leq i \leq n, \ j = n+1.
\end{cases}
$$.
\begin{proposition} \label{cor:socp-block}
The closed convex hull of $X_{\bm{Q}}^B$ can be written as follows:
\begin{align*}
    \cl \conv \left( X_{\bm{Q}}^B\right) = \Bigg\{&(\bm{x},\bm{z},\tau) \in \R^{(d+1)n+1}: \exists  w_{ij}, \tau_{ij} \in \R_+ \text{ and } \bm{h}_{ij} \in \R^{d} \text{ for } \ 0 \leq i < j \leq n+1, \text{ such that } \\
    & \    \eqref{spp_flow_1d}, \eqref{spp_z_1d}, \ \tau \geq \sum_{i=1}^n \sum_{j=i+1}^{n+1} \tau_{ij}, \ \sum_{i=1}^n \sum_{j=i+1}^{n+1}  \bm{\phi}_{[i \to j]}\boldsymbol{h}_{ij} = \boldsymbol{x}, \  \left\|\boldsymbol{h}_{ij}\right\|_2^2 \leq \tau_{ij} w_{ij}, \  1 \leq i < j \leq n+1\Bigg\} \cdot  
\end{align*}
\end{proposition}

\noindent Observe that the formulation in Proposition~\ref{cor:socp-block} does not involve matrix variables $\bm{W}\in \R^{dn\times dn}$, but includes $\mathcal{O}(n^2d)$ additional variables, $\mathcal{O}(dn)$ linear constraints, and $\mathcal{O}(n^2)$ $d$-dimensional conic quadratic constraints. 

Specializing the results to the factorable case
with similar notations $\bm{\phi}_{i \to j}$ such that $\bm{\Lambda}_{i \to j} = \bm{\phi}_{i \to j} \bm{\phi}_{i \to j}^\top$, or equivalently from Definition~\ref{def:consecutive1d}, 

$$\bm{\phi}_{i \to j} = 
\begin{cases}
\sqrt{\frac{u_j}{u_i(u_j v_i - u_i v_j)}}\left( \bm{e}_i - \frac{u_i}{u_j}\bm{e}_j\right) & \text{for } 1 \leq i < j \leq n \\
 \sqrt{\frac{1}{u_i v_i}} \ \bm{e}_i & \text{for } 1 \leq i \leq n, \ j = n+1,
\end{cases}
$$
we arrive at the following corollary.

\begin{cor} \label{cor:socp-singleton}
The closure of the convex hull of $X_{\bm{Q}}$ can be written as follows: 
\begin{align*}
    \cl \conv \left( X_{\boldsymbol{Q}}\right)=\Bigg\{ &(\bm{x},\bm{z},\tau) \in \R^{2n+1}: \exists  w_{ij}, \tau_{ij} \in \R_+ \text{ and } h_{ij} \in \R \text{ for }  0 \leq i < j \leq n+1, \text{ such that } \\
    & \quad \eqref{spp_flow_1d}, \eqref{spp_z_1d} , \ \tau \geq \sum_{i=1}^n \sum_{j=i+1}^{n+1} \tau_{ij}, \ \sum_{i=1}^n \sum_{j=i+1}^{n+1} \boldsymbol{\phi}_{i\to j} h_{ij} = \boldsymbol{x}, \  h_{ij}^2 \leq \tau_{ij} w_{ij}, \  1 \leq i < j \leq n+1\Bigg\}.
\end{align*}
\end{cor}

\section{Case study 1: Deconvolution of calcium imaging data}  \label{sec:calcium}

We consider the calcium imaging data deconvolution problem in statistical learning to demonstrate the real-world applicability of our study.
Calcium imaging is a widely used technique to indirectly measure neural activities. 
Deconvolution of calcium imaging data requires estimating the true spike events from the observed calcium signals, which is particularly challenging as the recorded signals are noisy.

For our experiment, we adopt the generative model introduced in \cite{jewell2018exact}, which expresses the observed fluorescence trace $r_i$ as a noisy version of the unobserved calcium concentration $s_i$. In this model, the calcium concentration $s_i$ decays exponentially over time until a spike $x_i$ occurs, instantaneously increasing the concentration. With a decay rate $\alpha \in (0,1)$ and Gaussian noise $\epsilon_i \overset{{\text{\tiny i.i.d.}}}{\sim} N (0, \sigma^2)$, $ i \in [n]$, the generative model is expressed by
\begin{equation}
    \begin{aligned}
        r_i & = \rho_0 + \rho_1 s_i + \epsilon_i, && i \in [n+1]\\
        s_{i+1} &= \alpha s_{i} + x_{i}, && i \in [n]
    \end{aligned}\label{calcium_generative_model}
\end{equation}
where $x_i \geq 0$, $ i \in [n]$. As the spike detection is scale-invariant, we assume $\rho_1 = 1$, and $\rho_0$ can be set to $0$ with a minor modification. Then, the calcium concentration can be estimated by solving the following optimization problem: 
\begin{subequations}
    \begin{align}
        \hspace{4cm} \min \ \ & \frac{1}{2} \sum_{i=1}^{n+1} \left( s_i - r_i \right)^2 + \lambda \sum_{i=1}^{n} z_i  \label{calcium_obj}\\
        \text{s.t. } \ & x_{i} = s_{i+1} - \alpha s_{i} \geq 0, &&  i \in [n] \label{calcium_linear}\\
        & x_i (1-z_i) = 0, &&  i \in [n] \hspace{3cm}\label{calcium_indicator}\\
        & z_i \in \{0,1\}, &&  i \in [n] \label{calcium_dim}
    \end{align} \label{calcium_nonnegative}%
\end{subequations}
Polynomial-time DP algorithms have been proposed for \eqref{calcium_nonnegative} and its relaxation without the nonnegativity constraint on $\boldsymbol{x}$ \cite{jewell2018exact, jewell2020fast}.
These DP algorithms cannot be generalized to problems that integrate the estimation problem \eqref{calcium_nonnegative} as a substructure. 

Projecting out $s_i$ using the equation in \eqref{calcium_linear}, $\forall i \in [n+1]$, and converting the objective function into vector form, \eqref{calcium_nonnegative} can be reformulated as
\begin{equation}
    \begin{aligned}
        \min \ \ & \boldsymbol{x}^\top \boldsymbol{Q} \boldsymbol{x} + \boldsymbol{a}^\top \boldsymbol{x} + \lambda \sum_{i=1}^n z_i \\
        \text{ s.t. } \ &\boldsymbol{x} \geq 0, \ \boldsymbol{x} \circ (\boldsymbol{e} - \boldsymbol{z}) = 0, \ \boldsymbol{z} \in \{0,1\}^n
    \end{aligned} \label{calcium_vec}
\end{equation}
with a factorizable matrix $\boldsymbol{Q}$. 
The details of constructing model \eqref{calcium_vec} are provided in Appendix \ref{Appendix:calcium}.

\subsection{Numerical results}

We test three variations of the problem \eqref{calcium_vec}: i) a relaxed problem without the nonnegativity constraint on $\boldsymbol{x}$, ii) the original problem, and iii) a capacity-constrained version. The test data is generated using  the model \eqref{calcium_generative_model}, where $x_i \overset{{\text{\tiny i.i.d.}}}{\sim} \text{Pois}(\mu)$. We randomly generate ten instances for each combination of $(n, \mu, \sigma)$, where $n \in \{50, 100, 150, 200, 250, 300\}$, $\mu \in \{0.01, 0.02, 0.03, 0.04, 0.05\}$ and $\sigma \in \{0.05,0.1,0.15,0.2,0.25\}$. We compare the performance of solving the MIQP model \eqref{calcium_nonnegative} with solving the proposed SOCP reformulations and SPP algorithm (only for the relaxed version) on the calcium data deconvolution problems. 
The experiments were conducted on a 2.3-GHz Intel Xeon 12-core Haswell processor with 64 GB of memory, using Python 3.11 and the Gurobi 10.0 solver with a 30-minute time limit. 

Table~\ref{table:calcium_var_num_all} reports the number of continuous and binary variables in MIQP and (MI)SOCP models, as well as the number of second-order cone constraints in the (MI)SOCP model. The model size of MIQP is linear in $n$, whereas the size of (MI)SOCP is quadratic in $n$. As the Gurobi solver struggles to handle models involving many conic constraints, the (MI)SOCP model can sometimes suffer from numerical errors, cutting off feasible solutions. 

\begin{table}[htp]
\footnotesize
\centering
\caption{Model size of MIQP and (MI)SOCP for calcium deconvolution.}
\makebox[\textwidth]{\setlength{\tabcolsep}{5pt}
\begin{tabular}{c|cc|ccc}
\hline \hline
 & \multicolumn{2}{c|}{MIQP}                                                                                                                      & \multicolumn{3}{c}{(MI)SOCP}                                                                                                                      \\ \hline
   n                & \multicolumn{1}{c}{\# cont.} & \multicolumn{1}{c|}{\# binary}  & \multicolumn{1}{c}{\# cont.} & \multicolumn{1}{c}{\# indicator $\dagger$} & \multicolumn{1}{c}{\# conic}  \\ \hline
50                 & \multicolumn{1}{c}{101}           & \multicolumn{1}{c|}{50}          &  \multicolumn{1}{c}{4,131}        & \multicolumn{1}{c}{50} & \multicolumn{1}{c}{1,326}            \\ 
100                & \multicolumn{1}{c}{200}          & \multicolumn{1}{c|}{100}           & \multicolumn{1}{c}{15,756}        & \multicolumn{1}{c}{100}    & \multicolumn{1}{c}{5,151}        \\ 
150                & \multicolumn{1}{c}{300}          & \multicolumn{1}{c|}{150}         & \multicolumn{1}{c}{34,881}        & \multicolumn{1}{c}{150}     & \multicolumn{1}{c}{11,476}     \\ 
200                & \multicolumn{1}{c}{400}          & \multicolumn{1}{c|}{200}         & \multicolumn{1}{c}{61,506}       & \multicolumn{1}{c}{200}     & \multicolumn{1}{c}{20,301}       \\ 
250                & \multicolumn{1}{c}{500}          & \multicolumn{1}{c|}{250}      & \multicolumn{1}{c}{95,631}       & \multicolumn{1}{c}{250}    & \multicolumn{1}{c}{31,626}        \\ 
300                & \multicolumn{1}{c}{600}          & \multicolumn{1}{c|}{300}          & \multicolumn{1}{c}{137,256}       & \multicolumn{1}{c}{300}   & \multicolumn{1}{c}{45,451}        \\ \hline \hline
\end{tabular}
}
\vskip 2mm
($\dagger$) The indicator $\bm{z}$ can be relaxed to a continuous variable in the SOCP model of the relaxed calcium deconvolution problem while remaining binary for the MISOCP models of the original and capacity-constrained versions. 
\hfill
\captionsetup{justification=centering}
\label{table:calcium_var_num_all}
\end{table}

\paragraph{Relaxed calcium data deconvolution problem.} For the relaxed version of \eqref{calcium_nonnegative}, where the spike $\boldsymbol{x}$ is not restricted in sign, the SOCP representation of $\cl \conv (X_{\bm{Q}})$ is tight. Moreover, the SPP algorithm proposed in Corollary~\ref{cor:DP-singleton} can be applied to solve the problem directly. Therefore, we compare the performance of solving the MIQP model \eqref{calcium_nonnegative} without $\bm{x} \geq 0$, its exact SOCP reformulation, and solving the problem with the SPP algorithm.

The boxplots of the overall computational time for these three approaches are given in Figure~\ref{fig:default_stime_all}. The SPP algorithm solves all instances quickly, and the SOCP model also performs well. However, the MIQP model exhibits significant variability in computational time, with timeouts occurring for some instances.
The performance gap between the different approaches becomes clearer when they are compared as a function of $n$ as in Table~\ref{table:calcium_default_avg}. Notably, the root relaxation gap of the MIQP model is 100\% for every instance, while the SOCP formulation is tight. The computational time of the MIQP model increases rapidly with $n$, while the SOCP model and SPP algorithm scale very well as $n$ increases. Moreover, the number of branch-and-bound (B$\And$B) nodes explored to solve MIQP increases fast for large $n$. While the MIQP model often has a substantial final gap after the time limit, all returned solutions are optimal. The SOCP model performs well for most instances, though it reported a suboptimal solution with a large error rate (14.7\%) for one instance with $n=300$, $\mu=0.04$, $\sigma = 0.15$, and five instances with small error rates less than 1\%. Here, the error rate is the percent gap between the returned objective value and the actual optimal value. Interestingly, the computational time of the MIQP tends to increase as $\mu$ decreases and $\sigma$ increases. In contrast, the performance of the SOCP shows no remarkable change as a function of these parameters.

\begin{figure}[htp]
    \centering
    \includegraphics[height=2.0in]{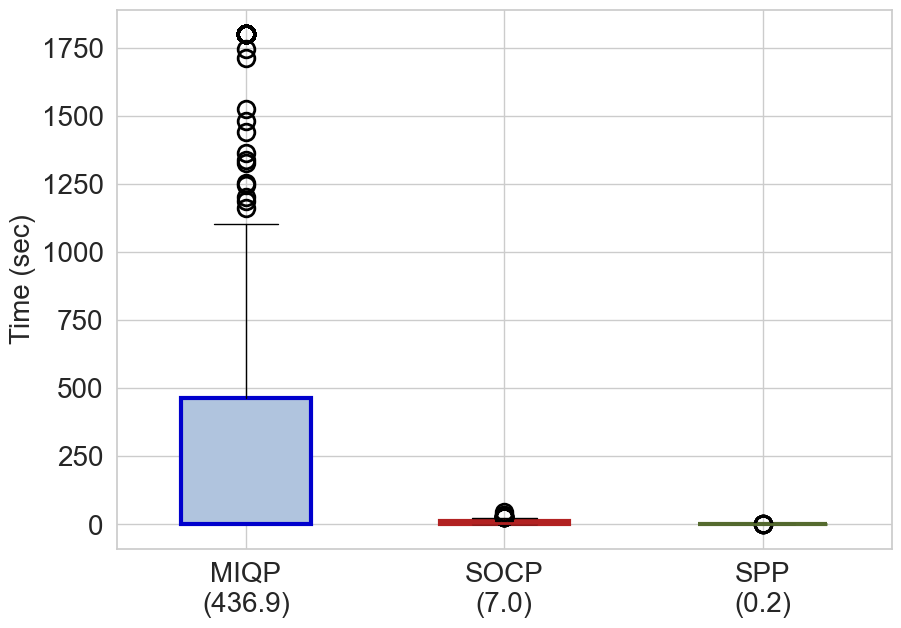}
    \caption{Computational time of MIQP, SOCP, and SPP on relaxed calcium deconvolution.}
    \label{fig:default_stime_all}
\end{figure}

\begin{table}[htp]
\vspace{-0.1cm}
\footnotesize
\centering
\caption{Model Performance on relaxed calcium deconvolution as a function of $n$.}
\setlength\tabcolsep{4pt}
\begin{tabular}{c|ccc|cccc|cc}
\hline \hline
            & \multicolumn{3}{c|}{Solve time (sec)}                          & \multicolumn{4}{c|}{MIQP}                                                                                    & \multicolumn{2}{c}{SOCP}                    \\ \hline 
 n & \multicolumn{1}{c}{MIQP}   & \multicolumn{1}{c}{SOCP} & SPP & \multicolumn{1}{c}{Root gap (\%)} & \multicolumn{1}{c}{Final gap (\%)} & \multicolumn{1}{c}{Optimal (\%)}   & \# B$\And$B Nodes & \multicolumn{1}{c}{Optimal (\%)}     & Err. (\%) \\ \hline
\multicolumn{1}{c|}{50}                                                                      & \multicolumn{1}{c}{{0.3}}    & \multicolumn{1}{c}{{0.2}}  & {0.0}  & \multicolumn{1}{c}{{100}}     & \multicolumn{1}{c}{{0.0}}       & \multicolumn{1}{c}{{100}} & {160}         & \multicolumn{1}{c}{{100}}   & {0.0}           \\ 
\multicolumn{1}{c|}{{100}}                                                                      & \multicolumn{1}{c}{{3.3}}    & \multicolumn{1}{c}{{1.2}}  & {0.0}  & \multicolumn{1}{c}{{100}}     & \multicolumn{1}{c}{{0.0}}       & \multicolumn{1}{c}{{100}} &{56,503}       & \multicolumn{1}{c}{{100}}   & {0.0}           \\ 
\multicolumn{1}{c|}{{150}}                                                             & \multicolumn{1}{c}{{129.3}}  & \multicolumn{1}{c}{{5.1}}  & {0.1} & \multicolumn{1}{c}{{100}} & \multicolumn{1}{c}{{1.2}}  & \multicolumn{1}{c}{{100}} & {1,920,016}  & \multicolumn{1}{c}{{99.6}} & {0.1}  \\ 
\multicolumn{1}{c|}{{200}}                                                             & \multicolumn{1}{c}{{502.6}}  & \multicolumn{1}{c}{{8.4}}  & {0.2} & \multicolumn{1}{c}{{100}} & \multicolumn{1}{c}{{5.9}}  & \multicolumn{1}{c}{{100}} & {7,934,164}  & \multicolumn{1}{c}{{100}}   & {0.0}  \\ 
\multicolumn{1}{c|}{{250}}                                                             & \multicolumn{1}{c}{{839.4}}  & \multicolumn{1}{c}{{11.2}} & {0.2} & \multicolumn{1}{c}{{100}} & \multicolumn{1}{c}{{13.7}} & \multicolumn{1}{c}{{100}} & {12,030,090} & \multicolumn{1}{c}{{98.8}} & {0.3}  \\ 
\multicolumn{1}{c|}{{300}}                                                             & \multicolumn{1}{c}{{1,146.7}} & \multicolumn{1}{c}{{15.7}} & {0.3} & \multicolumn{1}{c}{{100}} & \multicolumn{1}{c}{{23.3}} & \multicolumn{1}{c}{{100}} & {14,441,887} & \multicolumn{1}{c}{{99.2}} & {7.45}  \\ \hline 
\multicolumn{1}{c|}{\textbf{Avg.}}                                                     & \multicolumn{1}{c}{\textbf{436.9}}  & \multicolumn{1}{c}{\textbf{7.0}}  & \textbf{0.1} & \multicolumn{1}{c}{\textbf{100}} & \multicolumn{1}{c}{\textbf{7.4}}  & \multicolumn{1}{c}{\textbf{100}} & \textbf{6,063,803}  & \multicolumn{1}{c}{\textbf{99.6}} & \textbf{2.65}  \\ \hline \hline
\end{tabular}

\captionsetup{justification=centering}
\label{table:calcium_default_avg}
\end{table}

\paragraph{Original calcium data deconvolution problem.} For the original problem \eqref{calcium_nonnegative}, the SOCP reformulation does not yield a convex hull representation of the $X_{\bm{Q}}$, and the SPP algorithm is not applicable due to the nonnegativity constraint on $\bm{x}$. Therefore, we compare the performance of MIQP and MISOCP models. 

The results on the performance of the two models are reported in Table~\ref{table:calcium_nonnegative_avg}. Although the MISOCP takes a slightly longer computational time than the MIQP for $n \leq 100$, the computational time of MIQP rapidly increases for larger $n$, and several instances cannot be solved to optimality within the time limit. In contrast, the MISOCP model is solved relatively quickly, even for large $n$ values. Observe that the continuous relaxation of MISOCP is strong, even with the nonnegativity constraint on $\bm{x}$, whereas the MIQP model exhibits a large relaxation gap in most cases. Consequently, millions of B\&B nodes are required for the MIQP model, whereas MISOCP is solved within only 1-2 nodes. 
Notice that the relaxation gap of the MIQP is significantly influenced by the spike frequency parameter $\mu$. As $\mu$ increases, the relaxation gap of the MIQP grows substantially while the MISOCP maintains a small gap. Consequently, the computational time of the MIQP rapidly increases with $\mu$, whereas the MISOCP model can be solved fast, even for large $\mu$ values. No specific trends were observed in the performance of either model with respect to the standard deviation $\sigma$ of noise.

\begin{table}[htp]
\vspace{-0.2cm}
\footnotesize
\centering
\caption{Model performance on the original calcium deconvolution.}
\makebox[\textwidth]{\setlength\tabcolsep{4pt}
\begin{tabular}{cc|cc|cc|cc|cc|cc|c}
\hline \hline
\multicolumn{2}{c|}{\multirow{2}{*}{Instance}} & \multicolumn{2}{c|}{Solve time (sec)}                & \multicolumn{2}{c|}{Root gap (\%)}          & \multicolumn{2}{c|}{Final gap (\%)}               & \multicolumn{2}{c|}{Optimal (\%)}                                 & \multicolumn{2}{c|}{\# B$\And$B nodes}                          & Err. (\%) $\dagger$   \\ \cline{3-13}
&  & \multicolumn{1}{c}{miqp}           & misocp         & \multicolumn{1}{c}{miqp}          & misocp       & \multicolumn{1}{c}{miqp}          & misocp       & \multicolumn{1}{c}{miqp}             & misocp           & \multicolumn{1}{c}{miqp}               & misocp       & misocp       \\ \hline
\multicolumn{1}{c|}{\multirow{6}{*}{n}} &\multicolumn{1}{c|}{{50}}                                                             & \multicolumn{1}{c}{{0.3}}   & {0.6}   & \multicolumn{1}{c}{{57.5}} & {0.0} & \multicolumn{1}{c}{{0.0}}  & {0.0} & \multicolumn{1}{c}{{100}}   & {100}   & \multicolumn{1}{c}{{83}}      & {1.0} & {0.0} \\ 
\multicolumn{1}{c|}{}&\multicolumn{1}{c|}{{100}}                                                             & \multicolumn{1}{c}{{3.7}}   & {5.0}   & \multicolumn{1}{c}{{62.3}} & {0.0} & \multicolumn{1}{c}{{0.0}}  & {0.0} & \multicolumn{1}{c}{{100}}   & {99.2} & \multicolumn{1}{c}{{39,818}}   & {1.0} & {0.0 (1)} \\ 
\multicolumn{1}{c|}{}&\multicolumn{1}{c|}{{150}}                                                             & \multicolumn{1}{c}{{94.2}}  & {15.9}  & \multicolumn{1}{c}{{68.5}} & {0.0} & \multicolumn{1}{c}{{0.7}}  & {0.0} & \multicolumn{1}{c}{{99.6}} & {99.6} & \multicolumn{1}{c}{{886,573}}  & {1.0} & {0.0} \\ 
\multicolumn{1}{c|}{}&\multicolumn{1}{c|}{{200}}                                                             & \multicolumn{1}{c}{{321.2}} & {35.4}  & \multicolumn{1}{c}{{67.5}} & {0.0} & \multicolumn{1}{c}{{3.3}}  & {0.0} & \multicolumn{1}{c}{{98.8}} & {100}   & \multicolumn{1}{c}{{3,198,009}} & {1.2} & {0.0} \\ 
\multicolumn{1}{c|}{}&\multicolumn{1}{c|}{{250}}                                                             & \multicolumn{1}{c}{{564.8}} & {94.0}  & \multicolumn{1}{c}{{69.2}} & {0.0} & \multicolumn{1}{c}{{7.0}}  & {0.0} & \multicolumn{1}{c}{{97.2}} & {98.0}  & \multicolumn{1}{c}{{4,866,083}} & {1.5} & {0.03 (2)} \\ 
\multicolumn{1}{c|}{}&\multicolumn{1}{c|}{{300}}                                                             & \multicolumn{1}{c}{{811.9}} & {160.4} & \multicolumn{1}{c}{{69.1}} & {0.0} & \multicolumn{1}{c}{{12.3}} & {0.0} & \multicolumn{1}{c}{{96.8}} & {96.0}  & \multicolumn{1}{c}{{7,544,878}} & {1.6} & {0.05 (7)} \\ \hline 

\multicolumn{1}{c|}{\multirow{6}{*}{$\mu$}}  & \multicolumn{1}{c|}{{0.01}}  & 1.9	& 16.2	&41.6	&0.0	&0.0	&0.0	&100	&99.7	&8,268	&1.0	&0.0 (0) \\
\multicolumn{1}{c|}{}& \multicolumn{1}{c|}{{0.02}}  &27.3	&23.0	&59.8	&0.0	&0.2	&0.0	&100	&99.3	&205,229	&1.1	&0.16 (1)\\
\multicolumn{1}{c|}{}& \multicolumn{1}{c|}{{0.03}}  &262.3	&48.4	&70.1	&0.0	&2.6	&0.0	&99.0	&98.7	&2,098,894	&1.1	&0.03 (3) \\
\multicolumn{1}{c|}{}& \multicolumn{1}{c|}{{0.04}}  &463.9	&83.8	&75.6	&0.0	&5.5	&0.0	&99.0	&98.3	&3,800,353	&1.3	&0.03 (3) \\
\multicolumn{1}{c|}{}& \multicolumn{1}{c|}{{0.05}}  &741.5	&88.0	&81.2	&0.1	&11.2	&0.0	&95.6	&98.0	&7,666,793	&1.6	&0.01 (3) \\ \hline
\multicolumn{2}{c|}{\textbf{Avg.}}                                                     & \multicolumn{1}{c}{\textbf{299.4}} & \textbf{51.9}  & \multicolumn{1}{c}{\textbf{65.7}} & \textbf{0.0} & \multicolumn{1}{c}{\textbf{3.9}}  & \textbf{0.0} & \multicolumn{1}{c}{\textbf{98.7}} & \textbf{98.8} & \multicolumn{1}{c}{\textbf{2,755,907}} & \textbf{1.2} & \textbf{0.03 (10)} \\ \hline \hline
\end{tabular}

}
\vskip 2mm
($\dagger$) Average of instances with suboptimal solutions only. The number of instances reported infeasible given in parathesis. \hfill
\captionsetup{justification=centering}
\label{table:calcium_nonnegative_avg}
\end{table}

Despite its advantages, the MISOCP model encounters some numerical errors when solved with the Gurobi solver. Out of 1,500 instances, 8 instances exhibited a small error rate of 0.03\% on average and 10 instances were reported as infeasible. Increasing the numeric focus of Gurobi helped in these cases. Among the 10 instances reported infeasible, 7 were solved to optimality, 1 reported a solution with 5.67\% error, and 2 instances were still reported as infeasible with Gurobi numeric focus parameter set to 3. 

\paragraph{Capacity-constrained calcium deconvolution problem}
We extend the tests to MIQP and MISOCP formulations with an additional capacity constraint $\bm{g}^\top \bm{z} \leq h$. The weight $\bm{g}$ and capacity $h$ are randomly generated as $g_i \overset{{\text{\tiny i.i.d.}}}{\sim} \text{unif}\{1,5\}$ and $h = \frac{1}{2} \sum_{i \in [n]} g_i$. This constraint is relevant in statistical learning when there is prior knowledge of the activity likelihood. 
For example, in calcium concentration data collected over a day, brain activity is likely higher during daytime work hours and significantly lower during sleep. Another scenario involves tasks with multiple steps imposing varying cognitive loads. In such cases, different weights are assigned to the indicators to reflect the step-dependent activity levels.

The performance of the MIQP and MISOCP models for each combination of $n$ and $\mu$ is summarized in Table~\ref{table:calcium_knapsack_avg}. 
The computational time patterns observed are similar to those for
the original problem without the capacity constraint. Specifically, the computational time for the MIQP model increases rapidly for $n \geq 150$, while MISOCP remains tractable even for large $n$. The root relaxation gap and the number of B$\And$B nodes provide insights into computational time differences. While the MIQP model generally has a large relaxation gap, requiring a substantial number of B\&B nodes to solve. In contrast, the MISOCP model exhibits a tight relaxation, resulting in only a few B$\And$B nodes to explore.
Although the MIQP model fails to prove optimality for instances with large $n$, it finds an optimal solution within the time limit for every instance. On the other hand, the MISOCP model outperforms the MIQP model in most cases, but encountered numerical errors in 23 out of 1,500 instances. Most of them had a small error rate of 1.13\% on average, though two instances were reported as infeasible, cutting off all feasible solutions. 
The spike frequency $\mu$ significantly affects the performance of the MIQP model. As $\mu$ increases, the relaxation of the MIQP model weakens, leading to a noticeable increase in computational time.
The MISOCP model is far less sensitive to changes in $\mu$.
Neither model exhibited clear trends in computational time or relaxation gap as a function of the standard deviation of noise $\sigma$.
However, the relaxation gap of MIQP tended to decrease as $n$ increased.
These results suggest that the MISOCP model is more efficient when incorporating additional constraints, making it a practical choice for multi-period MIQPs arising in diverse domains. 

\begin{table}[htp]
\vspace{-0.2cm}
\footnotesize
\centering
\caption{Model performance on capacity-constrained calcium deconvolution.}
\makebox[\textwidth]{\setlength\tabcolsep{4pt}
\begin{tabular}{cc|cc|cc|cc|cc|cc|c}
\hline \hline
\multicolumn{2}{c|}{\multirow{2}{*}{Instance}}  & \multicolumn{2}{c|}{Solve time (sec)}                & \multicolumn{2}{c|}{Root gap (\%)}          & \multicolumn{2}{c|}{Final gap (\%)}               & \multicolumn{2}{c|}{Optimal (\%)}                                 & \multicolumn{2}{c|}{\# B$\And$B nodes}                          & Err. (\%) $\dagger$  \\ \cline{3-13}
\multicolumn{2}{c|}{} & \multicolumn{1}{c}{miqp}           & misocp         & \multicolumn{1}{c}{miqp}          & misocp       & \multicolumn{1}{c}{miqp}          & misocp       & \multicolumn{1}{c}{miqp}             & misocp           & \multicolumn{1}{c}{miqp}               & misocp       & misocp       \\ \hline
\multicolumn{1}{c|}{\multirow{6}{*}{n}} & 50            & \multicolumn{1}{c}{0.0}            & 0.2           & \multicolumn{1}{c}{47.5}          & 0.0          & \multicolumn{1}{c}{0.0}          & 0.0          & \multicolumn{1}{c}{100}          & 100           & \multicolumn{1}{c}{27}                 & 1.0          & 0.0           \\
\multicolumn{1}{c|}{} &100           & \multicolumn{1}{c}{0.1}            & 1.1           & \multicolumn{1}{c}{61.8}          & 0.0          & \multicolumn{1}{c}{0.0}          & 0.0          & \multicolumn{1}{c}{100}          & 100           & \multicolumn{1}{c}{3,287}              & 1.0          & 0.0           \\ 
\multicolumn{1}{c|}{} &150           & \multicolumn{1}{c}{2.9}            & 3.1           & \multicolumn{1}{c}{64.2}          & 0.0          & \multicolumn{1}{c}{0.0}          & 0.0          & \multicolumn{1}{c}{100}          & 100           & \multicolumn{1}{c}{149,475}            & 1.0          & 0.0           \\ 
\multicolumn{1}{c|}{} &200           & \multicolumn{1}{c}{131.9}          & 11.7          & \multicolumn{1}{c}{67.1}          & 0.0          & \multicolumn{1}{c}{0.7}          & 0.0          & \multicolumn{1}{c}{100}          & 98.8          & \multicolumn{1}{c}{5,112,404}          & 1.2          & 0.13          \\ 
\multicolumn{1}{c|}{} &250           & \multicolumn{1}{c}{429.0}          & 29.0          & \multicolumn{1}{c}{68.3}          & 0.0          & \multicolumn{1}{c}{4.4}          & 0.0          & \multicolumn{1}{c}{100}          & 97.6          & \multicolumn{1}{c}{14,159,775}         & 1.3          & 0.40 (1)      \\ 
\multicolumn{1}{c|}{} &300           & \multicolumn{1}{c}{548.7}          & 87.6          & \multicolumn{1}{c}{67.8}          & 0.0          & \multicolumn{1}{c}{7.6}          & 0.0          & \multicolumn{1}{c}{100}          & 94.4          & \multicolumn{1}{c}{17,678,772}         & 2.2          & 1.65 (1)      \\ \hline

\multicolumn{1}{c|}{\multirow{5}{*}{$\mu$}} & 0.01	&0.2	&5.1	&41.8	&0.0	&0.0	&0.0	&100	&100	&1,412	&1.0	&0.0\\
\multicolumn{1}{c|}{} &0.02	&2.0	&7.0	&53.5	&0.0	&0.0	&0.0	&100	&99.7	&50,246	&1.0	&0.16 \\
\multicolumn{1}{c|}{} &0.03	&116.9	&18.8	&67.5	&0.0	&0.8	&0.0	&100	&98.3	&3,526,656	&1.3	&0.89 \\
\multicolumn{1}{c|}{} &0.04	&298.9	&32.3	&74.2	&0.0	&2.7	&0.0	&100	&97.7	&9,705,110	&1.4	&0.87 \\
\multicolumn{1}{c|}{} &0.05	&509.3	&47.5	&76.9	&0.0	&7.0	&0.0	&100	&96.6	&17,636,359	&1.8	&1.95 (2) \\ \hline

\multicolumn{2}{c|}{\textbf{Avg.}} & \multicolumn{1}{c}{\textbf{185.4}} & \textbf{22.1} & \multicolumn{1}{c}{\textbf{62.8}} & \textbf{0.0} & \multicolumn{1}{c}{\textbf{2.1}} & \textbf{0.0} & \multicolumn{1}{c}{\textbf{100}} & \textbf{98.5} & \multicolumn{1}{c}{\textbf{6,183,957}} & \textbf{1.3} & \textbf{1.13} \\ \hline \hline
\end{tabular}

}\vskip 2mm
($\dagger$) Average of instances with suboptimal solutions only. The number of instances reported infeasible given in parathesis. \hfill
\captionsetup{justification=centering}
\label{table:calcium_knapsack_avg}
\end{table}

\section{Case study 2: Path-following problem in hybrid system control} \label{sec:path-following}

Hybrid system control is another domain where multi-period MIQP problems in the form of \eqref{MIQP-ddim} frequently arise in decision-making.
A notable example is the path-following problem, which aims to find an optimal control strategy that makes a system follow a predefined trajectory as closely as possible while satisfying system constraints.  For a general system with linear dynamics, the path-following problem can be formulated as follows:
\begin{equation}
    \begin{aligned}
        \min \ &\sum_{i=1}^{n+1} (\boldsymbol{s}_i - \boldsymbol{r}_i )^\top \boldsymbol{P}_i (\boldsymbol{s}_i - \boldsymbol{r}_i) + \sum_{i=1}^n \boldsymbol{y}_{[i]}^\top \boldsymbol{R}_i \boldsymbol{y}_{[i]} + \sum_{i=1}^n c_i z_i \\
        \text{s.t. } & \boldsymbol{s}_1 = \boldsymbol{b}_0\\
        & \boldsymbol{s}_{i+1} = \boldsymbol{A}_i \boldsymbol{s}_i + \boldsymbol{G}_i \boldsymbol{y}_{[i]} + \boldsymbol{k}_i z_i + \bm{b}_i, &&  i \in [n] \\
        & \boldsymbol{s}_i^{\min} \leq \boldsymbol{s}_i \leq \boldsymbol{s}_i^{\max}, &&  i \in [n+1]\\
        & \boldsymbol{y}_i^{\min} z_i \leq \boldsymbol{y}_{[i]} \leq \boldsymbol{y}_i^{\max} z_i, &&  i \in [n]\\
        & \boldsymbol{s}_{i+1} \in \bbR^{d_s}, \ \boldsymbol{y}_{[i]} \in \bbR^{d_y}, \ z_i \in \{0,1\}, &&  i \in [n].
    \end{aligned} \label{path-following}
\end{equation}
where the state of the system in period $i+1$, $s_{i+1}$, is an affine function of the control variable $\bm{y}_{[i]}$ and its indicator $\bm{z}_i$. Additional quadratic cost on the control variables $\bm{y}$ is also considered along with bound constraints on the state and control variables. $\boldsymbol{P}_i \text{ and } \boldsymbol{R}_i$ are positive definite for all $i$.

The energy management of power-split hybrid electric vehicles (PS-HEV) is a prominant application in hybrid control that often involves solving a path-following problem \eqref{path-following} with appropriately chosen parameters to determine optimal control strategies that balance the engine and battery usage \cite{borhan2011mpc}.
Given a reference vehicle cycle $V^{\text{ref}}$, the goal is to control the vehicle to follow the cycle as closely as possible, while optimizing energy efficiency. A model predictive control (MPC) approach \cite{borrelli2017predictive} is used to solve the discretized linear system control. At each period $t$, current state $\boldsymbol{s}_t$ is measured and the control decision for the next $n$ periods is determined by solving  \eqref{path-following}. In this formulation,  
the state $\boldsymbol{s}_i$, the control $\boldsymbol{y}_{[i]}$, and the indicator variables $z_i$ are defined as follows:
\begin{align}
    \boldsymbol{s}_i = \begin{bmatrix}
        {SoC}_{t+i} \\[0.2em] \dot{fc}_{t+i} 
    \end{bmatrix}, \quad \boldsymbol{y}_{[i]} = \begin{bmatrix}
        \omega^{\text{eng}}_{t+i} \\[0.4em] T^{\text{eng}}_{t+i} \\[0.3em] V_{t+i} - V_{t+i}^{\text{ref}}
    \end{bmatrix}, \quad z_i = \phi^{\text{eng}}_{t+i}, \qquad i \in [n]. \label{hev_vars}
\end{align}
Notice that the dimensions of the state and control variables do not need to be equal.
The state variable in this problem consists of the battery's state of charge (${SoC}$ )and the fuel consumption rate ($\dot{fc}$). The control variables are the engine speed ($\omega^{\text{eng}}$) and torque ($T^{\text{eng}}$), and the gap between the vehicle speed ($V$) and the reference speed ($V^{\text{ref}}$). The binary indicator $\phi^{\text{eng}}$ denotes whether the engine is on or off. 
Define auxiliary variables $\boldsymbol{x}_{[i]} = \boldsymbol{G} \boldsymbol{y}_{[i]} + \boldsymbol{k} z_i$ indicating the total control input at period $t+i$, $i \in [n]$. Then, by projecting out the state variables $\boldsymbol{s}_i$, the following formulation is obtained:
\begin{subequations}
    \begin{align}
        \min \ \ & \boldsymbol{x}\sz{dn}^\top \boldsymbol{Q}\sz{dn} \boldsymbol{x}\sz{dn} + \sum_{i=1}^{n} \boldsymbol{y}\sz{dn}_{[i]}^\top \boldsymbol{R}\sz{dn}_i \boldsymbol{y}\sz{dn}_{[i]} + \boldsymbol{a}\sz{dn}^\top \boldsymbol{x}\sz{dn} + \sum_{i=1}^{n} c_i z_i + v \label{hev_obj}\\
        \text{s.t. } \ & \boldsymbol{x}_{[i]} = \boldsymbol{G}_i \boldsymbol{y}_{[i]} + \boldsymbol{k}_i z_i, && i \in [n] \label{hev_input}\\
        &\boldsymbol{s}_{i}^{\min} \leq \left( \Pi_{t=1}^{i-1} \boldsymbol{A}_{i-t} \right) \boldsymbol{b}_0 + \sum_{\tau=1}^{i-1} \left( \Pi_{t=1}^{i -\tau-1} \boldsymbol{A}_{i-t} \right) \left( \boldsymbol{x}_{[\tau]} + \boldsymbol{b}_{\tau} \right) \leq \boldsymbol{s}_{i}^{\max}, && i \in [n+1] \label{hev_s_bd} \\
        & \boldsymbol{y}_{i}^{\min} z_{i} \leq \boldsymbol{y}_{[i]} \leq \boldsymbol{y}_{i}^{\max} z_{i}, && i \in [n] \label{hev_y_bd}\\
        &z_{i} \in \{0,1\}, && i \in [n] \label{hev_z_binary}
    \end{align} \label{hev_proj}%
\end{subequations}
with a block-factorizable positive definite matrix $\boldsymbol{Q}\sz{dn} \in \bbR^{dn \times dn}$ and $\boldsymbol{a}\sz{dn} \in \bbR^{dn}$ for $d=2$. Details of construction of $\boldsymbol{Q}\sz{dn}$ and $\boldsymbol{a}\sz{dn}$ are provided in Appendix \ref{Appendix:hev-block}. The same reformulation can be applied to path-following problems in various other domains.

\subsection{Numerical results}
In this experiment, we test the practical performance of the SOCP reformulation of the multi-period MIQP formulation \eqref{path-following} on generic instances of the energy management of PS-HEV with decision variables defined as in \eqref{hev_vars}.
The dataset consists of 10 randomly generated instances for each combination of $(n, \lambda)$, where $n \in \{10,20,\ldots,70\}$ and $\lambda \in \{2,4,6,8,10\}$ is the fixed indicator cost $c_i = \lambda$, $i \in [n]$.
We use fixed cost matrices in each instance constructed as $\boldsymbol{P}_i = \bm{P} = \frac{1}{2} \Tilde{\boldsymbol{P}} \Tilde{\boldsymbol{P}}^\top + \frac{1}{4} \boldsymbol{I}  \in \bbR^{2 \times 2} $ for ${\tilde P}_{ij} \overset{{\text{\tiny i.i.d.}}}{\sim} N(0,1)$ and $\boldsymbol{R}_i = \bm{R} = 0.1 \boldsymbol{I} \in \bbR^{3 \times 3} $ for all $i \in [n]$. The predefined trajectory $\boldsymbol{r}_i \in \bbR^{2}$ is generated as $r_{ij} \overset{{\text{\tiny i.i.d.}}}{\sim} \text{unif}[-2,2]$ for $i \in [n+1]$ and $j \in [2]$. The initial state $\boldsymbol{b}_0 \in \bbR^{2}$ is given as $b_{0j} \overset{{\text{\tiny i.i.d.}}}{\sim} \text{unif}[1,3]$, $j \in [2]$. The system dynamics is defined by a matrix $\boldsymbol{A} = \frac{1}{\max(|\text{eig}({\tilde A})|)}\Tilde{\boldsymbol{A}} \in \bbR^{2 \times 2}$ for ${\tilde A}_{ij} \overset{{\text{\tiny i.i.d.}}}{\sim} N(0,1)$, a full-row rank matrix $\boldsymbol{B} \in \bbR^{2 \times 3}$ with $B_{ij} \overset{{\text{\tiny i.i.d.}}}{\sim} N(0,1)$, and a vector $\boldsymbol{c} \in \bbR^{2}$ with $c_i \overset{{\text{\tiny i.i.d.}}}{\sim} \text{unif}[1,3]$. Fixed bounds of the state and control variables are used, which are given as $s_i^{\min} = -5$, $s_i^{\max} = 10$ for $i \in [2]$ and $y_j^{\min} = -2.3$, $y_j^{\max} = 2.3$ for $j \in [3]$. Each randomly generated parameter is rounded to one decimal place. The experimental setup follows the same procedure as in the calcium data deconvolution study in the previous section.

Three formulations are compared: i) the MIQP formulation \eqref{path-following} (MIQP), ii) the perspective reformulation \cite{frangioni2006perspective} of MIQP \eqref{path-following} (MIQP-p), where the quadratic cost on the control variable $\bm{y}_{[i]}^\top \bm{R}_i \bm{y}_{[i]}$ is replaced by its perspective function $\frac{1}{z_i}\bm{y}_{[i]}^\top \bm{R}_i \bm{y}_{[i]}$ for all $i \in [n]$, and iii) the MISOCP formulation given by
\begin{align*}
    \min \ & \tau + \sum_{i=1}^{n} \frac{1}{z_i}\boldsymbol{y}\sz{dn}_{[i]}^\top \boldsymbol{R}\sz{dn}_i \boldsymbol{y}\sz{dn}_{[i]} + \boldsymbol{a}\sz{dn}^\top \boldsymbol{x}\sz{dn} + \sum_{i=1}^{n} c_i z_i + v  \\
    \text{s.t. } & (\tau, \bm{x}, \bm{z}) \in \cl \conv \left( X_{\bm{Q}}^B\right) \\
    & \eqref{hev_input}, \eqref{hev_s_bd}, \eqref{hev_y_bd}
\end{align*}
for $\cl \conv \left( X_{\bm{Q}}^B\right)$ defined as in Proposition~\ref{cor:socp-block}.
The sizes of the three models are given in Table~\ref{table:block-factorizable_syntheic_size}. The sizes of MIQP and MIQP-p are $\mathcal{O}((d_s + d_y)n)$, whereas the MISOCP has $\mathcal{O}(d_s n^2 + d_y n)$ continuous variables and $\mathcal{O}(n^2)$ second-order cone constraints.

\begin{table}[H]
\footnotesize
\centering
\caption{Model size of MIQP, MIQP-p and MISOCP of path-following problem.}
\makebox[\textwidth]{\setlength\tabcolsep{4pt}
\begin{tabular}{c|c|ccc|cc}
\hline \hline
   & \# binary & \multicolumn{3}{c|}{\# cont.} & \multicolumn{2}{c}{\# conic} \\ \hline
n  & all       & miqp   & miqp-p   & misocp   & miqp-p        & misocp       \\ \hline
10 & 10        & 52     & 62       & 316      & 10            & 76           \\
20 & 20        & 102    & 122      & 1,026     & 20            & 251          \\
30 & 30        & 152    & 182      & 2,136     & 30            & 526          \\
40 & 40        & 202    & 242      & 3,646     & 40            & 901          \\
50 & 50        & 252    & 302      & 5,556     & 50            & 1,376         \\
60 & 60        & 302    & 362      & 7,866     & 60            & 1,951         \\
70 & 70        & 352    & 422      & 10,576    & 70            & 2,626        \\
\hline \hline
\end{tabular}}
\captionsetup{justification=centering}
\label{table:block-factorizable_syntheic_size}
\end{table}

The performance of the three models is summarized in Table~\ref{table:block-factorizable_syntheic_avg}. The computational times for the MIQP and MIQP-p models dramatically increase for $n \geq 50$, whereas the MISOCP model remains tractable even for large $n$. Although the perspective reformulation improves the average relaxation gap by $2.8-3.9\%$, it requires somewhat longer computation time and a higher number of B$\And$B nodes. Meanwhile, MISOCP yields a very small relaxation gap even for large $n$, resulting in short computational times and a limited number of B$\And$B nodes.
We further examine the relaxation gap and the computational time of each model as $\lambda$ changes. The relaxation gaps for the MIQP and MIQP-p show similar trends: they gradually increase until $\lambda \leq 6$ and decrease thereafter. Both models exhibit shorter computational times for large $\lambda$. The MISOCP model outperforms  MIQP and MIQP-p for all $\lambda$.
In Table~\ref{table:synthetic-subopt}, we detail the instances where the MISOCP model encountered numerical errors.
Among 350 instances tested, the MISOCP model reached the time limit in 13 instances, returned suboptimal solutions in 7 instances, and returned infeasible cutting off all solutions in 11 instances.
Although the MISOCP generally outperforms the other models, these numerical issues highlight areas for further improvement in robustness.

\begin{table}[htp]
\footnotesize
\centering
\caption{Performance of MIQP, MIQP-p, and MISOCP on path-following problem.}
\makebox[\textwidth]{\setlength\tabcolsep{4pt}
\begin{tabular}{c|c|ccc|ccc|ccc}
\hline \hline
\multicolumn{2}{c|}{\multirow{2}{*}{Instance}}     & \multicolumn{3}{c|}{Solve time (sec)} & \multicolumn{3}{c|}{Root gap (\%)} & \multicolumn{3}{c}{\# B \& B nodes} \\ \cline{3-11}
\multicolumn{2}{c|}{}   & miqp       & miqp-p     & misocp     & miqp     & miqp-p     & misocp    & miqp        & miqp-p     & misocp   \\ \hline
\multicolumn{1}{c|}{\multirow{7}{*}{n}}   &10   & 0.04       & 0.06       & 0.10       & 23.0     & 20.2       & 0.5       & 31          & 50         & 1.0      \\
\multicolumn{1}{c|}{}   &20   & 0.14       & 0.12       & 0.25       & 27.2     & 24.0       & 0.2       & 553         & 636        & 1.0      \\
\multicolumn{1}{c|}{}   &30   & 0.41       & 0.36       & 0.65       & 31.8     & 28.1       & 0.2       & 4,680        & 5,124       & 1.0      \\
\multicolumn{1}{c|}{}   &40   & 5.31       & 5.12       & 1.40       & 28.8     & 25.4       & 0.1       & 125,533      & 128,635     & 1.0      \\
\multicolumn{1}{c|}{}   &50   & 113.18     & 110.20     & 3.13       & 32.4     & 28.8       & 0.2       & 3,065,524     & 3,278,168    & 12.7     \\
\multicolumn{1}{c|}{}   &60   & 297.67     & 316.62     & 4.45       & 30.1     & 26.4       & 0.2       & 7,172,776     & 7,956,593    & 18.0     \\
\multicolumn{1}{c|}{}   &70   & 769.27     & 795.31     & 12.72      & 35.2     & 31.3       & 0.4       & 13,109,308    & 14,015,702   & 56.0     \\ \hline

\multicolumn{1}{c|}{\multirow{5}{*}{$\lambda$}}   &2      & 256.29	&272.18      & 4.0        & 29.3     & 23.6       & 0.2       & 5,091,287    & 5,578,520    & 15.4      \\
\multicolumn{1}{c|}{}   &4      & 249.23	&257.89      & 3.3        & 31.0     & 27.0       & 0.3       & 4,938,694    & 5,463,695    & 15.9      \\
\multicolumn{1}{c|}{}   &6      & 135.18	&138.95     & 3.9        & 31.2     & 28.0       & 0.3       & 2,705,863    & 2,744,289    & 21.4      \\
\multicolumn{1}{c|}{}   &8      & 118.61	&119.16      & 2.3        & 29.9     & 27.3       & 0.4       & 2,391,683    & 2,569,761    & 6.6       \\
\multicolumn{1}{c|}{}   &10     & 65.23	&65.56       & 2.4        & 27.1     & 25.1       & 0.3       & 1,260,723    & 1,357,922    & 4.2       \\ \hline
\multicolumn{2}{c|}{\textbf{Avg.}} & \textbf{165.60}     & \textbf{171.45}     & \textbf{3.16}       & \textbf{29.7}     & \textbf{26.2}       & \textbf{0.3}       & \textbf{3,291,128}     & \textbf{3,559,475}    & \textbf{12.7}    \\ \hline \hline
\end{tabular}}
\captionsetup{justification=centering}
\label{table:block-factorizable_syntheic_avg}
\end{table}

\begin{table}[H]
\caption{Suboptimal solutions and numerical errors of MISOCP.}
\footnotesize
\centering
\makebox[\textwidth]{\setlength\tabcolsep{4pt}
\begin{tabular}{c|c|c|cc|c}
\hline \hline
$n$                   & $\lambda$ & Time limit & Suboptimal & Final gap (\%) $\dagger$ & Cut-off $\star$ \\ \hline
\multirow{1}{*}{30} & 2       & 2/10          & 1/10          & 0.17      &    -     \\ \hline
\multirow{4}{*}{40} & 2        & 3/10          & 2 /10         & 10.00     &  -       \\
                    & 6           & 2/10         & 1/10          & 35.41     & 1/10       \\
                    & 8         & 1/10          & 1/10          & 42.08     &   -      \\
                    & 10         & 1/10          & 1/10          & 22.42     &   -      \\ \hline
\multirow{2}{*}{50}                    & 6         & -          &    -        &     -      & 1/10       \\
                    & 10         & -          &      -      &      -     & 2/10       \\ \hline
\multirow{2}{*}{60} & 8           & -          &     -       &       -    & 2/10       \\
                    & 10          & 1/10          &      -      &      -     & 2/10       \\ \hline
\multirow{3}{*}{70}                    & 4         & -          &       -     &     -      & 1/10       \\
                    & 6          & 2/10          & 1/10          & 15.07     & 1/10      \\
                    & 10           & 1/10          &      -      &    -       & 1/10     \\ \hline \hline 
\end{tabular}}
\label{table:synthetic-subopt}
\vspace{1mm}\\
\begin{tabular}{@{}l@{}}
($\dagger$) The average is reported only for instances with suboptimal solutions.\\
($\star$) Terminated with numerical errors. 
\end{tabular}
\end{table}

\section{Conclusion}\label{sec:conclusion}
In this paper, we consider a multi-period convex quadratic optimization problem with linear state dynamics. By projecting the state variables out using the linear dynamics constraints, we show that the problem can be reformulated as a mixed-integer quadratic programming problem (MIQP) with a quadratic cost represented by a (block-) factorizable matrix. We analyze the properties of (block-) factorizable matrices and leverage these properties to derive an explicit expression for the inverse of any (block) principal submatrix of (block-) factorizable matrices as a compact sum of low-rank matrices. This expression allows us to establish the closure of the convex hull representation of the epigraph of the quadratic cost over the feasible region in a compact convex formulation. We further give a tight SOCP reformulation of the MIQP with $\mathcal{O}(n^2)$ second-order cone constraints. Moreover, we show that the MIQP problem can be cast as a shortest path problem on a directed acyclic graph, allowing a combinatorial polynomial-time algorithm.  We present two case studies on multi-period optimization problems in statistical learning and hybrid system control, illustrating the practical applicability and highlighting the numerical challenges of the proposed approaches.
Our study advances the theoretical understanding of multi-period MIQPs with dependencies between state variables of consecutive periods and proposes computationally efficient solution methodologies.

\bibliographystyle{plain}
\bibliography{Reference_arXiv}

\newpage
\appendix
\section*{Appendix}
\renewcommand{\thesubsection}{\Alph{subsection}}
\renewcommand{\thesubsubsection}{\Alph{subsection}.\arabic{subsubsection}}

\subsection{Reformulation of multi-period convex quadratic optimization problems} \label{Appendix:motivating_ex}
\subsubsection{One-dimensional state variables} \label{Appendix:motivating_ex_singleton}
Consider the following $n$-period mixed-integer quadratic optimization problem with linear transitions between periods: 
\begin{subequations}
    \begin{align}
        \hspace{3cm}\min \ & \sum_{i=1}^{n+1} p_{i} \left(s_{i} - r_{i}\right)^2 + \sum_{i=1}^n f_i x_i + \sum_{i=1}^{n} c_i z_i \hspace{-5cm} \label{MIQP-1dim_obj_rep} \\ 
        \text{s.t. } \ & s_1 = \beta_0, \label{MIQP-1dim_init_rep} \\
        & s_{i+1} = \alpha_i s_i + x_i + \beta_i, &&i \in [n] \hspace{3cm} \label{MIQP-1dim_linear_rep} \\
        & x_i (1-z_i) = 0,  &&i \in [n]  \\
        & s_{i+1}, x_i \in \mathbb{R}, \ z_i \in \{0,1\}, &&i \in [n] 
    \end{align} \label{MIQP-1dim_rep}%
\end{subequations}
where $\alpha_i \neq 0$ and $p_{i}>0$, $i \in [n+1]$. Note that \eqref{MIQP-1dim_rep} is a special case of \eqref{MIQP-ddim} with $d=1$. Employing \eqref{MIQP-1dim_init_rep} and \eqref{MIQP-1dim_linear_rep}, $s_i$ can be stated as
\begin{align*}
    s_i =  \left(\Pi_{t=1}^{i-1} \alpha_{i-t} \right) \beta_0  + \sum_{\tau = 1}^{i-1} \left( \Pi_{t=1}^{i - \tau -1} \alpha_{i - t}\right) \left(x_\tau + \beta_\tau \right) , \qquad i \in [n+1].
\end{align*}
Then, by projecting out the state variables $s_i$, $\forall i \in [n+1]$, a reduced formulation of \eqref{MIQP-1dim_rep} is obtained:
\begin{subequations}
    \begin{align}
        \min \ \ & \sum_{i=1}^{n+1} p_{i} \left[ \left(\Pi_{t=1}^{i-1} \alpha_{i-t} \right) \beta_0 + \sum_{\tau=1}^{i-1} \left(\Pi_{t=1}^{i-\tau-1} \alpha_{i-t}\right) \left( x_\tau + \beta_\tau \right) - r_{i}\right]^2 + \sum_{i=1}^n f_i x_i + \sum_{i=1}^n c_i z_i \label{MIQP-1dim_obj_proj_rep} \\ 
        \text{s.t. } \ & x_i (1-z_i) = 0, \ x_i \in \mathbb{R}, \ z_i \in \{0,1\},  \qquad i \in [n]. \label{MIQP-1dim_size_rep}
    \end{align} \label{MIQP-1dim_proj_rep}%
\end{subequations}
Letting $g_i = \left(\Pi_{t=1}^{i-1} \alpha_{i-t} \right) \beta_0 + \sum_{\tau=1}^{i-1} \left(\Pi_{t=1}^{i-\tau-1} \alpha_{i-t}\right) \beta_\tau - r_{i}$, $ i \in [n+1]$, \eqref{MIQP-1dim_obj_proj_rep} can be written as
\begin{align*}
    \sum_{i=1}^{n+1} p_{i} \left[ \left(\sum_{\tau=1}^{i-1} \left(\Pi_{t=1}^{i-\tau-1} \alpha_{i-t}\right) x_\tau \right) + g_i \right]^2 + \sum_{i=1}^n f_i x_i +  \sum_{i=1}^n c_i z_i =  \boldsymbol{x}^\top \boldsymbol{Q} \boldsymbol{x} + \boldsymbol{a}^\top \boldsymbol{x}  + \boldsymbol{c}^\top \boldsymbol{z} + v
\end{align*}
where $v = \sum_{i=1}^{n+1} p_{i} g_i^2$, and $\boldsymbol{a} \in \bbR^n$ is a vector with the $i$-th entry defined as
\begin{align*}
    a_i = 2 \sum_{\tau=i+1}^{n+1} p_{\tau} g_{\tau} \left( \Pi_{t=1}^{\tau - i-1} \alpha_{\tau-t} \right) + f_i, \qquad 1 \leq i \leq n
\end{align*}
and the matrix $\boldsymbol{Q} = \begin{bmatrix}
    Q_{ij}
\end{bmatrix}_{i,j \in [n]} \in \bbR^{n \times n}$ is formulated as
\begin{equation*}
    \begin{aligned}
        Q_{ij} = Q_{ji} & = \sum_{\tau=j+1}^{n+1} p_{\tau} \left(\Pi_{t=1}^{\tau-i-1} \alpha_{\tau -t}\right) \left(\Pi_{t=1}^{\tau-j-1} \alpha_{\tau-t}\right), && 1 \leq i \leq j \leq n .
    \end{aligned} 
\end{equation*}
Note that $\bm{Q} \succ \bm{0}$ if $p_i > 0$, $ i \in [2,n+1]$, as $\bm{Q}$ can be decomposed as
\begin{align*}
    \bm{Q} = \overbrace{\begin{bmatrix}
     1 & \alpha_2 & \alpha_2 \alpha_3 & \cdots & \alpha_2 \Compactcdots \alpha_n \\
     & 1 & \alpha_3 & \cdots & \alpha_3 \Compactcdots \alpha_n  \\
     & & 1 & \cdots & \alpha_4 \Compactcdots \alpha_n\\
     & & & \ddots & \vdots\\
     & & & & 1
\end{bmatrix}}^{\bm{L}^\top} \overbrace{\begin{bmatrix}
    p_2 & & & &\\
     & p_3 & & &\\
     & & p_4 & & \\
     & & &\ddots & \\
     & & & & p_{n+1}
\end{bmatrix}}^{\bm{D}} \overbrace{\begin{bmatrix}
    1 & & & & \\
    \alpha_2 & 1 & & &\\
    \alpha_3 \alpha_2 & \alpha_3 & 1 & & \\
    \vdots & \vdots & \vdots & \ddots & \\
    \alpha_n \Compactcdots \alpha_2 & \alpha_n \Compactcdots \alpha_3 & \alpha_n \Compactcdots \alpha_4 & \cdots & 1
\end{bmatrix}}^{L}.
\end{align*}
Define $\{u_i\}_{i \in [n]}$ and $\{v_i\}_{i \in [n]}$ as
\begin{align*}
    & u_i = \Pi_{t=1}^{n-i} \alpha_{n+1-t}, \quad v_i = \frac{\sum_{\tau=i+1}^{n+1} p_{\tau} \left(\Pi_{t=1}^{\tau-i-1} \alpha_{\tau - t}\right)^2 }{\Pi_{t=1}^{n-i} \alpha_{n+1-t}}, \qquad  i \in [n].
\end{align*}
Then, it holds
\begin{align*}
    \hspace{5cm} u_i v_j &= Q_{ij} = Q_{ji}, &&  1 \leq i \leq j \leq n, \hspace{3cm}
\end{align*}
and $\boldsymbol{Q}$ is a factorizable matrix defined by $\{u_i\}_{i \in [n]}$ and $\{v_i\}_{i \in [n]}$. 

As $\alpha_t \neq 0$ and $p_t >0$, $\forall t \in [n]$, it holds $u_i v_i > 0$ for $i \in [n]$. Additionally, we have
\begin{align*}
    u_{j} v_i - u_i v_{j} = \ & \left(\Pi_{t=1}^{n-j} \alpha_{n+1-t}\right) \cdot \frac{\sum_{\tau=i+1}^{n+1} p_{\tau} \left(\Pi_{t=1}^{\tau-i-1} \alpha_{\tau - t}\right)^2 }{\Pi_{t=1}^{n-i} \alpha_{n+1-t}} - \left(\Pi_{t=1}^{n-i} \alpha_{n+1-t}\right) \cdot \frac{\sum_{\tau=j+1}^{n+1} p_{\tau} \left(\Pi_{t=1}^{\tau-j-1} \alpha_{\tau - t}\right)^2 }{\Pi_{t=1}^{n-j} \alpha_{n+1-t}} \\
    = \ & \left(\Pi_{t=n-j+1}^{n-i} \alpha_{n+1-t}\right)^{-1} \sum_{\tau=i+1}^{j} p_{\tau} \left(\Pi_{t=1}^{\tau-i-1} \alpha_{\tau - t}\right)^2 \\
    = \ & \left(\Pi_{t=n-j+1}^{n-i} \alpha_{n+1-t}\right)^{-1} \left( p_{i+1} + \sum_{\tau=i+2}^{j} p_{\tau} \left(\Pi_{t=1}^{\tau-i-1} \alpha_{\tau - t}\right)^2 \right) ,
\end{align*}
and
\begin{align*}
    u_i v_j = \ & \left(\Pi_{t=1}^{n-i} \alpha_{n+1-t}\right) \cdot \frac{\sum_{\tau=j+1}^{n+1} p_{\tau} \left(\Pi_{t=1}^{\tau-j-1} \alpha_{\tau - t}\right)^2 }{\Pi_{t=1}^{n-j} \alpha_{n+1-t}} = \left(\Pi_{t=n-j+1}^{n-i} \alpha_{n+1-t}\right)^{-1} \left( p_{j+1} + \sum_{\tau=j+2}^{n+1} p_{\tau} \left(\Pi_{t=1}^{\tau-j-1} \alpha_{\tau - t}\right)^2 \right)
\end{align*}
for $1 \leq i < j \leq n$. Therefore, 
$$u_i v_j \left( u_{j} v_i - u_i v_{j}\right) = \left( p_{i+1} + \sum_{\tau=i+2}^{j} p_{\tau} \left(\Pi_{t=1}^{\tau-i-1} \alpha_{\tau - t}\right)^2 \right) \left( p_{j+1} + \sum_{\tau=j+2}^{n+1} p_{\tau} \left(\Pi_{t=1}^{\tau-j-1} \alpha_{\tau - t}\right)^2 \right) >0.$$
Consequently, the factorizable matrix $\bm{Q}$ satisfies Assumption~\ref{assump:factorable}.

\subsubsection{Calcium imaging deconvolution}  \label{Appendix:calcium}

Recall the calcium imaging data deconvolution problem given in \S\ref{sec:calcium}:
\begin{equation}
    \begin{aligned}
        \hspace{4cm} \min \ & \frac{1}{2} \sum_{i=1}^{n+1} \left( s_i - r_i \right)^2 + \lambda \sum_{i=1}^{n} z_i \\ 
        \text{s.t. } & s_1 = \beta_0, \\
        & x_{i} = s_{i+1} - \alpha s_{i} \geq 0, &&  i \in [n] \\ 
        & x_i (1-z_i) = 0, &&  i \in [n] \hspace{3cm}\\ 
        & z_i \in \{0,1\}, &&  i \in [n] 
    \end{aligned} \label{calcium_nonnegative_rep}%
\end{equation}
for $\alpha \in (0,1)$.
Then, if the nonnegativity of the spikes $x_i$ is relaxed, \eqref{calcium_nonnegative_rep} is in the form of \eqref{MIQP-1dim_rep} with $p_i = \frac{1}{2}$, $f_i = 0$, $c_i = \lambda$, $\alpha_i = \alpha$, and $\beta_i = 0$ for $i \in [n]$. Therefore, \eqref{calcium_nonnegative_rep} can be reformulated as 
\begin{equation*}
    \begin{aligned}
        \min \ & \boldsymbol{x}^\top \boldsymbol{Q} \boldsymbol{x} + \boldsymbol{a}^\top \boldsymbol{x} + \lambda \sum_{i=1}^n z_i +v\\
        \text{ s.t. } &\boldsymbol{x} \geq 0, \ \boldsymbol{x} \circ (\boldsymbol{e} - \boldsymbol{z}) = 0, \ \boldsymbol{z} \in \{0,1\}^n,
    \end{aligned} 
\end{equation*}
where $v = \frac{1}{2}\sum_{i=1}^{n+1} \left( \alpha^{i-1} \beta_0 - r_{i} \right)^2$, and $\boldsymbol{a}$, $\boldsymbol{Q}$ are defined as
\begin{align*}
    \boldsymbol{a} = \begin{bmatrix}
        \sum_{t=2}^{n+1} \alpha^{t-2} \left(\alpha^{t-1} \beta_0 - r_t \right) \\
        \sum_{t=3}^{n+1} \alpha^{t-3} \left(\alpha^{t-1} \beta_0 - r_t \right) \\
        \vdots \\
        \alpha^{n} \beta_0 - r_{n+1}  
    \end{bmatrix}, \qquad \boldsymbol{Q} = \frac{1}{2}\begin{bmatrix}
        \sum_{t=2}^{n+1} \alpha^{2t-4} & \sum_{t=3}^{n+1} \alpha^{2t-5} & \cdots & \alpha^{n-1} \\
        \sum_{t=3}^{n+1} \alpha^{2t-5} & \sum_{t=3}^{n+1} \alpha^{2t-6} & \cdots & \alpha^{n-2} \\
        \vdots & \vdots & \ddots & \vdots \\
        \alpha^{n-1} & \alpha^{n-2} & \cdots & 1
    \end{bmatrix}.
\end{align*}
Furthermore, $\boldsymbol{Q} = \boldsymbol{U} \circ \boldsymbol{V}$ holds for matricies $\boldsymbol{U}$ and $\boldsymbol{V}$ defined by $\{u_i\}_{i \in [n]}$ and $\{v_i\}_{i \in [n]}$ as in \eqref{two_seq_mat} with 
\begin{align*}
     u_i = \alpha^{n-i}, \  v_i = \frac{\sum_{t=i+1}^{n+1} \alpha^{2(t-i-1)}}{2\alpha^{n-i}} = \frac{\alpha^{2(n-i+1)} -1}{2\alpha^{n-i} \left(\alpha^2 - 1 \right)}, \quad i \in [n].
\end{align*}
Here, we assume $s_1 = \beta_0$ is known. When $s_1$ is unknown, a similar formulation can be obtained by introducing auxiliary variables $x_0$ and $z_0$, the input and indicator variables in period $0$, with constraints $x_0 (1-z_0)=0$ and $z_0=1$, ensuring that the initial state $s_1 = x_0$ has no restriction by the indicator $z_0$.

\subsubsection{Multi-dimensional variables} \label{Appendix:motivating_ex_block}
Consider the $n$-period quadratic optimization problem with linear transitions and multi-dimensional decision variables:
\begin{equation}
    \begin{aligned}
        \hspace{3cm}\min \ & \sum_{i=1}^{n+1} \left(\boldsymbol{s}_{i} - \boldsymbol{r}_{i}\right)^\top \boldsymbol{P}_{i} \left(\boldsymbol{s}_{i} - \boldsymbol{r}_{i}\right) + \sum_{i=1}^n \bm{f}_i^\top \bm{x}_{[i]} + \sum_{i=1}^{n} c_i z_i \hspace{-2.2cm} \\ 
        \text{s.t. } \ & \boldsymbol{s}_1 = \boldsymbol{b}_0, \\ 
        & \boldsymbol{s}_{i+1} = \boldsymbol{A}_i \boldsymbol{s}_i + \boldsymbol{x}_{[i]} + \boldsymbol{b}_i, && i \in [n] \hspace{3cm} \\ 
        & \boldsymbol{x}_{[i]} (1-z_i) = \boldsymbol{0}\sz{d}, \ z_i \in \{0,1\},  && i \in [n] \\
        & \left( \left\{\boldsymbol{s}_i\right\}_{i=1}^{n+1}, \bm{x}, \boldsymbol{z}\right) \in C \subseteq \bbR^{(n+1)d} \times \bbR^{nd} \times \bbR^n \hspace{-2.2cm}
    \end{aligned} \label{MIQP-ddim_rep}%
\end{equation}
where $\boldsymbol{P}_{i}, \boldsymbol{A}_i \in \bbR^{d \times d}$ are nonsingular, $\bm{c} \in \R^n$, $\boldsymbol{r}_i, \bm{f}_i, \boldsymbol{b}_i \in \bbR^d$, and $\boldsymbol{P}_i \succ \bm{0}$ for all $i$. As in Appendix~\ref{Appendix:motivating_ex_singleton}, the state variables $\boldsymbol{s}_i$, $i \in [n+1]$, can be expressed as a function of the input variables $\boldsymbol{x}_i$:
\begin{align*}
    \boldsymbol{s}_i = \left( \Pi_{t=1}^{i-1} \boldsymbol{A}_{i-t} \right) \boldsymbol{b}_0 + \sum_{\tau=1}^{i-1} \left( \Pi_{t=1}^{i -\tau-1} \boldsymbol{A}_{i-t} \right) \left( \boldsymbol{x}_{[\tau]} + \boldsymbol{b}_{\tau} \right).
\end{align*}
Then, \eqref{MIQP-ddim_rep} can be projected into a reduced space as follows:
\begin{equation*}
    \begin{aligned}
        \min \ & \sum_{i=1}^{n+1} \left[\sum_{\tau=1}^{i-1} \left(\Pi_{t=1}^{i-\tau-1} \boldsymbol{A}_{i-t}\right) \boldsymbol{x}_{[\tau]}  
        + \boldsymbol{g}_i \right]^\top \boldsymbol{P}_i \left[\sum_{\tau=1}^{i-1} \left(\Pi_{t=1}^{i-\tau-1} \boldsymbol{A}_{i-t}\right) \boldsymbol{x}_{[\tau]}  
        + \boldsymbol{g}_i \right] + \sum_{i=1}^n \bm{f}_i^\top \bm{x}_{[i]} + \sum_{i=1}^n c_i z_i \\ 
        \text{s.t. } \ & \bm{x}_{[i]} (1-z_i) = \bm{0},  && \hspace{-8.8cm} i \in [n] \hspace{6cm} \\ 
        & \bm{x}_{[i]} \in \mathbb{R}^d, \ z_i \in \{0,1\}, && \hspace{-8.8cm} i \in [n] 
    \end{aligned} 
\end{equation*}
where $\boldsymbol{g}_i = \left(\Pi_{t=1}^{i-1} \boldsymbol{A}_{i-t} \right) \boldsymbol{b}_0 + \sum_{\tau=1}^{i-1} \left(\Pi_{t=1}^{i-\tau-1} \boldsymbol{A}_{i-t}\right) \boldsymbol{b}_\tau - \boldsymbol{r}_{i}$, $i \in [n+1]$. After stacking the input variables $\boldsymbol{x}_{[i]}$, $\forall i \in [n]$, into a vector $\boldsymbol{x}\sz{dn}$, the following formulation is obtained:
\begin{equation*}
    \begin{aligned}
        \min \ & \boldsymbol{x}\sz{dn}^\top \boldsymbol{Q}\sz{dn} \boldsymbol{x}\sz{dn} + \boldsymbol{a}\sz{dn}^\top \boldsymbol{x}\sz{dn} + \boldsymbol{c}^\top \boldsymbol{z} + v \\ 
        \text{s.t. } \ & \boldsymbol{x}_{[i]}\sz{dn} (1 - z_i) = 0, \qquad \qquad i \in [n] \\
        & \boldsymbol{x}\sz{dn} \in \mathbb{R}^{dn}, \ \boldsymbol{z} \in \{0,1\}^n 
    \end{aligned}
\end{equation*}
where $v = \sum_{i=1}^{n+1} \boldsymbol{g}_i^\top \boldsymbol{P}_i \boldsymbol{g}_i$, $\boldsymbol{a}\sz{dn} \in \bbR^{dn}$ is a concatenated vector with $n$ of $d$-dimensional subvectors of which the $i$-th subvector is defined as
\begin{align*}
    \boldsymbol{a}_{[i]}\sz{d} = 2 \sum_{\tau=i+1}^{n+1} \left( \Pi_{t=1}^{\tau-i-1} \boldsymbol{A}_{\tau - t} \right)^\top \boldsymbol{P}_{\tau} \boldsymbol{g}_{\tau} + \bm{f}_i , \qquad i \in [n]
\end{align*}
and $\boldsymbol{Q}\sz{dn} \in \bbR^{dn \times dn}$ is a symmetric block matrix of which the $(i,j)$-th block $\boldsymbol{Q}\sz{d}_{[ij]} \in \bbR^{d \times d}$ is constructed as
\begin{equation*}
    \begin{aligned}
        \boldsymbol{Q}\sz{d}_{[ij]} = \boldsymbol{Q}\sz{d}_{[ji]}^\top & = \sum_{\tau=j+1}^{n+1} \left(\Pi_{t=1}^{\tau-i-1} \boldsymbol{A}_{\tau - t}\right)^\top \boldsymbol{P}_{\tau}  \left(\Pi_{t=1}^{\tau-j-1} \boldsymbol{A}_{\tau-t}\right), \qquad 1 \leq i \leq j \leq n .
    \end{aligned} 
\end{equation*}
The matrix $\bm{Q}$ can be decomposed as
\begin{align*}
    \bm{Q} = \overbrace{\begin{bmatrix}
     \bm{I} & \bm{A}_2^\top & \bm{A}_2^\top \bm{A}_3^\top & \cdots & {\left(\Pi_{i=1}^{n-1} \bm{A}_{n+1-i}\right)^\top}\\
     & \bm{I} & \bm{A}_3 & \cdots & \left(\Pi_{i=1}^{n-2} \bm{A}_{n+1-i}\right)^\top  \\
     & & \bm{I} & \cdots & \left(\Pi_{i=1}^{n-3} \bm{A}_{n+1-i}\right)^\top\\
     & & & \ddots & \vdots\\
     & & & & \bm{I}
\end{bmatrix}}^{\bm{L}^\top} & \overbrace{\begin{bmatrix}
    \bm{P}_2 & & & &\\
     & \bm{P}_3 & & &\\
     & & \bm{P}_4 & & \\
     & & &\ddots & \\
     & & & & \bm{P}_{n+1}
\end{bmatrix}}^{\bm{D}} \\
& \times \overbrace{\begin{bmatrix}
    \bm{I} & & & & \\
    \bm{A}_2 & \bm{I} & & &\\
    \bm{A}_3 \bm{A}_2 & \bm{A}_3 & \bm{I} & & \\
    \vdots & \vdots & \vdots & \ddots & \\
    \Pi_{i=1}^{n-1} \bm{A}_{n+1-i} & \Pi_{i=1}^{n-2} \bm{A}_{n+1-i} & \Pi_{i=1}^{n-3} \bm{A}_{n+1-i} & \cdots & \bm{I}
\end{bmatrix}}^{L}, 
\end{align*}
and, therefore, $\bm{Q} \succ \bm{0}$ if $\bm{P}_i \succ \bm{0}$, $\forall i \in [2,n+1]$.
Moreover, $\bm{Q} = \boldsymbol{U}\sz{dn} \bullet \boldsymbol{V}\sz{dn} $ is a block-factorizable matrix, where matrices $\bm{U}\sz{dn}$ and $\bm{V}\sz{dn}$ are of the form described in Definition~\ref{def:blockSep} with $\left\{\boldsymbol{U}_i\sz{d} \in \bbR^{d \times d} \right\}_{i \in [n]}$ and $\left\{\boldsymbol{V}_i\sz{d} \in \bbR^{d \times d} \right\}_{i \in [n]}$ given as
\begin{align*}
    & \boldsymbol{U}_i\sz{d} = \left(\Pi_{t=1}^{n-i} \boldsymbol{A}_{n-t+1}\right)^\top, \quad  i \in [n-1], \quad \text{ and } \quad \boldsymbol{U}_n\sz{dn}=\boldsymbol{I}\sz{dn}, \\
    & \boldsymbol{V}_i\sz{d} = \left[\sum_{\tau=i+1}^{n+1} \left(\Pi_{t=1}^{\tau-i-1} \boldsymbol{A}_{\tau - t}\right)^\top \boldsymbol{P}_{\tau} \left(\Pi_{t=1}^{\tau-i-1} \boldsymbol{A}_{\tau-t}\right) \right]\left( \Pi_{t=1}^{n-i} \boldsymbol{A}_{n-t+1} \right)^{-1}, \quad  i \in [n].
\end{align*}
Note that $\bm{U}_i, \bm{V}_i$ are nonsingular for all $t \in [n]$, and
\begin{align*}
    \bm{U}_i \bm{V}_i^\top 
    = \bm{P}_{i+1} + \sum_{\tau=i+2}^{n+1} \left(\Pi_{t=1}^{\tau-i-1} \boldsymbol{A}_{\tau - t}\right)^\top \boldsymbol{P}_{\tau} \left(\Pi_{t=1}^{\tau-i-1} \boldsymbol{A}_{\tau-t}\right) \succeq \bm{P}_{i+1} \succ \bm{0}.
\end{align*}
Moreover,
\begin{align*}
    \bm{U}_i \bm{U}_{j}^{-1} \left(\bm{U}_{j} \bm{V}_i^\top - \bm{V}_{j}\bm{U}_i^\top\right) 
    = \bm{P}_{i+1} + \sum_{\tau=i+2}^{j} \left(\Pi_{t=1}^{\tau-i-1} \boldsymbol{A}_{\tau - t}\right)^\top \boldsymbol{P}_{\tau} \left(\Pi_{t=1}^{\tau-i-1} \boldsymbol{A}_{\tau-t}\right) \succeq \bm{P}_{i+1} \succ \bm{0},
\end{align*}
for $1\leq i < j \leq n$. Therefore, the block-factorizable matrix $\bm{Q}$ satisfies Assumption~\ref{assum:block-factorizable}.


\subsubsection{Multi-dimensional path-following problem for power-split hybrid electric vehicle control}  \label{Appendix:hev-block}
Recall the multi-dimensional path-following problem \eqref{path-following}, which can be formulated as the following MIQP problem with auxiliary input variables $\boldsymbol{x}_{[i]}$:
\begin{equation}
    \begin{aligned}
        \min \ \ & \sum_{i = 1}^{n+1} \left( \boldsymbol{s}_i - \boldsymbol{r}_i\right)^\top \boldsymbol{P}_i \left( \boldsymbol{s}_i - \boldsymbol{r}_i\right) + \sum_{i = 1}^{n} \boldsymbol{y}_{[i]}^\top \boldsymbol{R}_i \boldsymbol{y}_{[i]} + \sum_{i = 1}^{n} c_i z_{i}  \hspace{-3cm} \\ 
        \text{s.t. } \ 
        & \boldsymbol{s}_1 = \bm{b}_0,\\ 
        & \boldsymbol{s}_{i+1} = \boldsymbol{A}_i \boldsymbol{s}_{i} + \boldsymbol{x}_{[i]}  + \boldsymbol{b}_{i} , &&  i \in [n] \\ 
        & \boldsymbol{x}_{[i]} = \boldsymbol{G}_i \boldsymbol{y}_{[i]} + \boldsymbol{k}_i z_{i}, && i \in [n]  \\ 
        & \boldsymbol{s}_{i}^{\min} \leq \boldsymbol{s}_{i} \leq \boldsymbol{s}_{i}^{\max}, && i \in [n+1] \\ 
        & \boldsymbol{y}_{i}^{\min} z_{i} \leq \boldsymbol{y}_{[i]} \leq \boldsymbol{y}_{i}^{\max} z_{i} , && i \in [n] \\ 
        &z_{i} \in \{0,1\}, && i \in [n]  
    \end{aligned} \label{HEV_multi_repeat}
\end{equation}
This problem corresponds to an $n$-period MIQP problem in the form of \eqref{MIQP-ddim_rep} with additional quadratic costs on the control variables $\boldsymbol{y}_{[i]}$. Specifically, $\bm{f}_i=\bm{0}$, $i \in [n]$, and $C$ incorporates additional constraints 
\begin{align*}
    C = \Bigg\{ \left( \left\{\boldsymbol{s}_i\right\}_{i=1}^{n+1}, \bm{x}, \boldsymbol{z}\right) &\in \bbR^{(n+1)d} \times \bbR^{nd} \times \bbR^n: \exists \boldsymbol{y}_{[i]} \in \bbR^{d_y} \text{ for } 1 \leq i \leq n \text{ such that } \\
    & \ \boldsymbol{x}_{[i]} = \boldsymbol{G}_i \boldsymbol{y}_{[i]} + \boldsymbol{k}_i z_i, \ \boldsymbol{s}_{i+1}^{\min} \leq \boldsymbol{s}_{i+1} \leq \boldsymbol{s}_{i+1}^{\max}, \ \boldsymbol{y}_{i}^{\min} z_{i} \leq \boldsymbol{y}_{[i]} \leq \boldsymbol{y}_{i}^{\max} z_{i}, \ \  i \in [n] \Bigg\}.
\end{align*}
Therefore, the problem \eqref{HEV_multi_repeat} can be reformulated as follows:
\begin{equation*}
    \begin{aligned}
        \min \ \ & \boldsymbol{x}\sz{dn}^\top \boldsymbol{Q}\sz{dn} \boldsymbol{x}\sz{dn} + \sum_{i=1}^n \boldsymbol{y}_{[i]}^\top \boldsymbol{R}_i \boldsymbol{y}_{[i]} + \boldsymbol{a}\sz{dn}^\top \boldsymbol{x}\sz{dn} + \sum_{i=1}^n c_i z_i + v \\
        \text{s.t. } \ & \boldsymbol{x}_{[i]} = \boldsymbol{G}_i \boldsymbol{y}_{[i]} + \boldsymbol{k}_i z_i, && i \in [n] \\
        & \boldsymbol{s}_{i}^{\min} \leq \left( \Pi_{t=1}^{i-1} \boldsymbol{A}_{i-t} \right) \boldsymbol{b}_0 + \sum_{\tau=1}^{i-1} \left( \Pi_{t=1}^{i -\tau-1} \boldsymbol{A}_{i-t} \right) \left( \boldsymbol{x}_{[\tau]} + \boldsymbol{b}_{\tau} \right) \leq \boldsymbol{s}_{i}^{\max}, \quad && i \in [n] \\
        & \boldsymbol{y}_{i}^{\min} z_{i} \leq \boldsymbol{y}_{[i]} \leq \boldsymbol{y}_{i}^{\max} z_{i}, && i \in [n] \\
        &z_{i} \in \{0,1\}, && i \in [n].
    \end{aligned} 
\end{equation*}


\subsection{Positive definiteness of \texorpdfstring{$\bm{Q}$}{\textbf{Q}}} \label{Appendix:PD}

Here we show that Assumption~\ref{assump:factorable} and Assumption~\ref{assum:block-factorizable} are necessary and sufficient conditions for the positive definiteness of a factorizable matrix \eqref{two_seq_mat_repeat} and a block-factorizable matrix \eqref{two_seq_block_mat}, respectively.

\begin{proposition}\label{prop:PSD}
A factorable matrix $\bm{Q}$ defined as \eqref{two_seq_mat_repeat} is positive definite if and only if Assumption~\ref{assump:factorable} holds.
\end{proposition}
\begin{proof}
The conditions $\bm{Q}\succ \bm{0}$ and $\bm{Q^{-1}}\succ \bm{0}$ are equivalent, thus we can use the representation of Proposition~\ref{prop:rank-1} to prove the result. From
\eqref{Q_inv_mat}, it is obvious that if $\delta_i >0$ for all $i \in [n]$, then $\bm{Q} \succ \bm{0}$. We now show that the converse is also true.
Indeed, suppose that there exists  $j \in [n]$ such that $\delta_j \leq 0$. Take $\boldsymbol{x}$ given as
\begin{align*}
    & x_j = 1,\; x_{j-1}= \theta_j,\; x_{j-2} =  \theta_j \theta_{j-1},\;\dots,\; x_1 = \Pi_{t=2}^{j}  \theta_t,\quad \text{and} \quad x_{j+1}, \ldots , x_n =0.
\end{align*}
Then, 
\begin{align*}
    \boldsymbol{x}^\top \boldsymbol{Q}^{-1} \boldsymbol{x} = \sum_{i=1}^{n-1} \delta_i \left (x_i -  \theta_{i+1} x_{i+1} \right )^2 + \delta_n x_n^2 = \sum_{i=1}^{j-1} \delta_i \left[\left(\Pi_{t=i+1}^j  \theta_t \right) -  \theta_{i+1} \left(\Pi_{t=i+2}^j  \theta_t \right)\right]^2 + \delta_j = \delta_j \leq 0;
\end{align*}
therefore,  $\boldsymbol{Q}^{-1}$ is not positive definite. 
Finally, we verify that $\bm{\delta}>\bm{0}$ under Assumption~\ref{assump:factorable}. 
From \eqref{pq_closed_form}, $\delta_i > 0$ for $i \in [n-1]$ by \eqref{assump1_cond2_v2}, an equivalent second condition, and $\delta_n >0$ by the first condition. 
\end{proof}

\begin{proposition} \label{prop:PD_block}
A block-factorable matrix $\boldsymbol{Q}\sz{dn}$ defined as \eqref{two_seq_block_mat} is positive definite if and only if Assumption~\ref{assum:block-factorizable} holds.
\end{proposition}
\begin{proof}
Since $\bm{Q}^{-1}$ can be decomposed as \eqref{Q_decomp_block}, it is trivial to check that if $\bm{\Delta}_i \succ \bm{0}$, $\forall i \in [n]$, then $\bm{Q} \succ \bm{0}$. 
Now suppose $\bm{\Delta}_j$ for some $j \in [n]$ is not positive definite while $\bm{Q} \succ \bm{0}$.
Then, there exists $\bm{y} \in \R^{d}$ such that $\bm{y} \neq \bm{0}$ and $\bm{y}^\top \bm{\Delta}_j \bm{y} \leq 0$.
Similarly, as in the proof of Proposition~\ref{prop:PSD}, take $\bm{x}$ as
\begin{align*}
    \bm{x}_{[j]} = \bm{y}, \; \bm{x}_{[j-1]} = \bm{\Theta}_j\bm{y}, \; \bm{x}_{[j-2]} = \bm{\Theta}_{j-1} \bm{\Theta}_{j}\bm{y}, \; \ldots, \; \bm{x}_{[1]} = \left(\Pi_{t=2}^{j} \bm{\Theta}_{t}\right) \bm{y}, \quad \text{and} \quad \bm{x}_{[j+1]},\ldots,\bm{x}_{[n]} = \bm{0}.
\end{align*}
Then, 
\begin{align*}
    \boldsymbol{x}^\top \boldsymbol{Q}^{-1} \boldsymbol{x} = \ & \sum_{i=1}^{n-1} \left( \bm{x}_{[i]} - \bm{\Theta}_{i+1} \bm{x}_{[i+1]} \right)^\top \bm{\Delta}_i \left( \bm{x}_{[i]} - \bm{\Theta}_{i+1} \bm{x}_{[i+1]} \right) + \bm{x}_{[n]}^\top \bm{\Delta}_n \bm{x}_{[n]} \\
    = \ & \sum_{i=1}^{j-1} \left[\left(\Pi_{t=i+1}^j \bm{\Theta}_t \right) \bm{y} - \bm{\Theta}_{i+1} \left(\Pi_{t=i+2}^j \bm{\Theta}_t \right) \bm{y}\right]^\top \bm{\Delta}_i  \left[\left(\Pi_{t=i+1}^j \bm{\Theta}_t \right) \bm{y} - \bm{\Theta}_{i+1} \left(\Pi_{t=i+2}^j \bm{\Theta}_t \right) \bm{y}\right] + \bm{y}^\top \bm{\Delta}_j \bm{y}\\
    = \ & \bm{y}^\top \bm{\Delta}_j \bm{y} \leq 0, 
\end{align*}
which contradicts $\bm{Q} \succ \bm{0}$. Therefore, for $\bm{Q}$ to be positive definite, it must be that $\bm{\Delta}_i \succ \bm{0}$ for all $i \in [n]$, and it is trivial to check that Assumption~\ref{assum:block-factorizable} guarantees this condition.
\end{proof}

\subsection{Proof of Proposition~\ref{prop:block_factorizable_inverse}} \label{Appendix:proof_prop_block_inv} 
In this section, we present a thorough proof of Proposition~\ref{prop:block_factorizable_inverse} from scratch.
Let 
\begin{align*}
    \boldsymbol{\Xi} 
    & = \sum_{i=1}^{n-1} \left( \boldsymbol{E}_i - \boldsymbol{E}_{i+1} \boldsymbol{\Theta}_{i+1}^\top \right) \boldsymbol{\Delta}_i \left( \boldsymbol{E}_i - \boldsymbol{E}_{i+1} \boldsymbol{\Theta}_{i+1}^\top \right)^\top + \boldsymbol{E}_{n} \boldsymbol{\Delta}_n \boldsymbol{E}_n^\top, 
\end{align*}
where
\begin{align*}
    \boldsymbol{\Delta}_i & = \left( \boldsymbol{U}_i \boldsymbol{V}_i^\top - \boldsymbol{U}_i \boldsymbol{U}_{i+1}^{-1} \boldsymbol{V}_{i+1} \boldsymbol{U}_i^\top \right)^{-1}, \ \boldsymbol{\Theta}_{i+1} = \boldsymbol{U}_{i} \boldsymbol{U}_{i+1}^{-1}, \ \ i \in [n-1], \quad \text{ and } \quad \boldsymbol{\Delta}_n  = \left( \boldsymbol{U}_n \boldsymbol{V}_n^\top \right)^{-1}.
\end{align*}
We show $\boldsymbol{\Xi}^{-1}  = \bm{U} \bullet \bm{V}$. 

The matrix $\boldsymbol{\Xi}$ can be factorized as follows:
    \begin{align}
        \boldsymbol{\Xi} = \overbrace{\begin{bmatrix}
            \boldsymbol{I}\sz{d} & & & & \\
            -\boldsymbol{\Theta}_2^\top & \boldsymbol{I}\sz{d} & & & \\
            & \ddots & \ddots & & \\
            & & -\boldsymbol{\Theta}_{n-1}^\top & \boldsymbol{I}\sz{d} & \\
            & & & -\boldsymbol{\Theta}_{n}^\top & \boldsymbol{I}\sz{d}
        \end{bmatrix}}^{\boldsymbol{L}} \overbrace{\begin{bmatrix}
            \boldsymbol{\Delta}_1 & & & & \\
            & \boldsymbol{\Delta}_2 & & & \\
            & & \ddots & & \\
            & & & \boldsymbol{\Delta}_{n-1} & \\
            & & & & \boldsymbol{\Delta}_n
        \end{bmatrix}}^{\boldsymbol{\Delta}} \overbrace{\begin{bmatrix}
            \boldsymbol{I}\sz{d} & -\boldsymbol{\Theta}_2 & & & \\
            & \boldsymbol{I}\sz{d} & -\boldsymbol{\Theta}_3 & & \\
            & & \ddots & \ddots & \\
            & & & \boldsymbol{I}\sz{d} & -\boldsymbol{\Theta}_{n}\\
            & & & & \boldsymbol{I}\sz{d}
        \end{bmatrix}}^{\boldsymbol{L}^\top} \label{Q_decomp_block}
    \end{align}
    Let $\Bar{\boldsymbol{L}} = \boldsymbol{L} - \boldsymbol{I}\sz{dn}$, a block subdiagonal matrix. 
    \begin{claim}\label{claim:Lbar_k}
    \begin{align}
        \Bar{\boldsymbol{L}}^k = \begin{bmatrix}
            \boldsymbol{O}\sz{d} & \boldsymbol{O}\sz{d} & \cdots & \boldsymbol{O}\sz{d} & \boldsymbol{O}\sz{d} & \cdots & \boldsymbol{O}\sz{d} \\
            \vdots & \vdots & \ddots & \vdots & \vdots & & \vdots \\
            \boldsymbol{O}\sz{d} & \boldsymbol{O}\sz{d} & \cdots & \boldsymbol{O}\sz{d} & \boldsymbol{O}\sz{d} & \cdots & \boldsymbol{O}\sz{d} \\
            (-1)^k \boldsymbol{\Theta}_{k+1}^\top \cdots \boldsymbol{\Theta}_{2}^\top  & \boldsymbol{O}\sz{d} & \cdots & \boldsymbol{O}\sz{d} & \boldsymbol{O}\sz{d} & \cdots & \boldsymbol{O}\sz{d} \\
            \boldsymbol{O}\sz{d} & (-1)^k \boldsymbol{\Theta}_{k+2}^\top \cdots \boldsymbol{\Theta}_{3}^\top & \cdots & \boldsymbol{O}\sz{d} & \boldsymbol{O}\sz{d} & \cdots & \boldsymbol{O}\sz{d} \\
            \vdots & \vdots & \ddots & \vdots & \vdots & & \vdots \\
            \boldsymbol{O}\sz{d} & \boldsymbol{O}\sz{d} & \cdots & (-1)^k \boldsymbol{\Theta}_{n}^\top \cdots \boldsymbol{\Theta}_{n-k+1}^\top  & \boldsymbol{O}\sz{d} & \cdots & \boldsymbol{O}\sz{d}
            \end{bmatrix} \label{Lbar_k}
    \end{align}    
    and, thus, $\Bar{\boldsymbol{L}}$ is a nilpotent matrix with the index $n$. 
    \end{claim}
    \begin{proof}
        We prove the claim by induction. 
        \begin{enumerate}[label=\roman*.]
            \item $k=2$: The $(i,j)$-th block $\left(\Bar{\boldsymbol{L}}^2\right)_{[ij]}$ is 
            \begin{align*}
                \left(\Bar{\boldsymbol{L}}^2\right)_{[ij]} = \left(- \boldsymbol{\Theta}_i^\top \boldsymbol{E}_{i-1}^\top \right) \left( - \boldsymbol{E}_{j+1} \boldsymbol{\Theta}_{j+1}^\top \right) = \begin{cases}
                    \boldsymbol{\Theta}_{i}^\top \boldsymbol{\Theta}_{j+1}^\top, & \text{ if } i = j+2, \\
                    \boldsymbol{O}\sz{d}, & \text{ o.w.}
                \end{cases}
            \end{align*}
            for $i \geq 2$ and $j \leq n-1$, and $\left(\Bar{\boldsymbol{L}}^2\right)_{[1,k]}, \left(\Bar{\boldsymbol{L}}^2\right)_{[\ell,n]} = 0$ for $k \in [2,n]$ and $\ell \in [n-1]$.
            \item Suppose $\Bar{\boldsymbol{L}}^k$ is in the form of \eqref{Lbar_k}. Then, $\Bar{\boldsymbol{L}}^{k+1} = \Bar{\boldsymbol{L}}^{k} \Bar{\boldsymbol{L}}$, and the $(i,j)$-th block $\left(\Bar{\boldsymbol{L}}^{k+1}\right)_{[ij]}$ is as follows:
            \begin{align*}
                \left(\Bar{\boldsymbol{L}}^{k+1}\right)_{[ij]} = \left( (-1)^k \boldsymbol{\Theta}_{i}^\top \cdots \boldsymbol{\Theta}_{i-k+1}^\top \boldsymbol{E}_{i-k}^\top \right) \left( - \boldsymbol{E}_{j+1} \boldsymbol{\Theta}_{j+1}^\top \right) = \begin{cases}
                    (-1)^{k+1} \boldsymbol{\Theta}_{i}^\top \cdots \boldsymbol{\Theta}_{i-k+1}^\top \boldsymbol{\Theta}_{j+1}^\top, & \text{ if } i = j+k+1, \\
                    \boldsymbol{O}\sz{d}, & \text{ o.w.}
                \end{cases}
            \end{align*}
            for $i \geq k+1$ and $ j \leq n-1$, and $\left(\Bar{\boldsymbol{L}}^2\right)_{[ij]} = 0$, if $i \leq k$ or $j = n$.
        \end{enumerate}
        Therefore, \eqref{Lbar_k} holds for $k \in [n-1]$. Note that since $\Bar{\boldsymbol{L}}^{n-1}$ has all zero except the $(n,1)$-th block
        \begin{align*}
            \left(\Bar{\boldsymbol{L}}^{n-1}\right)_{n,1} = (-1)^{n-1} \boldsymbol{\Theta}_{n}^\top \cdots \boldsymbol{\Theta}_2^\top.
        \end{align*}
        Thus, $\Bar{\boldsymbol{L}}^{n} = \boldsymbol{O}\sz{dn}$ as the $(1,j)$-th block of $\Bar{\boldsymbol{L}}$ is $\boldsymbol{O}\sz{d}$, $\forall j \in [n]$.
    \end{proof}

\noindent Note that $\bm{\Theta}_{k+\ell}^\top \cdots \bm{\Theta}_{\ell + 1}^\top = \boldsymbol{U}_{k + \ell}^{-\top} \boldsymbol{U}_{\ell}^\top$ for $k \in [n]$ and $\ell \in [n-k]$, as $\boldsymbol{\Theta}_{i+1} = \boldsymbol{U}_{i} \boldsymbol{U}_{i+1}^{-1}$ for $i \in [n-1]$. Utilizing Claim \ref{claim:Lbar_k}, $\boldsymbol{L}^{-1}$ can be computed as follows:
\begin{align*}
    \boldsymbol{I}\sz{dn} &= \boldsymbol{I}\sz{dn} + \Bar{\boldsymbol{L}}^n = \left( \boldsymbol{I}\sz{dn} + \Bar{\boldsymbol{L}} \right) \left( \boldsymbol{I}\sz{dn} - \Bar{\boldsymbol{L}} + \Bar{\boldsymbol{L}}^2 - \Bar{\boldsymbol{L}}^3 + \cdots + (-1)^{n-1} \Bar{\boldsymbol{L}}^{n-1} \right) \\
    \Rightarrow \  \boldsymbol{L}^{-1} &= \left( \boldsymbol{I}\sz{dn} + \Bar{\boldsymbol{L}} \right)^{-1} \\
    &= \boldsymbol{I}\sz{dn} - \Bar{\boldsymbol{L}} + \Bar{\boldsymbol{L}}^2 - \Bar{\boldsymbol{L}}^3 + \cdots + (-1)^{n-1} \Bar{\boldsymbol{L}}^{n-1}\\
    & = \begin{bmatrix}
            \boldsymbol{I}\sz{d} & &&&&& \\
            \boldsymbol{\Theta}_{2}^\top & \boldsymbol{I}\sz{d} & &&&&\\
            \vdots & \vdots & \ddots & &&& \\
            \left(\boldsymbol{\Theta}_{k}^\top \cdots \boldsymbol{\Theta}_{2}^\top\right) & \left(\boldsymbol{\Theta}_{k}^\top \cdots \boldsymbol{\Theta}_{3}^\top\right) & \cdots & \boldsymbol{I}\sz{d} & && \\
            \left(\boldsymbol{\Theta}_{k+1}^\top \cdots \boldsymbol{\Theta}_{2}^\top\right)  & \left(\boldsymbol{\Theta}_{k+1}^\top \cdots \boldsymbol{\Theta}_{3}^\top\right) & \cdots & \boldsymbol{\Theta}_{k+1}^\top & \boldsymbol{I}\sz{d} & & \\
            \vdots & \vdots & \ddots & \vdots & \vdots & \ddots & \\
            \left(\boldsymbol{\Theta}_{n}^\top \cdots \boldsymbol{\Theta}_{2}^\top\right) & \left(\boldsymbol{\Theta}_{n}^\top \cdots \boldsymbol{\Theta}_{3}^\top\right) & \cdots & \left(\boldsymbol{\Theta}_{n}^\top \cdots \boldsymbol{\Theta}_{n-k+1}^\top\right)  & \left(\boldsymbol{\Theta}_{n}^\top \cdots \boldsymbol{\Theta}_{n-k+2}^\top\right) & \cdots & \boldsymbol{I}\sz{d}
            \end{bmatrix}\\
            & = \begin{bmatrix}
            \boldsymbol{I}\sz{d} & & & & & & \\
            \boldsymbol{U}_{2}^{-\top} \boldsymbol{U}_{1}^\top & \boldsymbol{I}\sz{d} &  & & & & \\
            \vdots & \vdots & \ddots  & & & & \\
            \boldsymbol{U}_{k}^{-\top} \boldsymbol{U}_{1}^\top & \boldsymbol{U}_{k}^{-\top} \boldsymbol{U}_{2}^\top & \cdots & \boldsymbol{I}\sz{d} & & &  \\
            \boldsymbol{U}_{k+1}^{-\top} \boldsymbol{U}_{1}^\top  & \boldsymbol{U}_{k+1}^{-\top} \boldsymbol{U}_{2}^\top & \cdots & \boldsymbol{U}_{k+1}^{-\top} \boldsymbol{U}_{k}^\top & \boldsymbol{I}\sz{d} & & \\
            \vdots & \vdots & \ddots & \vdots & \vdots & \ddots &  \\
            \boldsymbol{U}_{n}^{-\top} \boldsymbol{U}_{1}^\top & \boldsymbol{U}_{n}^{-\top} \boldsymbol{U}_{2}^\top & \cdots & \boldsymbol{U}_{n}^{-\top} \boldsymbol{U}_{n-k}^\top  & \boldsymbol{U}_{n}^{-\top} \boldsymbol{U}_{n-k+1}^\top & \cdots & \boldsymbol{I}\sz{d}
        \end{bmatrix}.
\end{align*}
Then, $\boldsymbol{\Xi}^{-1} = \boldsymbol{L}^{-\top} \boldsymbol{\Delta}^{-1} \boldsymbol{L}^{-1}$, where $\boldsymbol{\Delta}$ is a block diagonal matrix with $i$-th diagonal block is defined as 
\begin{align*}
    \boldsymbol{\Delta}_i & = \left( \boldsymbol{U}_i \boldsymbol{V}_i^\top - \boldsymbol{U}_i \boldsymbol{U}_{i+1}^{-1} \boldsymbol{V}_{i+1} \boldsymbol{U}_i^\top \right)^{-1}, \ \  i \in [n-1], \quad \text{ and} \quad 
    \boldsymbol{\Delta}_n = \left( \boldsymbol{U}_n \boldsymbol{V}_n^\top \right)^{-1}.
\end{align*}
Then, the $(i,j)$-th block of $\boldsymbol{\Xi}^{-1}$ for $1 \leq i \leq j \leq n$ can be computed as follows:
\begin{align*}
    \boldsymbol{\Xi}^{-1}_{[ij]} & = \left(\boldsymbol{L}^{-1}_i \right)^\top \boldsymbol{\Delta}^{-1} \boldsymbol{L}^{-1}_j\\
    & = \boldsymbol{U}_i \begin{bmatrix}
        \boldsymbol{O}\sz{d} & \cdots & \boldsymbol{O}\sz{d} & \boldsymbol{U}_i^{-1} & \cdots & \boldsymbol{U}_{n}^{-1}
    \end{bmatrix} 
    \begin{bmatrix}
        \boldsymbol{\Delta}_1^{-1} & & & & &\\
        & \ddots & & & & \\
        & & \boldsymbol{\Delta}_{j-1}^{-1} & & & \\
        & & & \boldsymbol{\Delta}_j^{-1} & & \\
        & & & & \ddots & \\
        & & & & & \boldsymbol{\Delta}_n^{-1}
    \end{bmatrix}
    \begin{bmatrix}
        \boldsymbol{O}\sz{d} \\ \vdots \\ \boldsymbol{O}\sz{d} \\ \boldsymbol{U}_j^{-\top} \\ \vdots \\ \boldsymbol{U}_{n}^{-\top}
    \end{bmatrix} \boldsymbol{U}_j^\top\\
    & = \sum_{k=j}^{n} \boldsymbol{U}_i \boldsymbol{U}_k^{-1} \boldsymbol{\Delta}_j^{-1} \boldsymbol{U}_k^{-\top} \boldsymbol{U}_j^\top \\
    & = \sum_{k=j}^{n-1} \boldsymbol{U}_i \boldsymbol{U}_k^{-1} \left( \boldsymbol{U}_k \boldsymbol{V}_k^\top - \boldsymbol{U}_k \boldsymbol{U}_{k+1}^{-1} \boldsymbol{V}_{k+1} \boldsymbol{U}_k^\top \right) \boldsymbol{U}_k^{-\top} \boldsymbol{U}_j^\top + \boldsymbol{U}_i \boldsymbol{U}_n^{-1} \left(\boldsymbol{U}_n \boldsymbol{V}_n^\top \right) \boldsymbol{U}_n^{-\top} \boldsymbol{U}_j^\top\\
    & = \sum_{k=j}^{n-1} \left( \boldsymbol{U}_i \boldsymbol{V}_k^\top \boldsymbol{U}_k^{-\top} \boldsymbol{U}_j^\top - \boldsymbol{U}_i \boldsymbol{U}_{k+1}^{-1} \boldsymbol{V}_{k+1} \boldsymbol{U}_j^\top \right) + \boldsymbol{U}_i \boldsymbol{V}_n^\top \boldsymbol{U}_n^{-\top} \boldsymbol{U}_j^\top \\
    & = \boldsymbol{U}_i \boldsymbol{V}_j^\top \bm{U}_j^{-\top} \bm{U}_j^\top - \bm{U}_i \bm{U}_n^{-1} \left( \bm{V}_n \bm{U}_n^\top - \bm{U}_n \bm{V}_n^\top \right) \bm{U}_n^{-\top} \bm{U}_j \\
    & = \boldsymbol{U}_i \boldsymbol{V}_j^\top
\end{align*}
where $\boldsymbol{L}^{-1}_k$ is the $k$-th column vector of $\boldsymbol{L}^{-1}$ and the last equation holds as $\bm{U}_n \bm{V}_n^\top$ is symmetric. Since $\boldsymbol{\Xi}$ is symmetric, $\boldsymbol{\Xi}^{-1}_{[ij]} = \left(\boldsymbol{\Xi}^{-1}_{[ji]}\right)^\top = \left(\boldsymbol{U}_j \boldsymbol{V}_i^\top \right)^\top  = \boldsymbol{V}_i \boldsymbol{U}_j^\top $ for $1 \leq j < i \leq n$.

\end{document}